\pdfoutput=1
\documentclass[a4paper,12pt,reqno,oneside]{amsart}

\usepackage[utf8x]{inputenc}
\usepackage[T1]{fontenc}
\usepackage{geometry}
\usepackage{lmodern}
\usepackage[stretch=10]{microtype}
\usepackage{graphicx}

\usepackage{amsmath}
\usepackage{amssymb}
\usepackage{amsthm}
\usepackage[shortlabels]{enumitem}
\PassOptionsToPackage{hyphens}{url}
\usepackage[bookmarksnumbered,bookmarksdepth=2,bookmarksopen,colorlinks,linkcolor=black,citecolor=black,urlcolor=black,linktoc=all]{hyperref}
\usepackage[capitalise]{cleveref}

\newcommand{\bC}{\mathbb{C}}

\newcommand{\bG}{\mathbb{G}}
\newcommand{\bH}{\mathbb{H}}
\newcommand{\bK}{\mathbb{K}}

\newcommand{\bQ}{\mathbb{Q}}
\newcommand{\bR}{\mathbb{R}}
\newcommand{\bS}{\mathbb{S}}
\newcommand{\bZ}{\mathbb{Z}}
\newcommand{\ZZ}{\mathbb{Z}}
\newcommand{\QQ}{\mathbb{Q}}
\newcommand{\RR}{\mathbb{R}}
\newcommand{\CC}{\mathbb{C}}

\newcommand{\cA}{\mathcal{A}}

\newcommand{\cC}{\mathcal{C}}
\newcommand{\cF}{\mathcal{F}}
\newcommand{\cH}{\mathcal{H}}
\newcommand{\cM}{\mathcal{M}}
\newcommand{\cO}{\mathcal{O}}
\newcommand{\cV}{\mathcal{V}}

\newcommand{\fc}{\mathfrak{c}}

\newcommand{\fS}{\mathfrak{S}}

\newcommand{\Ag}{\mathcal{A}_g}
\newcommand{\Fg}{\mathcal{F}_g}
\newcommand{\Hg}{\mathcal{H}_g}

\newcommand{\gG}{\mathbf{G}}
\newcommand{\gGL}{\mathbf{GL}}
\newcommand{\gGSp}{\mathbf{GSp}}
\newcommand{\gPGSp}{\mathbf{PGSp}}

\newcommand{\gH}{\mathbf{H}}

\newcommand{\gP}{\mathbf{P}}

\newcommand{\gS}{\mathbf{S}}

\newcommand{\gSp}{\mathbf{Sp}}

\newcommand{\rH}{\mathrm{H}}
\newcommand{\rM}{\mathrm{M}}

\DeclareMathOperator{\Aut}{Aut}

\DeclareMathOperator{\characteristic}{char}
\DeclareMathOperator{\covol}{covol}

\DeclareMathOperator{\disc}{disc}
\DeclareMathOperator{\End}{End}

\DeclareMathOperator{\GL}{GL}

\DeclareMathOperator{\Hom}{Hom}

\DeclareMathOperator{\Nm}{Nm}
\DeclareMathOperator{\Nrd}{Nrd}

\DeclareMathOperator{\rk}{rk}

\DeclareMathOperator{\Stab}{Stab}

\DeclareMathOperator{\Tr}{Tr}
\DeclareMathOperator{\Trd}{Trd}

\newcommand{\ad}{\mathrm{ad}}

\newcommand{\id}{\mathrm{id}}
\newcommand{\op}{\mathrm{op}}

\newcommand{\extpower}{\bigwedge\nolimits}

\newcommand{\abs}[1]{\lvert #1 \rvert}
\newcommand{\Bigabs}[1]{\Bigl\lvert #1 \Bigr\rvert}

\newcommand{\length}[1]{\lVert #1 \rVert}

\newcommand{\bs}{\backslash}

\newcommand{\ov}{\overline}

\newcommand{\fullmatrix}[4]{\left( \begin{matrix} #1 & #2 \\ #3 & #4 \end{matrix} \right)}
\newcommand{\fullsmallmatrix}[4]{\bigl( \begin{smallmatrix} #1 & #2 \\ #3 & #4 \end{smallmatrix} \bigr)}

\newcommand{\defterm}[1]{\textbf{#1}}

\newtheorem{lemma}{Lemma}[section]
\newtheorem{proposition}[lemma]{Proposition}
\newtheorem{theorem}[lemma]{Theorem}
\newtheorem{corollary}[lemma]{Corollary}
\newtheorem{conjecture}[lemma]{Conjecture}

\Crefname{conjecture}{Conjecture}{Conjectures} 

\Crefname{claim}{Claim}{Claims}

\newtheorem*{lemma*}{Lemma}
\newtheorem*{proposition*}{Proposition}
\newtheorem*{theorem*}{Theorem}
\newtheorem*{corollary*}{Corollary}
\newtheorem*{claim*}{Claim}

\theoremstyle{definition}

\usepackage{ifthen}
\newcounter{constant}
\newcommand{\createC}[1]{\refstepcounter{constant} \label{C:#1}}
\newcommand{\newC}[1]{%
   \ifthenelse{\equal{#1}{*}} {%
      \stepcounter{constant} C_{\theconstant}%
   } {%
      \refstepcounter{constant} C_{\theconstant} \label{C:#1}%
   }%
}
\newcommand{\refC}[1]{C_{\ref*{C:#1}}}

\usepackage{color}

\newcommand{\Qbar}{\overline \bQ}

\newcommand{\fA}{\mathfrak{A}}

\title[Lattices with skew-Hermitian forms and unlikely intersections]{Lattices with skew-Hermitian forms over division algebras and unlikely intersections}

\author{Christopher Daw}
\author{Martin Orr}

\address{Daw: Department of Mathematics and Statistics, University of Reading,
    White\-knights,  PO Box 217,  Reading,  Berkshire RG6 6AH,  United Kingdom}
\email{chris.daw@reading.ac.uk}

\address{Orr: Department of Mathematics, The University of Manchester, Alan Turing Building, Oxford Road, Manchester M13 9PL, United Kingdom}
\email{martin.orr@manchester.ac.uk}


\begin{document}

\begin{abstract}
This paper has two objectives. First, we study lattices with skew-Hermitian forms over division algebras with positive involutions.  For division algebras of Albert types I and~II, we show that such a lattice contains an ``orthogonal'' basis for a sublattice of effectively bounded index.  Second, we apply this result to obtain new results in the field of unlikely intersections. More specifically, we prove the Zilber--Pink conjecture for the intersection of curves with special subvarieties of simple PEL type I and II under a large Galois orbits conjecture. We also prove this Galois orbits conjecture for certain cases of type~II.

\vspace{1em}

\noindent \textsc{Résumé}
(Réseaux munis de formes anti-Hermitiennes sur des algèbres à division et intersections atypiques)

Cet article a deux objectifs.  Nous étudions d'abord les réseaux munis de formes anti-Hermitiennes sur des algèbres à division avec involutions positives.  Pour les algèbres à division de type I et~II dans la classification d'Albert, nous montrons qu'un tel réseau contient une base ``orthogonale'' pour un sous-réseau d'indice effectivement bornée.  Ensuite, nous appliquons ce résultat pour obtenir des nouveaux résultats dans la théorie d'intersections atypiques.  En particulier, nous prouvons la conjecture de Zilber--Pink pour l'intersection des courbes avec les sous-variétés spéciales de type PEL simple I et~II sous une conjecture de grandes orbites de Galois.  De plus, nous prouvons cette conjecture sur les orbites Galoisiennes dans certains cas de type~II.
\end{abstract}

\maketitle

\setcounter{tocdepth}{1}

\tableofcontents

\section{Introduction}

In this paper we develop a quantitative result on reduction theory for lattices over division algebras equipped with skew-Hermitian forms.  Our main theorem is inspired by Minkowski's theorems on lattices and Masser and Wüstholz's class index lemma \cite{MW95}, with the additional ingredient of looking for a basis which interacts nicely with a skew-Hermitian form.

Our purpose in proving this theorem is to apply it to certain cases of the Zilber--Pink conjecture in moduli spaces of abelian varieties.  The theorem on lattices supplies the ``parameter height bound'' needed for the Pila--Zannier strategy.
This generalises our earlier paper \cite{QRTUI}, where we proved some cases of Zilber--Pink for the moduli space of abelian \textit{surfaces} using quantitative reduction theory.

\subsection{Bases and skew-Hermitian forms over division algebras}

A classical result in algebraic number theory, due to Minkowski, asserts that if $R$ is the ring of integers of a number field, then every ideal $I \subset R$ contains an element~$x$ such that the index $[I:Rx]$ is bounded by an explicit multiple of $\sqrt{\disc(R)}$.
A similar result can be proved for torsion-free modules of finite rank over the ring of integers of a number field, by combining Minkowski's theorem with the structure theory of finite-rank modules over a Dedekind domain (see \cite[\S 22, Exercise~6]{CR62}).

In \cite{MW95}, Masser and Wüstholz generalised this theorem to torsion-free $R$-modules $L$ of finite rank over any order $R$ in a division $\bQ$-algebra.
This generalisation shows that there is a free $R$-submodule of finite index in~$L$, with index $[L:R]$ bounded polynomially in terms of $\disc(R)$.
The statement is as follows. (See section \ref{sec:orders} for the definition of the discriminant of an order in a semisimple $\QQ$-algebra.)

\begin{theorem} {\cite[Chapter 2, Class Index Lemma]{MW95}} \label{minkowski-general-index}
Let $D$ be a division $\bQ$-algebra and let $R$ be an order in $D$.
Let $L$ be a torsion-free $R$-module of finite rank $m$.
Then there exists a left $D$-basis $v_1, \dotsc, v_m$ for $D \otimes_R L$ such that $v_1, \dotsc, v_m$ are in $L$ and $[L:Rv_1 + \dotsb + Rv_m] \leq \abs{\disc(R)}^{m/2}$.
\end{theorem}

In another direction, if $L$ is a $\bZ$-module of finite rank equipped with a positive definite symmetric bilinear form $\psi \colon L \times L \to \bZ$, then one can use the classical reduction theory of quadratic forms to find an \textit{orthogonal} basis $v_1, \dotsc, v_m$ for $L \otimes_\bZ \bQ$ such that $v_1, \dotsc, v_n \in L$ and $[L:\bZ v_1 + \dotsb + \bZ v_m]$ is bounded by a polynomial in $\abs{\disc(L)}$.
A similar result for a $\bZ$-module of finite rank equipped with a symplectic form can be found in \cite{Orr15} (see Lemma~4.3 therein).

In this paper, we obtain a version of \cref{minkowski-general-index} in which $L$ is equipped with a $(D,\dag)$-skew-Hermitian form (see section \ref{subsec:skew-hermitian-forms} for the definition of a $(D,\dag)$-skew-Hermitian form.)
We seek a basis of $D \otimes_R L$ which is weakly symplectic or weakly unitary with respect to this form.  Weakly symplectic or weakly unitary bases are the analogues for $(D,\dag)$-skew-Hermitian forms of bases which are orthogonal but not necessarily orthonormal: we say that a $D$-basis $v_1, \dotsc, v_m$ is \defterm{weakly symplectic} if $\psi(v_i, v_j) = 0$ for all $i, j$ except when $\{i,j\} = \{2k-1,2k\}$ for some $k \in \bZ$, and that the basis is \defterm{weakly unitary} if $\psi(v_i, v_j) = 0$ for all $i, j \in \{ 1, \dotsc, m \}$ such that $i \neq j$.

\begin{theorem} \label{minkowski-hermitian-perfect}
Let $D$ be either a totally real number field or a totally indefinite quaternion algebra over a totally real number field.
Let $\dag$ be a positive involution of~$D$.
Let $V$ be a left $D$-vector space of dimension~$m$, equipped with a non-degenerate $(D,\dag)$-skew-Hermitian form $\psi \colon V \times V \to D$.
Let $L$ be a $\bZ$-lattice of full rank in~$V$ such that $\Trd_{D/\bQ} \psi(L \times L) \subset \bZ$.
Let $R=\Stab_D(L)$ denote the stabiliser of $L$ in~$D$.

Then there exists a $D$-basis $v_1, \dotsc, v_m$ for $V$ such that:
\begin{enumerate}[(i)]
\item $v_1, \dotsc, v_m \in L$;
\item the basis is weakly symplectic (when $D$ is a field) or weakly unitary (when $D$ is a quaternion algebra) with respect to~$\psi$;
\item $[L:Rv_1 + \dotsb + Rv_m] \leq \newC{minkowski-main-first} \abs{\disc(R)}^{\newC*} \abs{\disc(L)}^{\newC*}$;
\item $\abs{\psi(v_i, v_j)}_D \leq \newC* \abs{\disc(R)}^{\newC*} \abs{\disc(L)}^{\newC{minkowski-main-last}}$ for $1 \leq i, j \leq m$.
\end{enumerate}
The constants $\refC{minkowski-main-first}, \dotsc, \refC{minkowski-main-last}$ depend only on $m$ and $\dim_\bQ(D)$.
\end{theorem}

Explicit, but not optimal, values for the constants are given in \cref{weakly-unitary-induction}.
One could also prove a version of this theorem that bounds the lengths of the vectors $v_i$, in the style of Minkowski's second theorem, but this is stronger than needed for our application, and according to the proof that we know, the constants are exponential instead of polynomial in~$m$.

Division $\bQ$-algebras with positive involution were classified by Albert into four types (see section~\ref{subsec:albert}).
The division algebras treated in \cref{minkowski-hermitian-perfect} are those of types I and~II in Albert's classification.
It is likely that this paper's strategy could be adapted to prove \cref{minkowski-hermitian-perfect} for division $\bQ$-algebras with positive involution of types III and~IV, as well as a version for Hermitian forms instead of skew-Hermitian forms, although various steps in the argument would require modification.

\subsection{Applications to the Zilber--Pink conjecture} \label{subsec:intro-zp-theorem}

We apply \cref{minkowski-hermitian-perfect} to prove certain cases of the Zilber--Pink conjecture on unlikely intersections in the moduli space $\cA_g$ of principally polarised abelian varieties of dimension $g$ (which is an example of a Shimura variety), as follows.

\begin{theorem} \label{main-theorem-zp}
Let $g \geq 3$.
Let $\Sigma$ denote the set of points $s \in \Ag(\bC)$ for which the endomorphism algebra of the associated abelian variety $A_s$ is either a totally real field, other than~$\bQ$, or a non-split totally indefinite quaternion algebra over a totally real field.
Let $C$ be an irreducible Hodge generic algebraic curve in $\Ag$.

If $C$ satisfies \cref{galois-orbits}, then $C \cap \Sigma$ is finite.
\end{theorem}

 Throughout this paper, whenever we refer to endomorphisms of an abelian variety, we refer to its endomorphisms over an algebraically closed field.

The analogous statement to \cref{main-theorem-zp} for $g=2$ was proved in our earlier work \cite{QRTUI}.
In that paper, \cite[Lemma~5.7]{QRTUI} played the role which is now played by \cref{minkowski-hermitian-perfect}.
Indeed this paper represents the next stage of our programme on the Zilber--Pink conjecture for Shimura varieties, following on from \cite{ExCM} and \cite{QRTUI}, which were inspired by the earlier papers \cite{HP12}, \cite{HP16}, \cite{DR18}, \cite{OrrUI}.

\Cref{galois-orbits}, referred to in \cref{main-theorem-zp}, is a large Galois orbits conjecture, of the sort appearing in many works on unlikely intersections (for example, \cite[Conjecture~2.7]{Ull14}, \cite[Conjecture~8.2]{HP16}, \cite[Conjecture~11.1]{DR18}, \cite[Conjecture~6.2]{QRTUI}).

\begin{conjecture} \label{galois-orbits}
Define $\Sigma \subset \Ag$ as in \cref{main-theorem-zp} and let $C\subset\cA_g$ denote an irreducible Hodge generic algebraic curve defined over a finitely generated field $L \subset \bC$.
Then there exist positive constants $\newC{galois-orbits-mult}$ and $\newC{galois-orbits-exp}$, depending only on $g$, $L$ and~$C$, such that, for any point $s \in C \cap \Sigma$,
\[ \# \Aut(\bC/L) \cdot s \geq \refC{galois-orbits-mult} \abs{\disc(\End(A_s))}^{\refC{galois-orbits-exp}}. \]
\end{conjecture}

The most general conjecture of this type in the context of Shimura varieties which has been written down is \cite[Conjecture~11.1]{DR18}.
It is not clear whether \cite[Conjecture~11.1]{DR18} implies \cref{galois-orbits}, because it is not clear how $\abs{\disc(\End(A_s))}$ is related to the complexity $\Delta(\langle s \rangle)$ defined in \cite{DR18}.
For example, in \cite[Conjecture~11.1]{DR18}, $\Delta(\langle s \rangle)$ is the complexity of the smallest special subvariety of~$\Ag$ containing~$s$.
In \cref{galois-orbits}, $\abs{\disc(\End(A_s))}$ is a measure of the complexity of the smallest special subvariety \textit{of PEL type} containing~$s$, which might not be the same as the smallest special subvariety containing~$s$.
However, for the purpose of proving cases of the Zilber--Pink conjecture, the precise definition of complexity is not important: we only need a parameter height bound and a Galois orbits bound which involve the same notion of complexity.  Since we are focussing on special subvarieties of PEL type, the discriminant of the endomorphism ring is a natural measure of complexity. 

Using André's G-functions method \cite{And89}, in the form of \cite[Theorem~8.2]{ExCM}, we prove \cref{galois-orbits} in certain cases and thus establish \cref{main-theorem-zp} unconditionally in those settings.
The proof of large Galois orbits in \cref{unconditional} does not involve new ideas beyond those in \cite{ExCM}, \cite{QRTUI}: the new contribution of this paper is in the parameter height bound.

\begin{theorem}\label{unconditional}
Let $g$ be an even positive integer.
Let $\Sigma^*$ denote the set of points $s \in \Ag$ for which $\End(A_s) \otimes_\bZ \bQ$ is a non-split totally indefinite quaternion algebra whose centre is a totally real field of degree $e$ such that $4e$ does not divide $g$.

Let $C\subset\cA_g$ denote an irreducible Hodge generic algebraic curve defined over a number field.  Suppose that the Zariski closure of $C$ in the Baily--Borel compactification of $\Ag$ intersects the 0-dimensional stratum.

Then $C$ satisfies \cref{galois-orbits} for $\Sigma^*$ (in the place of $\Sigma$). Hence, $C \cap \Sigma^*$ is finite.
\end{theorem}

Compared with \cref{galois-orbits}, \cref{unconditional} adds two restrictions: $\Sigma^*$ is defined by a smaller class of endomorphism algebras than $\Sigma$, and there is a condition on the intersection of the Zariski closure of $C$ with the boundary of the Baily--Borel compactification.
We recall that the Baily--Borel compactification of the moduli space~$\cA_g$ is naturally stratified as a disjoint union
\[\cA_g\sqcup\cA_{g-1}\sqcup\cdots\sqcup\cA_1\sqcup\cA_0\]
of locally closed subvarieties.
The zero-dimensional stratum is $\cA_0$, which is a point.
The condition that $C$ intersects the zero-dimensional stratum is equivalent to saying that the associated family of principally polarised abelian varieties degenerates to a torus (this informal statement can be made precise as in \cite[Theorem~1.4]{ExCM}).

\subsection{The Zilber--Pink conjecture and special subvarieties of PEL type} \label{subsec:intro-zp-context}

Let us recall a general statement of the Zilber--Pink conjecture for Shimura varieties.  A \defterm{special subvariety} of a Shimura variety $S$ means an irreducible component of a Shimura subvariety of~$S$.
An irreducible subvariety of $S$ is \defterm{Hodge generic} if it is not contained in any special subvariety other than a component of $S$ itself.

\begin{conjecture} \label{zilber-pink} \cite[Conjecture~1.3]{pink:generalisation}
Let $S$ be a Shimura variety and let $V$ be an irreducible Hodge generic subvariety of $S$.
Then the intersection of $V$ with the special subvarieties of $S$ having codimension greater than $\dim V$ is not Zariski dense in $V$.
\end{conjecture}

In order to relate this to \cref{main-theorem-zp}, we introduce a class of special subvarieties of~$\Ag$ which come from endomorphisms of abelian varieties.
We recall that $\cA_g$ is an irreducible algebraic variety over~$\bQ$.
For any algebraically closed field $k$ containing $\bQ$ and any point $s \in \Ag(k)$, we write $A_s$ for the principally polarised abelian variety over $k$ (defined up to isomorphism) corresponding to the point $s$.

For any ring $R$, the set
\[ \cM_R = \{ s \in \Ag(\bC) : \text{there exists an injective homomorphism } R \to \End(A_s) \} \]
is a countable union of algebraic subvarieties of $\Ag$.
Each irreducible component of $\cM_R$ is a special subvariety of $\Ag$.
We call a subvariety of $\Ag$ a \defterm{special subvariety of PEL type} if it is an irreducible component of $\cM_R$ for some $R$.

If $R \not\cong \bZ$, then $\cM_R$ is strictly contained in $\Ag$.
Hence the set $\Sigma$ defined in \cref{main-theorem-zp} is contained in the union of the proper special subvarieties of PEL type of $\Ag$.
Furthermore, as we prove in \cref{codim-pel}, for $g \geq 3$, all proper special subvarieties of PEL type of $\Ag$ have codimension at least~$2$.
Thus, \cref{zilber-pink} predicts that the intersection $C \cap \Sigma$ of \cref{main-theorem-zp} should not be Zariski dense in the curve~$C$, that is, it should be finite.

For each special subvariety of PEL type $S \subset \Ag$, there is a largest ring $R$ such that $S$ is a component of $\cM_R$.
We call this ring $R$ the \defterm{generic endomorphism ring} of $S$, and we call $R \otimes_\bZ \bQ$ the \defterm{generic endomorphism algebra} of $S$.
We say that a point $s\in S(\CC)$ is \defterm{endomorphism generic} if the endomorphism ring of $A_s$ is equal to $R$.  Note that all points in the complement of countably many proper subvarieties of $S$ are endomorphism generic.

We call $S \subset \Ag$ a \defterm{special subvariety of simple PEL type} if it is a special subvariety of PEL type and its generic endomorphism algebra is a division algebra.
(Equivalently, $A_s$ is a simple abelian variety for endomorphism generic points $s \in S(\bC)$.)
We call $S$ a \defterm{special subvariety of simple PEL type I or~II} if it is a special subvariety of PEL type whose generic endomorphism ring is a division algebra of type I or~II in the Albert classification (see section~\ref{subsec:albert}).
Thus the set $\Sigma$ in \cref{main-theorem-zp} is the union of the endomorphism generic loci of all special subvarieties of simple PEL type I or~II, excluding $\Ag$ itself.

In section \ref{sec:dims} we establish the following bounds on the dimensions of special subvarieties of PEL type in $\Ag$.  These are not necessary for proving \cref{main-theorem-zp}, but they are interesting for understanding the Zilber--Pink conjecture in the context of special subvarieties of PEL type.
In particular, when $g \geq 3$, \cref{codim-pel} guarantees that intersections between a Hodge generic curve and all proper special subvarieties of PEL type in $\Ag$ are predicted to be ``unlikely'' by the Zilber--Pink conjecture.

\begin{proposition}\label{codim-pel-intro}
Let $S \subset \Ag$ be a special subvariety, not equal to~$\Ag$.
\begin{enumerate}[(i)]
\item If $S$ is of simple PEL type, then $\dim(S) \leq \dim(\Ag) - g^2/4$.
\item If $S$ is of PEL type, then $\dim(S) \leq \dim(\Ag) - g + 1$.
\end{enumerate}
\end{proposition}

We also prove a finiteness result for special subvarieties of simple PEL type I or~II of bounded complexity (\cref{prop:finitely-many}).
This is the analogue of a special case of \cite[Conjecture~10.3]{DR18}, using our notion of complexity (cf.\ discussion of complexity of special subvarieties below \cref{galois-orbits}).

\begin{proposition}\label{prop:finitely-many-intro}
Define $\Sigma \subset \Ag$ as in \cref{main-theorem-zp}.
For each $b\in\RR$, the points $s\in\Sigma$ such that $\abs{\disc(\End(A_s))} \leq b$ belong to only finitely many proper special subvarieties of simple PEL type I or II.
\end{proposition}

\subsection{High-level proof strategy for Theorem~\ref{main-theorem-zp}} \label{subsec:proof-strategy-high-level}

We now outline the strategy of the proof of \cref{main-theorem-zp}, which is carried out in sections \ref{sec:ZP-high-level} to~\ref{cases-of-ZP}, making use of \cref{minkowski-hermitian-perfect}. For our notation and terminology around Shimura datum components, see \cite[sec.\ 2.A and~2.B]{ExCM}.

Let $\gG = \gGSp_{2g}$ and let $(\gG, X^+)$ denote the Shimura datum component defined in section \ref{subsec:shimura-data}, which gives rise to the Shimura variety~$\Ag$. By \cref{conj-class-mt}, the Shimura subdatum components $(\gH, X_\gH^+)$ associated with special subvarieties of simple PEL type I or~II lie in only finitely many $\gG(\bR)$-conjugacy classes.  Hence it suffices to prove \cref{main-theorem-zp} ``one $\gG(\bR)$-conjugacy class at a time.''
Thanks to \cref{conj-class-mt}, this means that we choose positive integers $d,e,m$ and let $\gH_0$ be the subgroup of $\gG = \gGSp_{2g}$ defined in \eqref{eqn:H0} for these $d,e,m$.
We prove \cref{main-theorem-zp} with $\Sigma$ replaced by $\Sigma_{d,e,m}$, namely, the union of the endomorphism generic loci of all proper special subvarieties of $\Ag$ of simple PEL type I or~II whose underlying group is $\gG(\bR)$-conjugate to $\gH_0$.

Let $\pi$ denote the standard quotient map $X^+ \to \Ag(\bC)$ and let $\Fg$ denote a Siegel fundamental set of $X^+$, as defined in \cite[sec. 2]{Orr18} and \cite[sec.~2.G]{ExCM}.

In order to prove \cref{main-theorem-zp} for $\Sigma_{d,e,m}$, we follow the same proof strategy as \cite{QRTUI} (which proves the analogous result for $g=2$, $d=2$, $e=m=1$).
The idea is to apply the Habegger--Pila--Wilkie counting theorem \cite[Corollary 7.2]{HP16} to a definable set of the form
\[ D = \{ (y,z) \in Y \times \cC : z \in X_y \} \]
where $Y \subset \bR^n$ is a semialgebraic parameter space for subsets $X_y \subset X^+$ and $\cC = \pi^{-1}(C(\bC)) \cap \Fg$.
The parameter space $Y$ has the following properties:
\begin{enumerate}[itemsep=2pt]
\item For every rational point $y \in Y \cap \bQ^n$, $X_y$ is a pre-special subvariety of $X^+$ whose underlying group is $\gG(\bR)$-conjugate to $\gH_0$.
\item For every point $s \in \Sigma_{d,e,m}$, there exists $z \in \pi^{-1}(s) \cap \Fg$ such that $z$ lies in $X_y$ for some rational point $y \in Y \cap \bQ^n$, with the height $H(y)$ polynomially bounded in terms of $\End(A_s)$.
\end{enumerate}

Consequently, if $C \cap \Sigma_{d,e,m}$ is infinite, and if the large Galois orbits conjecture holds, then the number of points $(y,z) \in D$ with $y \in Y \cap \bQ^n$ grows reasonably quickly with respect to $H(y)$.
Then the Habegger--Pila--Wilkie theorem tells us that $D$ contains a path whose projection to $Y$ is semialgebraic and whose projection to~$\cC$ is non-constant.
We can conclude by a functional transcendence argument as in \cite[sec.~6.5]{QRTUI}.

\subsection{Proof strategy: parameter space} \label{subsec:parameter-space}

The strategy described in section~\ref{subsec:proof-strategy-high-level} is the same as that applied in \cite{HP16}, \cite{DR18}, \cite{ExCM}, \cite{QRTUI}, and others.  The new ingredient required to apply the strategy described in section~\ref{subsec:proof-strategy-high-level} in our case is to construct a suitable parameter space $Y$ for special subvarieties of simple PEL type I or~II and prove that it satisfies property~(2) above.

To construct $Y$, we will choose a suitable representation $\rho \colon \gG \to \gGL(W)$, where $W$ is a $\bQ$-vector space, and a vector $w_0 \in W$ such that $\Stab_{\gG}(w_0) = \gH_0$.
Then we define $Y$ to be the ``expanded $\rho$-orbit'' of $w_0$ in $W_\bR$:
\[ Y = \Aut_{\rho(\gG)}(W_\bR) \, \rho(\gG(\bR)) \, w_0. \]
For each $y \in Y$, we define $\gH_y = \Stab_{\gG_\bR}(y)$ and
\[ X_y = \{ z \in X^+ : z(\bS) \subset \gH_y \}. \]

If $y \in Y \cap \bQ^n$, then $\gH_y$ is a $\bQ$-algebraic subgroup of $\gG$, which is $\gG(\bR)^+$-conjugate to $\gH_0$.
By \cref{conj-class-mt}, we have $\gH_{y,\bR} = g\gH_{0,\bR}g^{-1}$ for some $g \in \gG(\bR)^+$ and, for each component $X_y^+$ of $X_y$, $(\gH_0, g^{-1}X_y^+)$ is a Shimura subdatum component.
By \cref{unique-datum}, there is only one Shimura subdatum component with group $\gH_0$.
We denote this component by~$X^+_0$. Therefore, $g^{-1}X_y^+=X^+_0$ for every component $X_y^+$ of $X_y$.
Hence, $X_y$ is connected and $(\gH_y, X_y)$ is a Shimura subdatum component of $(\gG, X^+)$.

This  establishes property~(1) of section~\ref{subsec:proof-strategy-high-level}. To establish property~(2) of section~\ref{subsec:proof-strategy-high-level}, we use the method of \cite[Proposition~6.3]{QRTUI}.
All we have to do is understand how fundamental sets in $\gH_y$ vary through the $\gG(\bR)$-conjugacy class.
A quantitative description of these fundamental sets is given by \cite[Theorem~1.2]{QRTUI}, but it requires as input a suitable representation $\rho$ and bounds on the lengths of certain vectors in~$\rho$.
This input is given in \cref{reps-closed-orbits,rep-bound-arithmetic}, which together generalise \cite[Proposition~5.1]{QRTUI} (which is the case $d=2$, $e=m=1$).
The representation is constructed in \cref{reps-closed-orbits}, and the construction of vectors $w_u$ with bounds for their lengths is found in \cref{rep-bound-arithmetic}.

To explain how \cref{minkowski-hermitian-perfect} is used, we outline the proof of \cref{rep-bound-arithmetic}. Let $L = \ZZ^{2g}$ and let $V = L_\QQ = \QQ^{2g}$, with the standard action of $\gG=\gGSp_{2g}$ on $V$.  Choosing a lift $\tilde s \in \pi^{-1}(s)$ induces an isomorphism $L \cong H_1(A_s, \ZZ)$, hence an action of $\End(A_s)$ on $L$.  The polarisation induces a $(D,\dag)$-skew-Hermitian form $\psi$ on $V$, where $D = \End(A_s) \otimes \QQ$ and $\dag$ is the Rosati involution.
We use \cref{minkowski-hermitian-perfect} to choose a weakly unitary or weakly symplectic $D$-basis $\{v_i\}$ for $(V, \psi)$ contained in~$L$.
Suitable multiples of the $\{v_i\}$ yield a symplectic or $\alpha$-unitary $D_\RR$-basis for $(V_\RR, \psi)$ (see section~\ref{subsec:unitary-bases} for definitions).
The choice of a symplectic or $\alpha$-unitary $D_\RR$-basis for $(V_\RR, \psi)$  is equivalent to the choice of an element $u' \in \gSp_{2g}(\RR)$ such that $\tilde s \in u'X_0^+$.
This element $u'$ is called $\theta^{-1}=uh$ in section~\ref{sec:rep-bound}, and its construction is detailed in \cref{theta-def,h-in-h0}.

We then use the bound from \cref{minkowski-hermitian-perfect}(iv), via \cref{theta-def}(iii), together with the fact that $v_i \in L$, to obtain $\gamma \in \gSp_{2g}(\ZZ)$ such that the entries of the matrices $\gamma g$ and $(\gamma g)^{-1}$ are polynomially bounded (Lemmas \ref{z-basis-bound} to~\ref{entries-bound}).
Since $\pi$ is $\gSp_{2g}(\ZZ)$-invariant, we still have $\pi^{-1}(s) \cap \gamma gX_0^+ \neq \emptyset$.
From $\gamma g$, we construct a vector (denoted $w_u$ in section \ref{sec:rep-bound}) suitable for use as input to \cite[Theorem~1.2]{QRTUI}, which gives the height bound for $y = \rho(b^{-1}u)w_u$.

\subsection{Remark on effectivity}

We note that \cref{minkowski-hermitian-perfect} and \cref{EQ-scheme} are effective. As such, the obstructions to effectivity in \cref{unconditional} are (1) its dependence on the (ineffective) Habegger--Pila--Wilkie theorem (as stated in \cite[Theorem 9.1]{DR18}) from o-minimality and (2) the ineffectivity in \cite{QRTUI}, as explained in Remark 4.3 therein. Obstruction (1) was recently overcome for the Andr\'e--Oort conjecture for non-compact curves in Hilbert modular varieties by Binyamini and Masser \cite{BM:AO} using so-called $Q$-functions. It seems plausible that these techniques could also apply to our setting.

\subsection{Outline of the paper}

The paper is in two parts.
The first part, sections \ref{sec:division-algebras} to~\ref{sec:minkowski-proof}, proves \cref{minkowski-hermitian-perfect}.  It deals only with modules over division algebras and skew-Hermitian forms, with no mention of Shimura varieties.
The second part, sections \ref{sec:ZP-high-level} to~\ref{cases-of-ZP}, proves \cref{main-theorem-zp}. The main new ingredient is \cref{minkowski-hermitian-perfect}.

In section~\ref{sec:division-algebras}, we introduce terminology around division algebras and their orders, as well as various lemmas used throughout the calculations in sections \ref{sec:skew-hermitian} and~\ref{sec:minkowski-proof}.
In section~\ref{sec:skew-hermitian}, we define the notion of a skew-Hermitian form on a module over a division algebra with involution and define several notions of well-behaved bases with respect to a skew-Hermitian form.
Section~\ref{sec:minkowski-proof} consists of the proof of \cref{minkowski-hermitian-perfect}, which involves substantial calculations.

Section~\ref{sec:ZP-high-level} introduces Shimura data and establishes the basic properties of special subvarieties of simple PEL type I and~II in $\Ag$.
The representation and vectors required as input for \cite[Theorem~1.2]{QRTUI} are constructed in sections \ref{sec:representation} and~\ref{sec:rep-bound}, as sketched in section~\ref{subsec:parameter-space}.  The application of \cref{minkowski-hermitian-perfect} is found in section~\ref{sec:rep-bound}, specifically \cref{theta-def}.  Finally section~\ref{cases-of-ZP} states some slightly stronger versions of \cref{main-theorem-zp,unconditional} and completes their proofs.

\subsection{Notation} \label{subsec:notation}

We shall use the following notation for matrices.
If $A$ and $B$ are square matrices, we will denote by $A\oplus B$ the block diagonal matrix with blocks $A$ (top-left) and $B$ (bottom-right). We will write $A^{\oplus d}$ to denote the block diagonal matrix $A\oplus\cdots\oplus A$ with $A$ appearing $d$ times.

We shall write $J_2 = \fullsmallmatrix{0}{1}{-1}{0}$ and $J_n = J_2^{\oplus n/2}$ for each even positive integer~$n$.

\subsection*{Acknowledgements}  This work was supported by the Engineering and Physical Sciences
Research Council [EP/S029613/1 to C.D., EP/T010134/1 to M.O.].

The authors are grateful to the referee for their careful reading of the paper and helpful suggestions and corrections.
They are also grateful to Bijay Bhatta, who found several errors in a preprint version of the paper.

\section{Division algebras}
\label{sec:division-algebras}

In this section, we introduce the notation and terminology we shall use for division algebras.  A key definition is a norm $\abs{\cdot}_D$ on an $\bR$-algebra with positive involution.  We establish useful properties of this norm and of the discriminants of orders in division algebras.
We also include some broader preliminary lemmas, on discriminants of bilinear forms and versions of Minkowski's second theorem.

In this paper, our main interest will be in division $\bQ$-algebras with positive involution of Albert types I and~II.
However, we have stated many of the definitions and results in this section in greater generality, such as for semisimple algebras over any subfield of $\bR$.
We do this not only because this greater generality is often natural, but also it is sometimes necessary as we wish to apply the results to $D \otimes_\bQ \bR$ where $D$ is a division $\bQ$-algebra, but $D \otimes_\bQ \bR$ might not be a division algebra.
We have not stated all results at their greatest possible generality, if doing so would require additional complications while not being required for our application.

Throughout this section, $k$ denotes a subfield of $\bR$.
Later in the paper, we will usually use $k = \bQ$ or $k = \bR$.
Whenever we say \defterm{$k$-algebra}, we mean a $k$-algebra of finite dimension.
If $V$ is a $k$-vector space or $k$-algebra, then $V_\bR$ denotes $V \otimes_k \bR$.

\subsection{Semisimple algebras, norms and traces} \label{subsec:semisimple-algebras}

As a reference on semisimple algebras, reduced norm and trace, see \cite[sec.~9]{Rei75}.

Let $D$ be a semisimple $k$-algebra.
Then $D \cong \prod_{i=1}^s D_i$ for some simple $k$-algebras $D_1, \dotsc, D_s$.
For each $i$, let $F_i$ be the centre of $D_i$, which is a field.

We write $\Trd_{D_i/F_i}$ and $\Nrd_{D_i/F_i}$ for the reduced trace and reduced norm respectively of the central simple algebra $D_i/F_i$.
Letting $\Tr_{F_i/k}$ and $\Nm_{F_i/k}$ denote the trace and norm of finite extensions of fields, we define 
\[ \Trd_{D/k} = \sum_{i=1}^s \Tr_{F_i/k} \circ \Trd_{D_i/F_i}, \quad \Nrd_{D/k} = \prod_{i=1}^s \Nm_{F_i/k} \circ \Nrd_{D_i/F_i}. \]

Note that $\Trd_{D/k}$ and $\Nrd_{D/k}$ are compatible with extension of scalars.
By this, we mean that, if $K$ is a field containing $k$ and $D_K = D \otimes_k K$, then $\Trd_{D/k} = \Trd_{D_K/K}|_D$ and similarly for $\Nrd_{D/k}$.
This is true even though the simple factors of $D$ might not remain simple after extension of scalars.

Note also that $\Trd_{D/k}(ab) = \Trd_{D/k}(ba)$ for all $a, b \in D$.

Suppose that $D$ is a simple $k$-algebra and let $F$ be the centre of $D$.
Let
\[ d = \sqrt{\dim_F(D)} = \Trd_{D/F}(1),  \qquad  e = [F:k]. \]
Then $\dim_k(D) = d^2e$.
We will use the notation $F$, $d$, $e$ from this paragraph throughout the paper whenever we talk about simple algebras, without further comment.
Note that $\Tr_{D/F}(a) = d\Trd_{D/F}(a)$ and $\Nm_{D/F}(a) = \Nrd_{D/F}(a)^d$ for all $a \in D$, where $\Tr_{D/F}$ and $\Nm_{D/F}$ are the non-reduced trace and norm.

\subsection{Division algebras with positive involution} \label{subsec:albert}

Let $D$ be a semisimple $k$-algebra.
An \defterm{involution} $\dag$ of $D$ means a $k$-linear map $D \to D$ such that $\dag \circ \dag = \id_D$ and $(ab)^\dag = b^\dag a^\dag$ for all $a,b \in D$.
(We follow the convention of \cite[sec.~8]{Mil05} by requiring involutions of $k$-algebras to be $k$-linear.  This is important for \cref{tr-skew-hermitian-form}.  Thus, with our definition, an ``involution of the second kind'' of a central simple $F$-algebra is not an $F$-algebra involution.  However, an involution of the second kind can still be handled within our framework by taking $k$ to be the fixed subfield of~$F$.)

For every $a \in D$, we have $\Trd_{D/k}(a^\dag) = \Trd_{D/k}(a)$.
Consequently the bilinear form $D \times D \to k$ given by $(a,b) \mapsto \Trd_{D/k}(ab^\dag)$ is symmetric.  The involution $\dag$ is said to be \defterm{positive} if this bilinear form is positive definite (equivalently, if the non-reduced trace bilinear form $(a,b) \mapsto \Tr_{D/k}(ab^\dag)$ is positive definite).

Division $\bQ$-algebras with positive involution $(D,\dag)$ were classified by Albert into four types, depending on the isomorphism type of $D_\bR$ \cite[sec.~21, Theorem~2]{Mum74}.

\begin{enumerate}[label=\textbf{Type \Roman*.},align=left,widest*=3,leftmargin=*,itemsep=2pt]
\item $D = F$, a totally real number field.  The involution is trivial. (In this case $D_\bR \cong \bR^e$.)
\item $D$ is a non-split totally indefinite quaternion algebra over a totally real number field $F$.  (Totally indefinite means that $D_\bR \cong \rM_2(\bR)^e$.)
The involution is of orthogonal type, meaning that after extending scalars to $\bR$ it becomes matrix transpose on each copy of $\rM_2(\bR)$.
\item $D$ is a totally definite quaternion algebra over a totally real number field~$F$.  (Totally definite means that $D_\bR \cong \bH^e$ where $\bH$ is Hamilton's quaternions.)
The involution is the ``canonical involution'' $a \mapsto \Trd_{D/F}(a) - a$.
\item $D$ is a division algebra whose centre is a CM field~$F$.
The involution restricts to complex conjugation on $F$. (In this case $D_\bR \cong \rM_d(\bC)^e$.)
\end{enumerate}

\subsection{The norm \texorpdfstring{$\abs{\cdot}_D$}{||D}} \label{subsec:norm-D}

Let $(D,\dag)$ be a semisimple $k$-algebra with a positive involution.
We define a norm $\abs{\cdot}_D$ on $D_\bR$ by:
\[ \abs{a}_D = \sqrt{\Trd_{D_\bR/\bR}(aa^\dag)}. \]
This is a norm in the sense of a real vector space norm (that is, a length function).
Note that $\abs{a^\dag}_D = \abs{a}_D$ for all $a \in D_\bR$.

The norm $\abs{\cdot}_D$ is induced by the inner product $(a,b) \mapsto \Trd_{D_\bR/\bR}(ab^\dag)$ on $D_\bR$.
This inner product (together with an orientation of $D_\bR$) also induces a volume form.  Whenever we refer to the covolume of a lattice in $D_\bR$, we use this volume form.
(Note that the covolume is the absolute value of the integral of the volume form over a fundamental domain, so it is independent of the choice of orientation.)

If $D$ is a semisimple $k$-algebra, then
\[ D_\bR \cong \prod_{i=1}^r \rM_{s_i}(\bK_i) \]
where $\bK_i = \bR$, $\bC$ or~$\bH$.
If $D$ is equipped with a positive involution~$\dag$, then we can choose the isomorphism so that $\dag$ corresponds to conjugate-transpose on each simple factor \cite[Prop.~8.4.7]{Voi21}.
Throughout the paper, whenever we choose an isomorphism between $D_\bR$ and a product of matrix algebras, we implicitly assume that it has this property.

Let $\abs{\cdot}_F$ denote the \defterm{Frobenius norm} on any matrix algebra over $\bR$, $\bC$ or $\bH$:
\[ \abs{M}_F^2 = \sum_{j,k=1}^s M_{jk} \ov{M_{jk}}. \]
Then, for any $a = (a_1, \dotsc, a_r) \in \prod_i \rM_{s_i}(\bK_i)$, we have
\[ \abs{a}_D^2 = \sum_{i=1}^r \abs{a_i}_F^2. \]

The following lemma will be used repeatedly throughout sections \ref{sec:skew-hermitian} and~\ref{sec:minkowski-proof}.  It is well-known in the case $D_\RR = \rM_n(\RR)$ or $\rM_n(\CC)$ -- see, for example, \cite[p.~291]{HJ85}.

\begin{lemma} \label{length-submult}
Let $(D, \dag)$ be a semisimple $k$-algebra with positive involution.
Then $\abs{ab}_D \leq \abs{a}_D \abs{b}_D$ for all $a, b \in D_\bR$.
\end{lemma}

\begin{proof}
Identify $D_\bR$ with $\prod_{i=1}^r \rM_{s_i}(\bK_i)$ and write
\[ a = (a_1, \dotsc, a_r), \; b = (b_1, \dotsc, b_r) \in \prod_{i=1}^r \rM_{s_i}(\bK_i). \]
Then
\[ \abs{ab}_D^2
   = \sum_{i=1}^r \length{a_ib_i}_F^2
   \leq \sum_{i=1}^r \length{a_i}_F^2 \length{b_i}_F^2 
   \leq \Bigl( \sum_{i=1}^r \length{a_i}_F \Bigr)^2 \Bigl( \sum_{i=1}^r \length{b_i}_F \Bigr)^2
   = \abs{a}_D^2 \abs{b}_D^2. \]
This calculation uses the submultiplicativity of the Frobenius norm and the following inequality, valid for all non-negative real numbers $x_1, \dotsc, x_r, y_1, \dotsc, y_r$:
\begin{equation} \label{eqn:sum-squares}
\sum_{i=1}^r x_iy_i \leq \Bigl( \sum_{i=1}^r x_i \Bigr) \Bigl( \sum_{i=1}^r y_i \Bigr).
\end{equation}

Since the Frobenius norm is less well-known over $\bH$, we remark that, just as in the real and complex cases, submultiplicativity of the Frobenius norm follows from the Cauchy--Schwarz inequality
\[ \Bigl( \sum_{j=1}^s x_j \ov{y}_j \Bigr) \Bigl( \sum_{j=1}^s y_j \ov{x}_j \Bigr)  \leq  \Bigl( \sum_{j=1}^s x_j \ov{x}_j \Bigr) \Bigl( \sum_{j=1}^s y_j \ov{y}_j \Bigr) \text{ for all } x, y \in \bK^n. \]
The Cauchy--Schwarz inequality over~$\bH$ can be proved by considering the discriminant of the quadratic polynomial $(\sum_{i=1}^s x_j t + y_j)(\sum_{i=1}^n \ov{x}_j t + \ov{y}_j)$, which is non-negative for all $t \in \bR$, and then applying the arithmetic mean--geometric mean inequality to the left hand side.
\end{proof}

We say that a semisimple $k$-algebra $D$ is \defterm{$\bR$-split} if $D_\bR \cong \rM_d(\bR)^e$ for some positive integers $d$ and~$e$.
Note that a division $\bQ$-algebra with positive involution is $\bR$-split if and only if it has type I or~II in the Albert classification, and these are the types of algebras that we focus on in this paper.

\begin{lemma} \label{Nrd-length-bound}
Let $(D, \dag)$ be an $\bR$-split semisimple $k$-algebra with positive involution and let $F$ be its centre.
Then, for all $a \in D_\bR^\times$:
\begin{enumerate}[(i)]
\item $\abs{\Nrd_{D_\bR/F_\bR}(a)}_D \leq d^{(1-d)/2} \abs{a}_D^d$;
\item $\abs{\Nrd_{D_\bR/\bR}(a)} \leq (de)^{-de/2} \abs{a}_D^{de}$.
\end{enumerate}
\end{lemma}

\begin{proof}
Identify $D_\bR$ with $\rM_d(\bR)^e$ and write
\[ a = (a_1, \dotsc, a_e). \]
For each $i$, the matrix $a_i a_i^t \in \rM_d(\bR)$ is symmetric and positive definite and therefore diagonalisable with positive eigenvalues.  Let its eigenvalues be $\lambda_{i1}, \dotsc, \lambda_{id}$.
Note that $\abs{a_i}_F^2 = \Tr(a_i a_i^t) = \lambda_{i1} + \dotsb + \lambda_{id}$.
By the arithmetic mean--geometric mean inequality,
\begin{equation} \label{eqn:det-ai}
\det(a_i)^{2/d} = \det(a_i a_i^t)^{1/d} = \bigl( \lambda_{i1} \dotsm \lambda_{id} \bigr)^{1/d} \leq d^{-1} \bigl( \lambda_{i1} + \dotsb + \lambda_{id} \bigr) = d^{-1} \abs{a_i}_F^2.
\end{equation}

\begin{enumerate}[(i)]
\item We have $\Nrd_{D_\bR/F_\bR}(a) = (\det(a_1) I_d, \dotsc, \det(a_e)I_d)$ where $I_d$ denotes the identity matrix in $\rM_d(\bR)$.
Hence
\begin{align*}
    \abs{\Nrd_{D_\bR/F_\bR}(a)}_D^2
  & = \sum_{i=1}^e \abs{\det(a_i) I_d}_F^2
    = \sum_{i=1}^e d \abs{\det(a_i)}^2
\\& \leq \sum_{i=1}^e d \bigl( d^{-1} \abs{a_i}_F^2 \bigr)^{d}
    = d^{1-d} \sum_{i=1}^e \abs{a_i}_F^{2d}
\\& \leq d^{1-d} \Bigl( \sum_{i=1}^e \abs{a_i}^2 \Bigr)^d
    = d^{1-d} \abs{a}_D^{2d}.
\end{align*}

\item Using \eqref{eqn:det-ai} and another application of the AM-GM inequality,
\[ \abs{\Nrd_{D_\bR/\bR}(a)}^{2/de} = \Bigl( \prod_{i=1}^e \abs{\det(a_i)}^{2/d} \Bigr)^{1/e}
   \leq e^{-1} \sum_{i=1}^e d^{-1} \abs{a_i}_F^2 = (de)^{-1} \abs{a}_D^2.
\qedhere
\]
\end{enumerate}
\end{proof}

\subsection{The Hermite constant and Minkowski's theorems}

Let $\gamma_n$ denote the Hermite constant for $\bR^n$, that is, the smallest positive real number such that the following holds:
\textit{For every lattice $L$ in $\bR^n$ with the Euclidean norm and volume form, there exists a vector $v \in L$ satisfying $\abs{v} \leq \sqrt{\gamma_n} \covol(L)^{1/n}$.}

It is immediate from the definition that $\gamma_n \geq 1$ for all~$n$.

As a consequence of Minkowski's theorem on convex bodies,
\begin{equation} \label{eqn:minkowski-gamma}
\gamma_n \leq 4 \cV_n^{-2/n} = \tfrac{4}{\pi} \Gamma(\tfrac{n}{2} + 1)^{2/n}
\end{equation}
where $\cV_n$ denotes the volume of the unit ball in $\bR^n$.

\begin{lemma} \label{hermite-constant-bound}
For all positive integers $n$, we have $\gamma_n \leq n$.
\end{lemma}

\begin{proof}
According to \cite[Theorem~1.5]{AQ97}, $\Gamma(x) \leq x^{x-1}$ for all real numbers $x>1$.
Hence
\[ \Gamma(x+1) = x\Gamma(x) < x^x \]
for all $x > 1$.
Furthermore, for $x=1$, we have $\Gamma(x+1) = 1 = x^x$.
Thus $\Gamma(x+1) \leq x^x$ for all $x \geq 1$.
Plugging this into \eqref{eqn:minkowski-gamma}, we obtain $\gamma_n \leq \frac{4}{\pi} \cdot \frac{n}{2} < n$ for all $n \geq 2$.

It is clear that $\gamma_1=1$ \cite[Appendix]{cassels:geom-of-numbers}, so the lemma is also true for $n=1$.
\end{proof}

\Cref{hermite-constant-bound} is not optimal for large~$n$.
Indeed, our proof itself shows that $\gamma_n \leq \frac{2}{\pi}n$ for $n \geq 2$.
Using Stirling's approximation to the Gamma function, one can obtain $\gamma_n \leq 2n/\pi e + o(n)$ as $n \to +\infty$.
However we have chosen to use \cref{hermite-constant-bound} because we need a simple bound for the Hermite constant which is valid for all $n \geq 1$, without hidden constants or special cases for small~$n$, as we wish to avoid fiddly special cases when calculating the (non-optimal) constants in \cref{weakly-unitary-induction}.

\pagebreak 

A version of Minkowski's second theorem for the Euclidean norm also holds with the Hermite constant:

\begin{theorem} \label{minkowski-2nd} {\cite[Ch.~VIII, Theorem~1]{cassels:geom-of-numbers}}
For every lattice $L$ in $\bR^n$ with the Euclidean norm and volume form, there exist vectors $e_1, \dotsc, e_n \in L$ which form a basis for $\bR^n$ and which satisfy $\abs{e_1}\dotsm\abs{e_n} \leq \gamma_n^{n/2} \covol(L)$.
\end{theorem}

With some book-keeping, we can obtain a version of \cref{minkowski-2nd} for vector spaces over a division $\bQ$-algebra.  This is the same method as the proof of a version of Minkowski's second theorem over number fields in \cite[C.2.18]{BG06}.

\begin{proposition} \label{D-minkowski}
Let $D$ be a division $\bQ$-algebra.
Let $V$ be a left $D$-vector space of dimension~$m$.
Let $L$ be a $\bZ$-lattice in~$V$.
Let $\abs{\cdot}$ be any norm on $V_\bR$ induced by an inner product, and use the associated volume form to define $\covol(L)$.

Then there exists a $D$-basis $w_1, \dotsc, w_m$ for $V$ such that:
\begin{enumerate}[(i)]
\item $w_1, \dotsc, w_m \in L$;
\item $\abs{w_1} \abs{w_2} \dotsm \abs{w_m} \leq \gamma_{[D:\bQ]m}^{m/2} \covol(L)^{1/[D:\bQ]}$.
\end{enumerate}
\end{proposition}

\begin{proof}
Let $n = \dim_\bQ(V) = [D:\bQ]m$.
Choose $e_1, \dotsc, e_n \in L$ as in \cref{minkowski-2nd}.
Order the $e_i$ so that $\abs{e_i} \leq \abs{e_{i+1}}$ for all $i = 1, \dotsc, n-1$.

For $i = 1, \dotsc, m$, let $q_i$ denote the smallest positive integer $q$ such that the $D$-span of $e_1, \dotsc, e_q$ has $D$-dimension equal to $i$.
Let $w_i = e_{q_i}$.
By construction, for each~$i$, the $D$-span of $w_1, \dotsc, w_i$ has $D$-dimension equal to $i$. Hence $w_1, \dotsc, w_m$ is a $D$-basis for $V$.

For $1 \leq i \leq m$, the vectors $e_1, \dotsc, e_{q_i-1}$ are contained in a $D$-vector space of $D$-dimension $i-1$, so they are contained in a $\bQ$-vector space of $\bQ$-dimension at most $[D:\bQ](i-1)$.  These vectors are $\bQ$-linearly independent, so
\[ q_i - 1 \leq [D:\bQ](i-1). \]
Since the lengths $\abs{e_i}$ are increasing, we deduce that
\[ \abs{w_i}^{[D:\bQ]} \leq \abs{e_{[D:\bQ](i-1) + 1}}^{[D:\bQ]} \leq \prod_{j=1}^{[D:\bQ]} \abs{e_{[D:\bQ](i-1) + j}}. \]
Hence by \cref{minkowski-2nd},
\[ \prod_{i=1}^m \abs{w_i}^{[D:\bQ]} \leq \prod_{i=1}^n \abs{e_i} \leq \gamma_n^{n/2} \covol(L).
\qedhere
\]
\end{proof}

Let $D$ be a division $\bQ$-algebra, $R$ an order in $D$ and $L$ a torsion-free $R$-module of rank~$m$.
Combining \cref{D-minkowski} with \cref{minkowski-2nd} applied to $R$ and Hadamard's inequality, we could prove that there exist $w_1, \dotsc, w_m \in L$ forming a $D$-basis for $D \otimes_R L$ and satisfying $[L : Rw_1 + \dotsb + Rw_m] \leq \newC{D-index-multiplier} \abs{\disc(R)}^{m/2}$.
However this method of proof gives a constant $\refC{D-index-multiplier} > 1$, so this is weaker than \cref{minkowski-general-index}.

\subsection{Discriminants of bilinear forms}

If $\Lambda$ is a $\ZZ$-module, we write $\Lambda_\QQ$ for $\Lambda\otimes_\ZZ\QQ$. If $\Lambda$ is free of finite rank and $\phi \colon \Lambda_\QQ \times \Lambda_\QQ \to \QQ$ is a bilinear form, we write $\disc(\Lambda,\phi)$ for the determinant of the matrix $(\phi(e_i,e_j))_{i,j}$ where $\{e_1,\ldots,e_n\}$ is a $\ZZ$-basis for $\Lambda$ (the determinant is independent of the choice of basis).

\begin{lemma} \label{disc-lattice-complement}
Let $L$ be a free $\bZ$-module of finite rank and let $\phi \colon L \times L \to \bZ$ be a non-degenerate bilinear form.
Let $M \subset L$ be a $\bZ$-submodule such that $\phi_{|M \times M}$ is non-degenerate.
Let
\[ M^\perp = \{ x \in L : \phi(x,y) = 0 \text{ for all } y \in M \}. \]
Then
\begin{enumerate}[(i)]
\item $[L : M + M^\perp] \leq \abs{\disc(M, \phi)}$; and
\item $\abs{\disc(M^\perp, \phi)} \leq \abs{\disc(L, \phi)} \abs{\disc(M, \phi)}$.
\end{enumerate}
\end{lemma}

\begin{proof}
Since $\phi_{|M \times M}$ is non-degenerate, $L_\bQ = M_\bQ \oplus M_\bQ^\perp$.
Let $\pi \colon L_\bQ \to M_\bQ$ denote the projection with kernel $M_\bQ^\perp$.

If $x \in L$ and $\pi(x) \in M$, then $x-\pi(x) \in \ker(\pi) \cap L = M^\perp$.
Hence $x \in M + M^\perp$.
Conversely, if $x \in M + M^\perp$, it is clear that $\pi(x) \in M$.
Thus $\pi^{-1}(M) = M + M^\perp$.

Let
\[ M^* = \{ x \in M_\bQ : \phi(x, y) \in \bZ \text{ for all } y \in M \}. \]
If $x \in L$, then $\phi(\pi(x),y) = \phi(x,y) \in \bZ$ for all $y \in M$ so $\pi(x) \in M^*$.
Thus $\pi(L) \subset M^*$.
Thus we obtain
\[ L/(M + M^\perp) = L/\pi^{-1}(M) \cong \pi(L)/M \subset M^*/M. \]
It is well-known that $[M^*:M] = \abs{\disc(M,\phi)}$, so this proves (i).

Since $M$ and $M^\perp$ are orthogonal with respect to $\phi$,
\begin{align*}
    \abs{\disc(M, \phi)} \abs{\disc(M^\perp, \phi)}
  & = \abs{\disc(M + M^\perp, \phi)}
\\& = [L:M+M^\perp]^2 \abs{\disc(L, \phi)}
    \leq \abs{\disc(M, \phi)}^2 \abs{\disc(L, \phi)}.
\end{align*}
Since $\abs{\disc(M, \phi)} \neq 0$, this proves~(ii).
\end{proof}

\subsection{Orders and discriminants}\label{sec:orders}

Let $k = \bQ$ or $\bR$.
If $V$ is a finite-dimensional $k$-vector space, then a \defterm{$\bZ$-lattice} in $V$ means a $\bZ$-submodule $L \subset V$ such that the natural map $L \otimes_\bZ k \to V$ is an isomorphism.

Let $D$ be a semisimple $\bQ$-algebra.
An \defterm{order} in $D$ is a $\bZ$-lattice in $D$ which is also a subring.
Note that if $V$ is a $D$-vector space and $L$ is a $\bZ$-lattice in $V$, then $\Stab_D(L) = \{ a \in D : aL \subset L \}$ is an order in~$D$.  (This is proved on \cite[p.~109]{Rei75} when $V=D$, and the proof generalises.)

If $R$ is an order in $D$, the \defterm{discriminant} $\disc(R)$ is defined to be the discriminant of the $k$-bilinear form $(a,b) \mapsto \Tr_{D/k}(ab)$ on $R$, where $\Tr_{D/\bQ}$ is the \emph{non-reduced} trace.
The trace form of a semisimple algebra is non-degenerate, so $\disc(R) \neq 0$.
Furthermore, $\Tr_{D/\bQ}(a) \in \bZ$ for all $a \in R$, so $\disc(R) \in \bZ$.

If $D$ is a simple $\bQ$-algebra, then $\Trd_{D/\bQ}(a) \in \bZ$ for all $a \in R$ \cite[Theorem~10.1]{Rei75}.
Since $\Tr_{D/\bQ} = d \Trd_{D/\bQ}$, it follows that $\disc(R) \in d^{d^2e} \bZ$ so
\begin{equation} \label{eqn:disc-lower-bound}
\abs{\disc(R)} \geq d^{d^2e}.
\end{equation}

Now suppose that $(D,\dag)$ is a simple $\bQ$-algebra with a positive involution.
According to \cite[Lemma~5.6]{QRTUI}, for any order $R \subset D$, $\abs{\disc(R)}$ is equal to the discriminant of the symmetric bilinear form $(a,b) \mapsto \Tr_{D/k}(ab^\dag)$.
Consequently, $\abs{\disc(R)}$ is equal to $d^{d^2e}$ multiplied by the discriminant on $R$ of the positive definite bilinear form which induces the norm $\abs{\cdot}_D$.
We conclude that
\begin{equation} \label{eqn:disc-covol}
\abs{\disc(R)} = d^{d^2e}\covol(R)^2.
\end{equation}

For an order $R$ in a simple $\bQ$-algebra $D$, let $R^*$ denote the dual lattice
\[ R^* = \{ a \in D : \Trd_{D/\bQ}(ab) \in \bZ \text{ for all } b \in R \}. \]

\begin{lemma} \label{index-RcapF}
Let $D$ be a semisimple $k$-algebra and let $R$ be an order in $D$.
Let $F$ be the centre of $D$ and let $\cO$ be an order in $F$ which contains $R \cap F$.
Then
\[ [\cO:R \cap F]^2 \, \abs{\disc(\cO R)} \leq \abs{\disc(R)}. \]
\end{lemma}

\begin{proof}
This follows from the facts $\cO + R \subset \cO R$ and $[\cO + R : R] = [\cO : R \cap F]$.
\end{proof}

\begin{lemma} \label{dual-ideal}
Let $D$ be a simple $\bQ$-algebra.
Let $F$ be the centre of $D$ and let $\cO_F$ be the maximal order of~$F$.
Let $S$ be an order in $D$ which contains $\cO_F$.
Define $S^*$ analogously to $R^*$.
Then there exists an ideal $I \subset \cO_F$ such that $IS^* \subset S$ and
\[ \Nm(I) \leq d^{-d^2e} \abs{\disc(S)}. \]
\end{lemma}

\begin{proof}
Let $I = \{ x \in \cO_F : xS^* \subset S \}$, that is, the annihilator of the finite $\cO_F$-module $S^*/S$. 
By the structure theorem for finitely generated torsion modules over a Dedekind domain, there is an isomorphism of $\cO_F$-modules
\[ S^*/S \cong \cO_F/I_1 \oplus \cO_F/I_2 \oplus \dotsb \oplus \cO_F/I_r \]
for some $\cO_F$-ideals $I_1, I_2, \dotsc, I_r$.
We have $I = I_1 \cap I_2 \cap \dotsb \cap I_r \supset I_1 I_2 \dotsm I_r$ and so
\[ \Nm(I) \leq \Nm(I_1)\Nm(I_2) \dotsm \Nm(I_r) = [S^*:S]. \]

The index $[S^*:S]$ is equal to the absolute value of the discriminant of $S$ with respect to the reduced trace form. Thus $[S^*:S] = d^{-d^2e} \abs{\disc(S)}$.
\end{proof}

\begin{lemma} \label{conductor-S}
Let $D$ be a simple $\bQ$-algebra.
Let $F$ be the centre of $D$ and let $\cO_F$ be the maximal order of~$F$.
Let $R$ be an order in $D$.
Let $S = \cO_F R$.
Let $\fc$ be the conductor of $R \cap F$ (as an order in the number field~$F$).
Then
\[ \fc S \subset R \quad \text{ and } \quad \fc R^* \subset S^*. \]
\end{lemma}

\begin{proof}
From the definitions of $S$ and $\fc$,
\[ \fc S = \fc \cO_F R \subset (R \cap F) R \subset R. \]
If $c \in \fc$ and $a \in R^*$, then for all $b \in S$ we have
\[ \Trd_{D/\bQ}((ca)b) = \Trd_{D/\bQ}(a(cb)) \in \bZ \]
because $c$ is in the centre of $D$ and $cb \in \fc S \subset R$.
Thus $ca \in S^*$.
\end{proof}

\begin{lemma} \label{disc-R-S}
Let $D$ be a division $\bQ$-algebra and let $V$ be a left $D$-vector space of dimension~$m$.
Let $L$ be a $\bZ$-lattice in $V$ and consider the order $R = \Stab_D(L)$ of~$D$.
Let $S = \End_R(L)$ denote the ring of endomorphisms of $L$ commuting with~$R$.
Then
\[ \abs{\disc(S)} \leq \abs{\disc(R)}^{(d^2em+1)m^2}. \]
\end{lemma}

\begin{proof}
By \cref{minkowski-general-index}, there
is a $D$-basis $v_1, \dotsc, v_m$ for $V$ such that $v_1, \dotsc, v_m \in L$ and
\begin{equation} \label{eqn:R-free-index}
[L:Rv_1 + \dotsb + Rv_m] \leq \abs{\disc(R)}^{m/2}.
\end{equation}
Let $N = [L:Rv_1 + \dotsb + Rv_m]$ and $s = \dim_\bQ(\End_D(V)) = d^2em^2$.

Using the $D$-basis $v_1, \dotsc, v_m$, we identify $\End_D(V)$ with $\rM_m(D^\op)$.
Note that $\End_R(L)$ and $\rM_m(R^\op)$ are both $\bZ$-lattices in $\End_D(V)$.

For every $a \in \rM_m(R^\op) \subset \End_D(V)$, we have
\[ aNL \subset a(Rv_1 + \dotsb + Rv_m) \subset Rv_1 + \dotsb + Rv_m \subset L. \]
Hence $Na \in \End_R(L)$.

Thus $N\rM_m(R^\op) \subset \End_R(L)$.
Therefore
\[ \abs{\disc(S)} \leq N^{2s} \abs{\disc(\rM_m(R^\op))}
   = N^{2s} \abs{\disc(R)}^{m^2}. \]
Combining this with the bound for~$N$ from \eqref{eqn:R-free-index} proves the lemma.
\end{proof}

\subsection{Anti-symmetric elements in division algebras of type~II}

If $(D, \dag)$ is a division $\bQ$-algebra with involution, we define
\[ D^- = \{ a \in D : a^\dag = -a \}. \]
If $\psi \colon V \times V \to D$ is a $(D,\dag)$-skew-Hermitian form on a $D$-vector space $V$ and $x \in V$, then $\psi(x,x) \in D^-$, so $D^-$ is important for the study of weakly unitary bases (see section \ref{subsec:skew-hermitian-forms} for the definition of $(D,\dag)$-skew-Hermitian forms).

Let $(D,\dag)$ be a division $\bQ$-algebra with a positive involution of Albert type~II.
Choose an isomorphism $D_\bR \cong \rM_2(\bR)^e$ (as always, we implicitly assume that $\dag$ corresponds to matrix transpose on each factor).
Then $D_\bR^-$ consists of those elements of $\rM_2(\bR)^e$ in which all matrices are anti-symmetric.
Hence $D_\bR^-$ is a free $F_\bR$-module of rank $1$, so $D^-$ is a $1$-dimensional $F$-vector space.  The following lemma can be proved by calculations in $D_\bR \cong \rM_2(\bR)^e$.

\pagebreak 

\begin{lemma} \label{action-on-antisymm}
Let $(D, \dag)$ be a division $\bQ$-algebra with a positive involution of type~II.
Let $F$ be the center of $D$.
\begin{enumerate}[(i)]
\item If $a, b \in D^-$, then $ab \in F$.
\item If $a \in D$ and $b \in D^-$, then $aba^\dag = \Nrd_{D/F}(a)b$.
\end{enumerate}
\end{lemma}

\begin{lemma} \label{small-antisymm-star}
Let $(D, \dag)$ be a division $\bQ$-algebra with a positive involution of type~II.
Let $R$ be an order in $D$ and let $\eta \in \bZ_{>0}$ be a positive integer such that $\eta R^\dag \subset R$.
Then there exists $\omega \in D$ such that:
\begin{enumerate}[(i)]
\item $\omega \in D^- \setminus \{0\}$;
\item $\omega R^* \subset R$ and $R^* \omega \subset R$;
\item $\abs{\omega}_D \leq 2^{-4} \gamma_e^{1/2} \eta^7 \abs{\disc(R)}^{2/e}$.
\end{enumerate}
\end{lemma}

\begin{proof}
Let $F$ be the centre of $D$ and let $\cO_F$ be the maximal order of~$F$.
Let $\fc=\{\alpha\in\cO_F:\alpha\cO_F\subset R\cap F\}$ be the conductor of the order $R \cap F$ in $\cO_F$.
By \cite[(2)]{DCD00}, we have the following inclusion of ideals in $\bZ$:
\[ \disc_{F/\bQ}(R \cap F) \subseteq \Nm_{F/\bQ}(\fc) \disc_{F/\bQ}(\cO_F). \]
This leads to the following inequality of integers:
\[ \Nm(\fc) \abs{\disc(\cO_F)} \leq \abs{\disc(R \cap F)}. \]
Since also $\abs{\disc(R \cap F)} = [\cO_F : R \cap F]^2 \abs{\disc(\cO_F)}$, we deduce that
\[ \Nm(\fc) \leq [\cO_F : R \cap F]^2. \]

Let $S = \cO_F R$ and $S^- = S \cap D^-$.
Let $I$ be the ideal of $\cO_F$ given by \cref{dual-ideal} applied to $S$.
Let $J = \fc^2 I$ (as a product of ideals of $\cO_F$).
Then by \cref{conductor-S},
\begin{gather*}
   JSR^* = \fc SI \fc R^* \subset \fc SIS^* \subset \fc SS \subset \fc S \subset R,
\\ R^*JS = \fc I \fc R^* S \subset \fc IS^*S \subset \fc SS \subset \fc S \subset R.
\end{gather*}
Hence if we choose $\omega \in JS \cap D^- \setminus \{0\} = JS^- \setminus \{ 0 \}$, then it will satisfy (i) and~(ii).

Since $S^-$ is a non-zero $\cO_F$-submodule of an $F$-vector space of dimension~$1$, we can write $S^- = I^-\alpha$ for some ideal $I^- \subset \cO_F$ and some $\alpha \in D^-$, then use the multiplicativity of ideal norms in $\cO_F$ to conclude that
\[ \covol(JS^-) = \Nm(J) \covol(S^-), \]
where we measure covolumes in $D_\bR^-$ by the volume form associated with the restriction of the inner product $\Trd_{D_\bR/\bR}(ab^\dag)$.

Let $S^+ = \{ a \in S : a^\dag = a \}$.
Then $S^+ \cap S^- = \{0\}$.
Thus the sum $S^+ + S^-$ is direct.  This sum is also orthogonal because, if $a \in S^+$ and $b \in S^-$, then
\[ \Trd_{D/\bQ}(ab^\dag) = \Trd_{D/\bQ}((ab^\dag)^\dag) = \Trd_{D/\bQ}(ba^\dag) = -\Trd_{D/\bQ}(ab^\dag) \]
so $\Trd_{D/\bQ}(ab^\dag) = 0$.

For every $a \in S$, we have $\eta a^\dag \in \eta (\cO_F R)^\dag = \cO_F \eta R^\dag \subset \cO_F R = S$.
Hence
\[ 2\eta a = (\eta a+\eta a^\dag) + (\eta a-\eta a^\dag) \in S^+ + S^-. \]
Thus $2\eta S \subset S^+\oplus S^-$, so
\begin{equation} \label{eqn:covolS+S-}
\covol(S^+) \covol(S^-) = \covol(S^+ \oplus S^-) \leq \covol(2\eta S) = 2^{4e} \eta^{4e} \covol(S).
\end{equation}
Here we measure covolumes in both $D_\bR^-$ and $S^+ \otimes_{\bZ} \bR$ by the volume forms associated with the restriction of the inner product $\Trd_{D_\bR/\bR}(ab^\dag)$.

For all $a,b \in S$, $\eta ab^\dag \in S$ and so $\Trd_{D/\bQ}(ab^\dag) \in \eta^{-1}\bZ$.
Consequently $\covol(S^+) \geq \eta^{-\rk_\bZ(S^+)} = \eta^{-3e}$ so by \eqref{eqn:disc-covol} applied to $S$ and \eqref{eqn:covolS+S-},
\[ \covol(S^-) \leq \eta^{3e} \cdot 2^{4e} \eta^{4e} \covol(S) = 2^{4e} \eta^{7e} \cdot 2^{-2e} \abs{\disc(S)}^{1/2}. \]
Therefore, using \cref{dual-ideal},
\begin{align*}
    \covol(JS^-)
  & = \Nm(\fc)^2 \Nm(I) \covol(S^-)
\\& \leq [\cO_F : R \cap F]^4 \cdot 2^{-4e} \abs{\disc(S)} \cdot 2^{2e} \eta^{7e} \abs{\disc(S)}^{1/2}
\\& = 2^{-2e} \eta^{7e} [\cO_F : R \cap F]^4 \abs{\disc(S)}^{3/2}.
\end{align*}
Applying \eqref{eqn:disc-lower-bound} to $S$, we see that $\abs{\disc(S)} \geq 2^{4e}$.
Using \cref{index-RcapF}, we deduce that
\begin{align*}
    \covol(JS^-)
  & \leq 2^{-2e} \abs{\disc(S)}^{-1/2} \eta^{7e} [\cO_F : R \cap F]^4 \abs{\disc(S)}^2
    = 2^{-4e} \eta^{7e} \abs{\disc(R)}^2.
\end{align*}

Since $JS^-$ is a free $\bZ$-module of rank $e$, there exists $\omega \in JS^- \setminus \{0\}$ with
\[ \abs{\omega}_D \leq \sqrt{\gamma_e} \covol(JS^-)^{1/e} \leq \sqrt{\gamma_e} \cdot 2^{-4} \eta^7  \abs{\disc(R)}^{2/e}.
\qedhere \]
\end{proof}

\section{Skew-Hermitian forms over division algebras} \label{sec:skew-hermitian}

In this section, we introduce the notion of a $(D,\dag)$-skew-Hermitian form on a vector space over a division algebra $D$ with an involution, and explain how this is related to skew-symmetric forms over the base field. We define several notions of good behaviour for bases relative to $(D,\dag)$-skew-Hermitian forms, such as symplectic and unitary bases and a weakened version of these notions.  Finally we prove the existence of norms on $D$-vector spaces, which we call $D$-norms, which behave well relative to the action of~$D$ and to a $(D,\dag)$-skew-Hermitian form.

As in section~\ref{sec:division-algebras}, we are interested in applying the results of this section when $(D, \dag)$ is either a division $\bQ$-algebra with a positive involution of type I or~II, or the semisimple $\bR$-algebra which arises from such a $\bQ$-algebra by extending scalars to~$\bR$, but we state the results in greater generality whenever it is convenient.

\subsection{Skew-Hermitian forms} \label{subsec:skew-hermitian-forms}

Let $k$ be any field.
Let $(D, \dag)$ be a semisimple $k$-algebra with an involution.
Let $V$ be a left $D$-module.

A \defterm{$(D,\dag)$-skew-Hermitian form} on $V$ is a $k$-bilinear map $\psi \colon V \times V \to D$ which satisfies
\[ \psi(y,x) = -\psi(x,y)^\dag \text{ and } \psi(ax, by) = a\psi(x,y)b^\dag  \]
for all $a, b \in D$ and $x,y \in V$.
We say that a $(D,\dag)$-skew-Hermitian form $\psi$ is \defterm{non-degenerate} if, for every $x \in V \setminus \{0\}$, there exists $y \in V$ such that $\psi(x,y) \neq 0$.

A \defterm{$(D,\dag)$-compatible skew-symmetric form} on $V$ is a skew-symmetric $k$-bilinear map $\phi \colon V \times V \to k$ which satisfies
\[ \phi(ax, y) = \phi(x, a^\dag y) \]
for all $a \in D$ and $x,y \in V$.
A pair $(V,\phi)$, where $\phi$ is a $(D,\dag)$-compatible skew-symmetric form, is called a symplectic $(D,\dag)$-module in \cite[section~8]{Mil05}.

\begin{lemma} \label{tr-skew-hermitian-form}
Let $(D, \dag)$ be a semisimple $k$-algebra with an involution.
Let $V$ be a left $D$-module.
Then the map $\psi \mapsto {\Trd_{D/k}} \circ \psi$ is a bijection between the set of $(D,\dag)$-skew-Hermitian forms on $V$ and the set of $(D,\dag)$-compatible skew-symmetric forms on~$V$.
\end{lemma}

\begin{proof}
It is clear that, if $\psi$ is a $(D,\dag)$-skew-Hermitian form on $V$, then $\Trd_{D/k} \psi$ is a $(D,\dag)$-compatible skew-symmetric form.

Let $\phi$ be a $(D,\dag)$-compatible skew-symmetric form.
We shall show that there is a unique $(D,\dag)$-skew-Hermitian form on $V$ such that $\phi = \Trd_{D/k} \psi$.

For each $x,y \in V$, define a $k$-linear map $\alpha_{x,y} \colon D \to k$ by $\alpha_{x,y}(a) = \phi(ax, y)$.
Because $D$ is a semisimple $k$-algebra, $(a,b) \mapsto \Trd_{D/k}(ab)$ is a non-degenerate bilinear form $D \times D \to k$ \cite[Theorem~9.26]{Rei75}.
Hence there exists a unique element $\beta_{x,y} \in D$ such that
\[ \alpha_{x,y}(a) = \Trd_{D/k}(a\beta_{x,y}) \text{ for all } a \in D. \]
Define $\psi(x,y) = \beta_{x,y}$.
Using the uniqueness of the elements $\beta_{x,y}$, it is clear that the resulting function $\psi \colon V \times V \to D$ is $k$-bilinear.

If $a, b \in D$ and $x,y \in V$, then
\[ \Trd_{D/k}(ab\beta_{x,y}) = \alpha_{x,y}(ab) = \phi(abx, y) = \alpha_{bx,y}(a) = \Trd_{D/k}(a\beta_{bx,y}). \]
By uniqueness of $\beta_{bx,y}$, we deduce that $\psi$ is $D$-linear in the first variable.

If $a \in D$ and $x, y \in V$, then
\begin{align*}
    \Trd_{D/k}(a\beta_{x,y}) 
  & = \phi(ax, y) = -\phi(a^\dag y, x)
    = -\Trd_{D/k}(a^\dag \beta_{y, x}) = -\Trd_{D/k}(a\beta_{y,x}^\dag).
\end{align*}
Again by uniqueness of $\beta_{bx,y}$, $\psi(x,y) = -\psi(y,x)^\dag$.

Since $\psi$ is $D$-linear in the first variable and satisfies $\psi(x,y) = -\psi(y,x)^\dag$, it is also $(D,\dag)$-anti-linear in the second variable.
Thus it is $(D,\dag)$-skew-Hermitian.
\end{proof}

\begin{lemma} \label{orthog-complements}
Let $(D, \dag)$ be a semisimple $k$-algebra with an involution.
Let $V$ be a left $D$-module.
Let $\psi \colon V \times V \to k$ be a $(D,\dag)$-skew-Hermitian form and let $\phi = \Trd_{D/k} \psi \colon V \times V \to k$.

Let $W \subset V$ be a left $D$-submodule and define
\begin{gather*}
   W_\psi^\perp = \{ x \in V : \psi(w,x) = 0 \text{ for all } w \in W \},
\\ W_\phi^\perp = \{ x \in V : \phi(w,x) = 0 \text{ for all } w \in W \}.
\end{gather*}
Then $W_\psi^\perp = W_\phi^\perp$.

In particular, $W_\phi^\perp$ is a left $D$-submodule of $V$.
\end{lemma}

\begin{proof}
It is clear that $W_\psi^\perp \subset W_\phi^\perp$.

If $x \in W_\phi^\perp$ and $w \in W$ then, for all $a \in D$, we have $aw \in W$ and so
\[ \Trd_{D/\bQ}(a\psi(w,x)) = \Trd_{D/\bQ}(\psi(aw,x)) = \phi(aw,x) = 0. \]
By the non-degeneracy of the reduced trace form, it follows that $\psi(w,x) = 0$, that is, $x \in W_\psi^\perp$.
Thus $W_\phi^\perp \subset W_\psi^\perp$.
\end{proof}

\begin{corollary} \label{tr-non-deg}
Let $(D, \dag)$ be a semisimple $k$-algebra with an involution.
Let $V$ be a left $D$-module.
Let $\psi \colon V \times V \to k$ be a $(D,\dag)$-skew-Hermitian form and let $\phi = \Trd_{D/k} \psi \colon V \times V \to k$.
Then $\psi$ is non-degenerate if and only if $\phi$ is non-degenerate.
\end{corollary}

\begin{proof}
Apply \cref{orthog-complements} to $W = V$.
\end{proof}

\subsection{Weakly symplectic and weakly unitary bases} \label{subsec:unitary-bases}

Let $k$ be a field satisfying $\characteristic(k) \neq 2$ and let $(D,\dag)$ be a semisimple $k$-algebra with an involution.
Let $V$ be a free left $D$-module and let $\psi \colon V \times V \to D$ be a $(D,\dag)$-skew-Hermitian form.

We will now define special properties relative to $\psi$ which may be possessed by a basis of $V$.  The notion of (weakly) symplectic basis is useful when $D$ a division $\bQ$-algebra of type~I or $k^e$, and the notion of (weakly) unitary basis is useful when $D$ is a division $\bQ$-algebra of type~II or $\rM_2(k)^e$.

We say that a $D$-basis $v_1, \dotsc, v_m$ for $V$ is \defterm{weakly symplectic} if $\psi(v_i, v_j) = 0$ for all $i, j$ except when $\{i,j\} = \{2k-1,2k\}$ for some $k \in \bZ$.
If $\psi$ is non-degenerate, then this implies that $\psi(v_{2k-1}, v_{2k}) \neq 0$ for all~$k$.

We say that a $D$-basis $v_1, \dotsc, v_m$ is \defterm{symplectic} if $\psi$ is non-degenerate, the basis is weakly symplectic and furthermore, $\psi(v_{2k-1}, v_{2k}) = 1$ for all~$k$.
When $D$ is a field and $\dag=\id$, a $(D,\dag)$-skew-Hermitian form is the same thing as a symplectic form and this definition agrees with the usual definition of symplectic basis.

We say that a $D$-basis $v_1, \dotsc, v_m$ is \defterm{weakly unitary} if $\psi(v_i, v_j) = 0$ for all $i, j \in \{ 1, \dotsc, m \}$ such that $i \neq j$.
If $\psi$ is non-degenerate, then this implies that $\psi(v_i, v_i) \neq 0$ for all~$i$.

For a general division algebra with involution~$(D,\dag)$, there is no canonical choice of a non-zero element of $D^-$, so there is no natural definition of ``unitary basis'' with respect to a $(D,\dag)$-skew-Hermitian form.
In the special case $D_0 = \rM_d(k)^e$ with $d$ even, let us define
\[ \omega_0 = ( J_d, \dotsc, J_d ) \in D_0^- \]
where $J_d \in \rM_d(k)$ was defined in section~\ref{subsec:notation}.
If $V$ is a free left $D_0$-module equipped with a $(D_0,t)$-skew-Hermitian form $\psi_0$, then we say that a left $D_0$-basis $v_1, \dotsc, v_m$ of $V$ is \defterm{unitary} if it is weakly unitary and $\psi(v_i, v_i) = \omega_0$ for all $i = 1, \dotsc, m$.

\pagebreak 

If $(D,\dag)$ is a division $\bQ$-algebra with positive involution of type~II, $\alpha \colon (D_{0,\bR},\dag) \to (D_\bR,t)$ is an isomorphism of $\bR$-algebras with involution, and $V$ is a left $D$-vector space equipped with a $(D,\dag)$-skew-Hermitian form $\psi$, then we say that a left $D_\bR$-basis for $V_\bR$ is \defterm{$\alpha$-unitary} if it forms a unitary $D_{0,\bR}$-basis for $V_\bR$ viewed as a $D_{0,\bR}$-module via $\alpha$, with respect to the $(D_{0,\bR},t)$-skew-Hermitian form $\alpha^{-1} \circ \psi \colon V_\bR \times V_\bR \to D_{0,\bR}$.
The elements $v_i$ of an $\alpha$-unitary basis satisfy $\psi(v_i, v_i) = \alpha(\omega_0)$.

As an aside, which will be used in later calculations, we remark that, for any $a \in D_0$, the entries of the matrices which make up $a \omega_0$ are (up to signs) a permutation of the matrix entries making up $a$.  Hence
\begin{equation} \label{eqn:a-omega0}
\abs{a\omega_0}_{D_0} = \abs{a}_{D_0}.
\end{equation}

The following lemma shows how we can adjust a weakly symplectic or weakly unitary basis to become symplectic or $\alpha$-unitary.  Note that it works only over $D_\bR$, not over~$D$, because it requires taking square roots.

\begin{lemma} \label{semi-orthogonal-normalise}
Let $(D,\dag)$ be a division $\bQ$-algebra with a positive involution of type I or~II.
Let $\alpha \colon (\rM_d(\bR)^e,t) \to (D_\bR,\dag)$ be an isomorphism of $\bR$-algebras with involution.
Let $V$ be a left $D$-vector space equipped with a $(D,\dag)$-skew-Hermitian form $\psi \colon V \times V \to D$.
Let $v_1, \dotsc, v_m$ be a left $D$-basis for $V$ which is weakly symplectic (when $D$ has type~I) or weakly unitary (when $D$ has type~II).

Then there exist $s_1, \dotsc, s_m \in D_\bR^\times$
such that $s_1^{-1} v_1, \dotsc, s_m^{-1} v_m$ form a symplectic or $\alpha$-unitary $D_\bR$-basis for $V_\bR$ (according to the type of~$D$) and, for all $i$,
\[ \abs{s_i}_D \leq (de)^{1/4} \abs{\psi(v_i, v_j)}_D^{1/2} \]
where $j$ is the unique index such that $\psi(v_i, v_j) \neq 0$.
\end{lemma}

\begin{proof}
The proof is in two parts, depending on the type of~$D$.

\subsubsection*{Type~I case}
For each $k = 1, \dotsc, m/2$, $i=2k-1$ and $j=2k$, let
\[ t_k = (de)^{-1/2} \abs{\psi(v_i, v_j)}_D \in \bR_{>0}. \]
Let $s_i = t_k^{-1/2} \psi(v_i, v_j)$ and $s_j = t_k^{1/2}$.
Then
\[ \psi(s_i^{-1} v_i, s_j^{-1} v_j) = s_i^{-1} \psi(v_i, v_j) (s_j^{-1})^\dag = 1 \]
since $s_j^\dag = s_j$ and $t_k \in \bR$ is in the centre of $D_\bR$.

Furthermore
\[ \abs{s_i}_D = t_k^{-1/2} \abs{\psi(v_i, v_j)}_D = (de)^{1/4} \abs{\psi(v_i, v_j)}_D^{1/2} \]
while
\[ \abs{s_j}_D = t_k^{1/2} \abs{1}_D = (de)^{1/2} t_k^{1/2} = (de)^{1/4} \abs{\psi(v_i, v_j)}_D^{1/2}.  \]

\subsubsection*{Type~II case}
For each~$i$, $\psi(v_i, v_i) \in D^- \setminus\{0\} \subset F_\bR^\times \alpha(\omega_0)$.
Thus $\psi(v_i, v_i) = t_i\alpha(\omega_0)$ for some $t_i \in F_\bR^\times$.
Write $\alpha^{-1}(t_i) = (t_{i1}, \dotsc, t_{ie}) \in (\bR^\times)^e$.
Let $s_i = \alpha(s_{i1}, \dotsc, s_{ie}) \in D_\bR^\times$ where $s_{ij} \in \GL_2(\bR)$ are defined as follows:
\begin{align*}
   s_{ij} &= \fullmatrix{\sqrt{t_{ij}}}{0}{0}{\sqrt{t_{ij}}} \text{ if } t_{ij} \geq 0,
\\ s_{ij} &= \fullmatrix{\sqrt{-t_{ij}}}{0}{0}{-\sqrt{-t_{ij}}} \text{ if } t_{ij} < 0.
\end{align*}

Then
\[ \Nrd_{D_\bR/F_\bR}(s_i) = \alpha(\det(s_{i1}), \dotsc, \det(s_{ie})) = \alpha(t_{i1}, \dotsc, t_{ie}) = t_i. \]
Hence by \cref{action-on-antisymm},
\[ \psi(s_i^{-1} v_i, s_i^{-1} v_i) = s_i^{-1} \psi(v_i, v_i) (s_i^{-1})^\dag = \Nrd_{D_\bR/F_\bR}(s_i^{-1}) \psi(v_i, v_i) = \alpha(\omega_0). \]

Furthermore,
\begin{align*}
    \abs{s_i}_D^2
  & = \sum_{j=1}^e \Tr(s_{ij} s_{ij}^t)
    = \sum_{j=1}^e 2\abs{t_{ij}}
\\& \leq \sqrt{4e \sum_{j=1}^e \abs{t_{ij}}^2}
    = \sqrt{2e \Trd_{D_\bR/\bR}(t_i t_i^\dag)}
    = (2e)^{1/2} \abs{t_i}_D.
\end{align*}
By \eqref{eqn:a-omega0}, this implies that
\[ \abs{s_i}_D^2 \leq (2e)^{1/2} \abs{t_i\alpha(\omega_0)}_D = (de)^{1/2} \abs{\psi(v_i, v_i)}_D.
\qedhere
\]
\end{proof}

\begin{lemma} \label{D0-basis} \leavevmode
Let $D_0 = \rM_d(k)^e$ where $d = 1$ or $2$ and let $t$ denote the involution of $D_0$ which is transpose on each factor.
Let $V$ be a free left $D_0$-module and let $\psi_0$ be a non-degenerate $(D_0,t)$-skew-Hermitian form $V \times V \to D_0$.

Then there exists a $D_0$-basis $v_1, \dotsc, v_m$ for $V$ and a $k$-basis $a_1, \dotsc, a_{d^2e}$ for $D_0$ with the following properties:
\begin{enumerate}[(i)]
\item $\{ v_1, \dotsc, v_m \}$ is symplectic with respect to $\psi_0$ if $d=1$ and unitary if $d=2$.
\item $\{ a_1, \dotsc, a_{d^2e} \}$ is an orthonormal basis for $D_0$ with respect to $\abs{\cdot}_D$.
\item $\{ a_r v_j : 1 \leq r \leq d^2e, 1 \leq j \leq m \}$ is a symplectic $k$-basis for $V$ with respect to $\Trd_{D_0/k} \psi_0$.
\end{enumerate}
\end{lemma}

\begin{proof}
Write $B_0 = \rM_d(k)$.
Write $F_0$ for the centre of $D_0$, namely $k^e$.
Let $u_1, \dotsc, u_e$ denote the standard $k$-basis of $F_0 = k^e$.

Let $V_i = u_i V$.
Then $V = \bigoplus_{i=1}^e V_i$ and each $V_i$ is a free left $B_0$-module.
Because $V_0$ is a free left $D_0$-module, $\rk_{B_0}(V_1) = \dotsb = \rk_{B_0}(V_e)$.
Let $m$ denote this rank.

Because $\psi_0 : V \times V \to D_0$ is $F_0$-bilinear, it takes the form
\[ \psi_0((x_1, \dotsc, x_e), (y_1, \dotsc, y_e)) = (\psi_1(x_1, y_1), \dotsc, \psi_e(x_e, y_e)) \text{ for all } x_i, y_i \in V_i, \]
where $\psi_i \colon V_i \times V_i \to B_0$ are some non-degenerate $(B_0, t)$-skew-Hermitian forms.

Below, we shall prove the lemma with $(D_0, V, \psi_0)$ replaced by $(B_0, V_i, \psi_i)$, yielding a $B_0$-basis $v_{i1}, \dotsc, v_{im}$ for $(V_i, \psi_i)$ and a $k$-basis $b_1, \dotsc, b_{d^2}$ for $B_0$.
Then letting $v_j = (v_{1j}, \dotsc, v_{ej})$, we obtain a symplectic or unitary $D_0$-basis for $V$.
Furthermore $\{ v_i b_j : 1 \leq i \leq e, 1 \leq j \leq d^2 \}$ forms a $k$-basis for $D_0$ which satisfies (ii) and (iii).

Now we prove the lemma for $(B_0, V_i, \psi_i)$, breaking into two cases depending on~$d$.

\subsubsection*{Case $d=1$}
When $d=1$, $B_0=k$.
Each $V_i$ is a $k$-vector space of dimension $m$ and $\psi_i$ is a non-degenerate symplectic form $V_i \times V_i \to k$.
By the theory of symplectic forms, there exists a symplectic $k$-basis $\{ v_{i1}, \dotsc, v_{im} \}$ for~$V_i$, proving (i).

Choosing $b_1 = 1$ gives an orthonormal $k$-basis of $B_0$ with respect to $\abs{\cdot}_{B_0}$.
Since $\Trd_{B_0/k} \psi = \psi$, the bases $v_{i1}, \dotsc, v_{im}$ and $b_1$ satisfy (iii).

\subsubsection*{Case $d=2$, part~(i)}
We prove by induction on $m$ that there is a unitary $B_0$-basis $v_{i1}, \dotsc, v_{im}$ using the Gram--Schmidt method.

First we claim that there exists $z \in V_i$ such that $\psi_i(z,z) \neq 0$.
The image of $\psi_i \colon V_i \times V_i \to B_0$ is a two-sided ideal in $B_0$, which is a simple algebra, so this image is all of $B_0$.
In particular, we can choose $x,y \in V_i$ such that $\psi_i(x,y)$ is not skew-symmetric, that is, $\psi_i(x,y) + \psi_i(y,x) = \psi_i(x,y) + \psi_i(x,y)^t \neq 0$.
Then $\psi_i(x,x)$, $\psi_i(y,y)$ and $\psi_i(x+y,x+y)$ are not all zero.
Choosing $z$ to be one of $x$, $y$ and $x+y$, we obtain $\psi_i(z,z) \neq 0$.

Then $\psi_i(z,z) \in B_0^- = k J_d$ so $\psi_i(z,z) = s J_d$ for some $s \in k^\times$.
Letting $v_{i1} = \fullsmallmatrix{s^{-1}}{0}{0}{1}z$, we obtain that $\psi_i(v_{i1}, v_{i1}) = J_d$.

Let $V_i' = \{ v \in V_i : \psi_i(v_{i1}, v) = 0 \} = \{ v \in V_i : \psi_i(v, v_{i1}) = 0 \}$, which is a left $B_0$-submodule of $V_i$.
For every $b \in B_0 \setminus \{0\}$, we have
\begin{equation} \label{eqn:aJ}
\psi_i(bv_{i1}, v_{i1}) = b\psi_i(v_{i1}, v_{i1}) = bJ_d \neq 0
\end{equation}
and so $B_0v_{i1} \cap V_i' = \{ 0 \}$.
For every $v \in V_i$, we have
\[ v - \psi_i(v, v_{i1}) J_d^{-1} v_{i1} \in V_i'. \]
Hence $V_i = B_0v_{i1} \oplus V_i'$ as a direct sum of left $B_0$-modules.

By \eqref{eqn:aJ}, $bv_{i1} \neq 0$ for all $b \in B_0 \setminus \{0\}$.
Hence $\dim_k(B_0v_{i1}) = 4$ and so $\dim_k(V_0') = 4(m-1)$.
Every $B_0$-module whose $k$-dimension is a multiple of 4 is a free $B_0$-module,
so $B_0v_{i1}$ and $V_0'$ are free left $B_0$-modules.
By induction, there is a unitary $B_0$-basis $v_{i2}, \dotsc, v_{im}$ for $V_i'$.
Then $v_{i1}, v_{i2}, \dotsc, v_{im}$ is a unitary $B_0$-basis for $V_i$.

\subsubsection*{Case $d=2$, part (ii) and~(iii)}
Let
\[ b_1 = \fullsmallmatrix{1}{0}{0}{0},
\quad  b_2 = \fullsmallmatrix{0}{1}{0}{0}, 
\quad  b_3 = \fullsmallmatrix{0}{0}{1}{0}, 
\quad  b_4 = \fullsmallmatrix{0}{0}{0}{1} \in B_0 = \rM_2(k). \]
These form an orthonormal $k$-basis for $B_0$ with respect to $\abs{\cdot}_{B_0}$.

Since $\psi_i$ is $(B_0,t)$-skew-Hermitian,
\[ \psi_i(b_rv_{ij}, b_{r'}v_{ij'}) = b_r\psi_i(v_{ij}, v_{ij'})b_{r'}^\dag. \]
Thus if $j \neq j'$, we obtain $\psi_i(b_rv_{ij}, b_{r'}v_{ij'}) = 0$.
If $j = j'$, then $\psi_i(v_{ij}, v_{ij}) = J_2$, so we can calculate
\[ \Trd_{B_0/k} \psi(b_rv_{ij}, b_{r'}v_{ij}) = \Trd_{\rM_2(k)/k}(b_rJ_2b_{r'}^t)
= \begin{cases}
   1 &\text{if } (r,r') = (1,2) \text{ or } (3,4),
\\ -1 &\text{if } (r,r') = (2,1) \text{ or } (4,3),
\\ 0 &\text{otherwise}.
\end{cases}
\]
Thus the bases $v_{i1}, \dotsc, v_{im}$ and $b_1, \dotsc, b_4$ satisfy (iii) for $(B_0, V_i, \psi_i)$.
\end{proof}

\subsection{Discriminants and skew-Hermitian forms}

The following lemmas are useful for calculating discriminants of skew-Hermitian forms.

\begin{lemma} \label{disc-a-form}
Let $(D,\dag)$ be a division $\bQ$-algebra with an involution and let $R$ be an order in $D$.
Let $r_1, \dotsc, r_{d^2e}$ be a $\bZ$-basis for $R$.
For $a \in D$, let $T_a \in \rM_{d^2e}(\bQ)$ be the matrix with entries $(T_a)_{ij} = \Trd_{D/\bQ}(r_i a r_j^\dag)$.
Then
\[ \det(T_a) = \pm d^{-d^2e} \disc(R) \Nm_{D/\bQ}(a). \]
\end{lemma}

\begin{proof}
Let $M_a \in \rM_{d^2e}(\bZ)$ denote the matrix which represents ``multiplication by $a$ on the right'' with respect to the basis $r_1, \dotsc, r_{d^2e}$.
Using the facts that $\Trd_{D/\bQ}(xy) = \Trd_{D/\bQ}(yx)$ for all $x,y \in D$ and that $\Trd_{D/\bQ}$ is $\bQ$-linear,
\begin{align*}
    (T_a)_{ij}
    = \Trd_{D/\bQ}(r_i a r_j^\dag)
  & = \Trd_{D/\bQ}(r_j^\dag r_i a)
    = \Trd_{D/\bQ} \bigl( r_j^\dag \sum_{k=1}^{d^2e} (M_a)_{ki} r_k \bigr)
\\& = (M_a)_{ki} \sum_{k=1}^{d^2e} \Trd_{D/\bQ}(r_k r_j^\dag)
    = \sum_{k=1}^{d^2e} (M_a)_{ki} (T_1)_{kj}.
\end{align*}
Thus $T_a = M_a^t T_1$ so
\[ \det(T_a) = \det(M_a) \det(T_1) = \Nm_{D/\bQ}(a) \det(T_1). \]

Now $T_1$ is the Gram matrix of the bilinear form $(x,y) \mapsto d^{-1}\Trd_{D/\bQ}(xy^\dag)$ with respect to $r_1, \dotsc, r_{d^2e}$.
Hence by \cite[Lemma~5.6]{QRTUI}, $\det(T_1) = \pm d^{-d^2e} \disc(R)$.
\end{proof}

The following lemma allows us to calculate the discriminant of $\Trd_{D/\bQ} \psi$ on a free $R$-module generated by a weakly symplectic or weakly unitary basis (weakly symplectic or weakly unitary bases with respect to a non-degenerate form automatically satisfy the condition about uniqueness of a permutation~$\sigma$).
We have stated the lemma more generally because we shall require it in one additional case: when $m=2$ and the matrix with entries $\psi(v_i, v_j)$ has the form $\fullsmallmatrix{0}{*}{*}{*}$.

\pagebreak 

\begin{lemma} \label{disc-triangular}
Let $(D, \dag)$ be a division $\bQ$-algebra with a positive involution of type I or~II.
Let $V$ be a left $D$-vector space with a non-degenerate $(D,\dag)$-skew-Hermitian form $\psi \colon V \times V \to D$.
Let $R$ be an order in $D$.

Let $v_1, \dotsc, v_m$ be a $D$-basis for $V$.
Suppose that there is exactly one permutation $\sigma \in S_m$ for which $\psi(v_i, v_{\sigma(i)}) \neq 0$ for all $i = 1, \dotsc, m$.
Then
\[ \abs{\disc(Rv_1 + \dotsb + Rv_m, \Trd_{D/\bQ} \psi)} = d^{-d^2em} \abs{\disc(R)}^m \prod_{i=1}^m \abs{\Nm_{D/\bQ}(\psi(v_i, v_{\sigma(i)}))}. \]
\end{lemma}

\begin{proof}
Choose a $\bZ$-basis $r_1, \dotsc, r_{d^2e}$ for $R$.
Let $A \in \rM_n(\bQ)$ be the Gram matrix of the bilinear form $\Trd_{D/\bQ} \psi \colon V \times V \to \bQ$ with respect to the $\bQ$-basis $r_1 v_1, r_2 v_1, \dotsc, r_{d^2e} v_1, r_1 v_2, \dotsc, r_{d^2e} v_m$ for $V$.
Then $A$ is made up of square blocks $B_{ij} \in \rM_{d^2e}(\bQ)$ where $B_{ij}$ is the matrix with entries
\[ (B_{ij})_{k\ell} = \Trd_{D/\bQ} \psi(r_k v_i, r_\ell v_j) = \Trd_{D/\bQ}(r_k\psi(v_i, v_j)r_\ell^\dag). \]
In other words, $B_{ij}$ is equal to the matrix $T_{\psi(v_i, v_j)}$ as defined in \cref{disc-a-form}.

Let $\sigma \in S_m$ be the permutation from the hypothesis of the lemma.
By the permutation formula for determinants, the blocks $B_{i\sigma(i)}$ are the only blocks that
contribute to $\det(A)$ (although they need not be the only non-zero blocks of~$A$).
Indeed, we have
\begin{align*}
    \disc(Rv_1 + \dotsb + Rv_m, \Trd_{D/\bQ} \psi)
  & = \det(A)
    = \pm \prod_{i=1}^m \det(B_{i\sigma(i)}),
\end{align*}
which, by \cref{disc-a-form}, is equal to
\[
\pm \prod_{i=1}^m d^{-d^2e} \disc(R) \Nm_{D/\bQ}(\psi(v_i, v_{\sigma(i)})).
\qedhere
\]
\end{proof}

\subsection{\texorpdfstring{$D$}{D}-norms}

Let $k$ be a subfield of $\bR$.
Let $(D,\dag)$ be a semisimple $k$-algebra with a positive involution.
Let $V$ be a left $D$-module.
We say that a function $\abs{\cdot} \colon V_\bR \to \bR$ is a \defterm{$D$-norm} if it is a norm induced by a positive definite inner product on $V_\bR$ and it satisfies the inequality
\[ \abs{av} \leq \abs{a}_D \abs{v} \text{ for all } a \in D_\bR, x \in V_\bR. \]
Note that $\abs{\cdot}_D$ is itself a $D$-norm on $D_\bR$ thanks to \cref{length-submult}.

Let $\psi \colon V \times V \to D$ be a non-degenerate $(D,\dag)$-skew-Hermitian form.
We say that a $D$-norm $\abs{\cdot}$ is \defterm{adapted to $\psi$} if it satisfies the following two conditions:
\begin{enumerate}
\item $\covol(L_1) = 1$ where $L_1 \subset V$ is the $\bZ$-module generated by a symplectic $k$-basis for $V$ with respect to $\Trd_{D/k} \psi$.
(Note that a symplectic basis always exists since $\Trd_{D/k} \psi$ is a symplectic form over a field.  Furthermore, this condition is independent of the choice of symplectic $k$-basis, because the matrix transforming one symplectic basis into another has determinant~$1$.)
\item $\abs{\psi(x, y)}_D \leq \abs{x} \abs{y}$ for all $x, y \in V_\bR$.
\end{enumerate}

The following two lemmas demonstrate the significance of condition~(1) and establish the existence of a $D$-norm adapted to~$\psi$.

\begin{lemma} \label{covol-disc-lattice}
Let $(D,\dag)$ be a division $\bQ$-algebra with a positive involution of type I or~II.
Let $V$ be a left $D$-vector space with a non-degenerate $(D,\dag)$-skew-Hermitian form $\psi \colon V \times V \to D$.
Let $\abs{\cdot}$ be a $D$-norm on $V_\bR$ which satisfies condition~(1) from the definition of ``adapted to $\psi$.''
Let $L$ be a $\bZ$-lattice in $V$.

Then $\covol(L) = \abs{\disc(L)}^{1/2}$, where we use the volume form associated with $\abs{\cdot}$.
\end{lemma}

\begin{proof}
Choose a symplectic $\bQ$-basis $e_1, \dotsc, e_n$ for $V$ with respect to $\Trd_{D/\bQ} \psi$ and a $\bZ$-basis $f_1, \dotsc, f_n$ for $L$.
Let $M$ be the matrix which maps $e_1, \dotsc, e_n$ to $f_1, \dotsc, f_n$.
The $\bZ$-module generated by $e_1, \dotsc, e_n$ has covolume~$1$ by condition~(1).
Hence $\covol(L) = \abs{\det(M)}$.
The matrix with entries $\psi(f_i, f_j)$ is equal to $MJ_nM^t$.  So
\[ \disc(L) = \det(MJ_nM^t) = \det(M)^2.
\qedhere
\]
\end{proof}

\begin{lemma} \label{good-norm}
Let $(D,\dag)$ be a division $\bQ$-algebra with a positive involution of type I or~II.
Let $V$ be a left $D$-vector space of dimension $m$, equipped with a non-degenerate $(D,\dag)$-skew-Hermitian form $\psi \colon V \times V \to D$.
Then there exists a $D$-norm $\abs{\cdot}$ on $V_\bR$ which is adapted to $\psi$.
\end{lemma}
\begin{proof}
Identify $D_\bR$ with $\rM_d(\bR)^e$ where $d = 1$ or $2$.
By \cref{D0-basis}(i), there exists a symplectic or unitary $D_\bR$-basis $v_1, \dotsc, v_m$ for $V_\bR$, according to the type of $(D,\dag)$.
Define the following norm on $V_\bR$:
\[ \Bigabs{\sum_{i=1}^m x_i v_i} = \sqrt{\sum_{i=1}^m \abs{x_i}_D^2}. \]
This is induced by the inner product $\langle \sum_{i=1}^m x_i v_i, \sum_{j=1}^m y_j v_i \rangle = \Trd_{D_\bR/\bR} \left(\sum_{i=1}^m x_i y_i^\dag\right)$.
It is a $D$-norm by \cref{length-submult}.

Let $a_1, \dotsc, a_{d^2e}$ be the $\bR$-basis for $D_\bR$ given by \cref{D0-basis}.
Since $a_1, \dotsc, a_{d^2e}$ is an orthonormal $\bR$-basis for $D_\bR$ with respect to $\abs{\cdot}_D$,
$\{ a_j v_i \}$ is an orthonormal basis for $V_\bR$ with respect to $\abs{\cdot}$.
Therefore the lattice generated by $\{ a_j v_i \}$ has covolume~$1$.
According to \cref{D0-basis}(iii), $\{ a_j v_i \}$ is a symplectic basis for $V_\bR$ with respect to $\Trd_{D_\bR/\bR} \psi$.
Thus the norm $\abs{\cdot}$ satisfies condition~(1).

By the triangle inequality for $\abs{\cdot}_D$, we have
\begin{equation} \label{eqn:psi-triangle-ineq}
\Bigabs{\psi \bigl( \sum_{i=1}^m x_i v_i, \sum_{j=1}^m y_j v_j \bigr)}_D
\leq \sum_{i=1}^m \sum_{j=1}^m \abs{x_i \psi(v_i, v_j) y_j^\dag}_D.
\end{equation}
If $\psi(v_i, v_j) \neq 0$, then $\psi(v_i, v_j) = \pm 1$ or $\omega_0$ for all $i$, $j$ and so by \eqref{eqn:a-omega0},
$\abs{x_i}_D = \abs{x_i \psi(v_i, v_j)}_D$.  Hence
\begin{equation} \label{eqn:psi-vivj}
\abs{x_i \psi(v_i, v_j) y_j^\dag}_D \leq \abs{x_i \psi(v_i, v_j)}_D \abs{y_j^\dag}_D = \abs{x_i}_D \abs{y_j}_D.
\end{equation}

\pagebreak 

Let $\sigma \in S_m$ be the permutation such that $\psi(v_i, v_{\sigma(i)}) \neq 0$ (thus if $(D,\dag)$ has type~I, then $\sigma = (1,2)(3,4)(5,6)\dotsm$, while if $(D,\dag)$ has type~II, then $\sigma=\id$).
From \eqref{eqn:psi-triangle-ineq} and~\eqref{eqn:psi-vivj}, we obtain
\[ \Bigabs{\psi \bigl( \sum_{i=1}^m x_i v_i, \sum_{j=1}^m y_j v_j \bigr)}_D
\leq \sum_{i=1}^m \abs{x_i}_D \abs{y_{\sigma(i)}}_D. \]
By the Cauchy--Schwarz inequality, we get
\[ \Bigabs{\psi \bigl( \sum_{i=1}^m x_i v_i, \sum_{j=1}^m y_j v_j \bigr)}_D
   \leq \Bigl( \sum_{i=1}^m \abs{x_i}_D^2 \Bigr)^{1/2} \Bigl( \sum_{j=1}^m \abs{y_j}_D^2 \Bigr)^{1/2}
   = \Bigabs{\sum_{i=1}^m x_iv_i}_D \, \Bigabs{\sum_{j=1}^m y_jv_j}_D. \]
Thus the norm $\abs{\cdot}$ satisfies condition~(2).
\end{proof}

\section{Proof of Theorem~\ref{minkowski-hermitian-perfect}} \label{sec:minkowski-proof}

In this section we prove our main theorem on weakly unitary or symplectic bases with respect to skew-Hermitian forms.
The proof is based on the Gram--Schmidt process, following an inductive structure.
For technical reasons we may construct either one or two basis vectors at each step of the induction.
\Cref{pre-induction} constructs the new basis vector(s) for each induction step, and then \cref{weakly-unitary-induction} consists of calculations to keep track of the bounds during this induction.

\subsection{Initial vectors of a weakly symplectic or unitary basis}

We would like to begin by choosing $v_1$ to be the shortest non-zero vector in $V$ (with respect to a suitable $D$-norm), then inductively choosing a basis for $V^\perp$, the orthogonal complement of $Dv_1$.
However if we do this, $\psi(v_1, v_1)$ might be zero (indeed, if $D$ has type~I, then it must be zero) and then $Dv_1 + V^\perp$ is not a direct sum.

We will therefore instead choose either
\begin{enumerate}[(1)]
\item one short vector $v_1 \in V$ such that $\psi(v_1, v_1) \neq 0$; or
\item two short vectors $v_1, v_2 \in V$ such that the restriction of $\psi$ to $Dv_1 + Dv_2$ is non-degenerate, and $v_1, v_2$ form a weakly symplectic or weakly unitary basis for $Dv_1 + Dv_2$.
\end{enumerate}

Let $V^\perp$ denote the orthogonal complement of $v_1$ (in case~(1))  or of $Dv_1 + Dv_2$ (in case~(2)).
We will bound the discriminant of $\Trd_{D/\bQ} \psi$ restricted to $V^\perp$, and then inductively obtain a weakly symplectic or weakly unitary basis for $V^\perp$.
Combining this with $v_1$ and perhaps~$v_2$ gives the basis for~$V$ required to prove \cref{minkowski-hermitian-perfect}.

The following lemmas choose $v_1$ and perhaps~$v_2$ satisfying (1) or~(2) above.

\begin{lemma} \label{non-zero-permutation}
Let $(D,\dag)$ be a division $\bQ$-algebra with an involution.
Let $V$ be a left $D$-vector space, equipped with a non-degenerate $(D,\dag)$-skew-Hermitian form $\psi \colon V \times V \to D$.
Let $w_1, \dotsc, w_m$ be a $D$-basis for $V$.
Then there exists a permutation $\sigma \in S_m$ such that
$\psi(w_i, w_{\sigma(i)}) \neq 0$
for all $i = 1, \dotsc, m$.
\end{lemma}

\begin{proof}
If $D$ is a field, then the non-degeneracy of $\psi$ implies that the matrix with entries $\psi(w_i,w_j)$ has non-zero determinant.
Then the result is immediate by expressing the determinant as an alternating sum over permutations in $S_m$.
When $D$ is non-commutative, we cannot use determinants so we instead use a combinatorial argument (which is also valid in the commutative case).

The argument is based on Hall's theorem on distinct representatives of subsets:
\begin{theorem}[\cite{Hal35}] \label{hall-marriage-thm}
Let $T$ be a set and let $T_1, \dotsc, T_m$ be subsets of~$T$.
Then there exist pairwise distinct elements $a_1, \dotsc, a_m$ satisfying $a_i \in T_i$ if and only if, for every $k = 1, \dotsc, m$ and every choice of $k$ distinct indices $i_1, \dotsc, i_k$, we have
\begin{equation} \label{eqn:hall-marriage-condition}
\abs{T_{i_1} \cup \dotsb \cup T_{i_k}} \geq k.
\end{equation}
\end{theorem}

We shall apply \cref{hall-marriage-thm} with $T = \{ 1, \dotsc, m \}$ and
\[ T_i = \{ j : 1 \leq j \leq m, \, \psi(w_i, w_j) \neq 0 \}. \]
We claim that these sets $T_i$ satisfy the condition~\eqref{eqn:hall-marriage-condition} in \cref{hall-marriage-thm}.

Indeed, suppose that \eqref{eqn:hall-marriage-condition} is not satisfied for some distinct $i_1, \dotsc, i_k$.
Let $W$ denote the left $D$-vector space spanned by $w_{i_1}, \dotsc, w_{i_k}$.
Consider the vectors $w \in W$ satisfying
\begin{equation} \label{eqn:perp-neighbours}
\psi(w, w_j) = 0 \text{ for all } j \in T_{i_1} \cup \dotsb \cup T_{i_k}.
\end{equation}
Since \eqref{eqn:hall-marriage-condition} is not satisfied, \eqref{eqn:perp-neighbours} imposes $\abs{T_{i_1} \cup \dotsb \cup T_{i_k}} < k = \dim_D(W)$ left $D$-linear conditions on $w$.
Hence, there exists a non-zero $w \in W$ which satisfies~\eqref{eqn:perp-neighbours}.
By the definition of the sets $T_i$ and of~$W$, $w$ is also orthogonal to $w_j$ for every $j \not\in T_{i_1} \cup \dotsb \cup T_{i_k}$.
Thus $w$ is orthogonal to all of~$V$.  This contradicts the non-degeneracy of~$\psi$.

Hence by \cref{hall-marriage-thm}, there exist pairwise distinct $a_1, \dotsc, a_m$ such that $a_i \in T_i$.  Since $a_1, \dotsc, a_m$ are $m$ distinct elements of $\{1, \dotsc, m \}$, the function $\sigma(i) = a_i$ is a permutation of $\{ 1, \dotsc, m \}$.
By the definition of $T_i$, we have $\psi(w_i, w_{\sigma(i)}) \neq 0$ for all~$i$.
\end{proof}

\begin{lemma} \label{short-non-degenerate-vectors}
Let $(D,\dag)$ be a division $\bQ$-algebra with a positive involution.
Let $V$ be a left $D$-vector space of dimension $m$, equipped with a non-degenerate $(D,\dag)$-skew-Hermitian form $\psi \colon V \times V \to D$.
Let $\abs{\cdot}$ be a $D$-norm on $V_\bR$.
Let $w_1, \dotsc, w_m$ be a $D$-basis for $V$.

Then there exist $i,j \in \{ 1, \dotsc, m \}$ satisfying the following conditions:
\begin{enumerate}[(i)]
\item $\abs{w_i}\abs{w_j} \leq \bigl( \abs{w_1} \abs{w_2} \dotsm \abs{w_m} \bigr)^{2/m}$;
\item $\psi(w_i, w_j) \neq 0$;
\item if $i \neq j$, then $\psi(w_i,w_i) = 0$.
\end{enumerate}
\end{lemma}

\begin{proof}
Let $\sigma$ be a permutation as in \cref{non-zero-permutation}.

Choose $k \in \{ 1, \dotsc, m \}$ so that $\abs{w_k} \abs{w_{\sigma(k)}}$ is minimal.
Then
\[ \abs{w_k} \abs{w_{\sigma(k)}}
   \leq \bigl( \prod_{i=1}^m \abs{w_i} \abs{w_{\sigma(i)}} \bigr)^{1/m}
   = \bigl( \prod_{i=1}^m \abs{w_i} \cdot \prod_{j=1}^m \abs{w_j} \bigr)^{1/m}
   = \bigl( \prod_{i=1}^m \abs{w_i} \bigr)^{2/m}. \]
By the choice of $\sigma$, we have $\psi(w_k, w_{\sigma(k)}) \neq 0$.

If $\sigma(k) = k$, then $i=j=k$ satisfies the conditions of the lemma.

Otherwise choose $i \in \{ k, \sigma(k) \}$ so that $\abs{w_i}$ is minimal.

If $\psi(w_i, w_i) \neq 0$, then choosing $j=i$ satisfies the required conditions.

If $\psi(w_i, w_i) = 0$, then choose $j$ to be the element of $\{ k, \sigma(k) \}$ which is different from $i$.
This $i$ and $j$ satisfy the required conditions.
\end{proof}

In the remainder of this section, whenever we refer to a discriminant other than $\disc(R)$, we mean the discriminant of $\Trd_{D/\bQ} \psi$ restricted to the specified $\bZ$-module.

\begin{lemma} \label{pre-induction}
Let $(D,\dag)$ be a division $\bQ$-algebra with a positive involution of type I or~II.
Let $V$ be a left $D$-vector space with a non-degenerate $(D,\dag)$-skew-Hermitian form $\psi \colon V \times V \to D$.
Let $L$ be a $\bZ$-lattice in $V$ such that $\Trd_{D/\bQ} \psi(L \times L) \subset \bZ$.
Let $R$ be an order which is contained in $\Stab_D(L)$ and let $\eta \in \bZ_{>0}$ be a positive integer such that $\eta R^\dag \subset R$.

Then there exists an $R$-submodule $M \subset L$ with the following properties:
\begin{enumerate}[(i)]
\item $r := \dim_D(D \otimes_R M) = 1$ or $2$;
\item the restriction of $\psi$ to $M$ is non-degenerate;
\item $\abs{\disc(M)} \leq (\gamma_{d^2em}^2/d^3e)^{d^2er/2} \abs{\disc(R)}^r \abs{\disc(L)}^{r/m}$;
\item one of the following occurs:
\begin{enumerate}[(a)]
\item $D$ has type~I, $r=2$ and $M = Rv_1 + Rv_2$ for some $v_1, v_2$ such that
\[ \abs{\psi(v_1, v_2)}_D \leq \gamma_{em} \abs{\disc(L)}^{1/em}; \]

\item $D$ has type~II, $r=1$ and $M = Rv_1$ for some $v_1$ such that
\[ \abs{\psi(v_1, v_1)}_D \leq \gamma_{4em} \abs{\disc(L)}^{1/4em}; \]
\item $D$ has type~II, $r=2$ and there exist $D$-linearly independent vectors $v_1, v_2 \in M$ such that $\psi(v_1, v_2) = 0$,
\[ \abs{\psi(v_1, v_1)}_D, \abs{\psi(v_2, v_2)}_D \leq 2^{-5/2} \gamma_e^{1/2} \gamma_{4em}^2 \eta^7 \abs{\disc(R)}^{2/e} \abs{\disc(L)}^{1/2em}, \]
and
\[ [M:Rv_1 + Rv_2] \leq (\gamma_e/8e)^{2e} (\gamma_{4em}^2/8e)^{2e} \eta^{28e} \abs{\disc(R)}^8 \abs{\disc(L)}^{1/m}. \]
\end{enumerate}
\end{enumerate}
\end{lemma}

\begin{proof}
By \cref{good-norm}, there is a $D$-norm $\abs{\cdot}$ on $V_\bR$ adapted to $\psi$.
By \cref{D-minkowski}, there exists a $D$-basis $w_1, \dotsc, w_m$ for $V$ satisfying $w_1, \dotsc, w_m \in L$ and
\[ \abs{w_1} \dotsm \abs{w_m} \leq \gamma_{d^2em}^{m/2} \covol(L)^{1/d^2e} \leq \gamma_{d^2em}^{m/2} \abs{\disc(L)}^{1/2d^2e}, \]
where the second inequality comes from \cref{covol-disc-lattice}.

Choose $i, j$ as in \cref{short-non-degenerate-vectors}.
Since $\abs{\cdot}$ is adapted to~$\psi$, we have
\begin{equation} \label{eqn:psi-fi-fj-bound}
\abs{\psi(w_i, w_j)}_D \leq \abs{w_i}\abs{w_j} \leq \gamma_{d^2em} \abs{\disc(L)}^{1/d^2em}.
\end{equation}

\subsubsection*{Proof of (i)--(iii)}

Let $M = Rw_i + Rw_j$, so that $r=1$ if $i=j$ and $r=2$ if $i \neq j$.

If $i=j$, then by \cref{short-non-degenerate-vectors}, $\psi(w_i, w_i) \neq 0$, so the restriction of $\psi$ to $M$ is non-degenerate.

If $i \neq j$, then by \cref{short-non-degenerate-vectors}, $\psi(w_i, w_i) = 0$ and $\psi(w_i, w_j) \neq 0$.  Consequently for any vector $x \in M$, if $x \in Dw_i \setminus \{0\}$ then $\psi(x, w_j) \neq 0$ while if $x \not\in Dw_i$ then $\psi(x, w_i) \neq 0$.
Thus the restriction of $\psi$ to $M$ is non-degenerate.

By \cref{disc-triangular}, \cref{Nrd-length-bound} and \eqref{eqn:psi-fi-fj-bound}, we obtain that in both cases $i=j$ or $i \neq j$,
\begin{align*}
    \abs{\disc(M)}
  & = d^{-d^2er} \abs{\disc(R)}^r \abs{\Nm_{D/\bQ}(\psi(w_i, w_j))}^r
\\& = d^{-d^2er} \abs{\disc(R)}^r \abs{\Nrd_{D/\bQ}(\psi(w_i, w_j))}^{dr}
\\& \leq d^{-d^2er} \abs{\disc(R)}^r (de)^{-d^2er/2} \abs{\psi(w_i, w_j)}_D^{d^2er}
\\& \leq (d^3e)^{-d^2er/2} \abs{\disc(R)}^r \cdot \gamma_{d^2em}^{d^2er} \abs{\disc(L)}^{r/m}.
\end{align*}

\medskip

For the proof of (iv), we split into cases depending on the type of~$D$ and on whether $i=j$ or $i \neq j$.

\subsubsection*{Case~(a)}

If $D$ has type~I, then $D$ is a field and $\psi$ is a symplectic form.
Hence $\psi(v,v) = 0$ for all $v \in V$, so we must have $i \neq j$.

Let $v_1 = w_i$ and $v_2 = w_j$.
The bound in (iv)(a) is~\eqref{eqn:psi-fi-fj-bound}.

\subsubsection*{Case~(b)}

If $D$ has type~II and $i=j$, then let $v_1 = w_i$.
Then (iv)(b) holds thanks to \eqref{eqn:psi-fi-fj-bound}.

\subsubsection*{Case~(c)}

If $D$ has type~II and $i \neq j$, then choose $\omega \in D^-$ as in \cref{small-antisymm-star}.
Let
\[ w_j' = 2\psi(w_i, w_j)\omega w_j - \omega\psi(w_j, w_j) w_i. \]
Since $\Trd_{D/\bQ} \psi(L \times L) \subset \bZ$, $\psi(L \times L) \subset R^*$.
Hence $\psi(w_i, w_j)\omega$ and $\omega\psi(w_j, w_j) \in R$, so $w_j' \in Rw_i + Rw_j = M$.
Furthermore $w_j'$ and $w_i$ are $D$-linearly independent because $\psi(w_i, w_j)\omega \neq 0$.

By \cref{action-on-antisymm}(i), $\omega\psi(w_j, w_j), \psi(w_j, w_j)\omega \in F$.
Using this, along with the facts that $\psi(w_i,w_i) = 0$ and $(\omega \psi(w_i, w_j))^\dag = \psi(w_j, w_i) \omega$, we can calculate
\begin{align*}
    \psi(w_j', w_j')
  & = 2\psi(w_i, w_j) \omega \, \psi(w_j, w_j) \, (2\psi(w_i, w_j) \omega)^\dag
\\& \qquad - 2\psi(w_i, w_j) \omega \, \psi(w_j, w_i) \, (\omega \psi(w_j, w_j))^\dag
\\& \qquad - \omega \psi(w_j, w_j) \, \psi(w_i, w_j) \, (2\psi(w_i, w_j) \omega)^\dag
      + 0
\\& = (4-2-2) \psi(w_i, w_j) \omega \psi(w_j, w_j) \omega \psi(w_j, w_i)
\\& = 0.
\end{align*}

Using \cref{action-on-antisymm}(ii) and the fact that $\psi(w_i, w_i) = 0$, we can calculate
\begin{align*}
    \psi(w_j', w_i)
  & = 2 \psi(w_i, w_j) \omega \, \psi(w_j, w_i) - 0
    = -2 \Nrd_{D/F}(\psi(w_i, w_j))\omega.
\end{align*}
Thus $\psi(w_j', w_i) \in F\omega = D^-$, so $\psi(w_i, w_j') = -\psi(w_j', w_i)^\dag = \psi(w_j', w_i)$.

Now let
\[ v_1 = w_i - w_j', \qquad v_2 = w_i + w_j'. \]
Clearly $v_1, v_2 \in Rw_i + Rw_j' \subset M$.
Since $w_i = \frac{1}{2}(v_1 + v_2)$ and $w_j' = \frac{1}{2}(v_2 - v_1)$, the vectors $v_1$ and $v_2$ are $D$-linearly independent.

Since $\psi(w_j', w_i) = \psi(w_i, w_j')$ we can calculate
\begin{align*}
    \psi(v_1, v_2)
  & = \psi(w_i, w_i) + \psi(w_i, w_j') - \psi(w_j', w_i) - \psi(w_j', w_j')
    = 0,
\\ \psi(v_1, v_1)
  & = \psi(w_i, w_i) - \psi(w_i, w_j') - \psi(w_j', w_i) + \psi(w_j', w_j')
    = -2\psi(w_j', w_i),
\\ \psi(v_2, v_2)
  & = \psi(w_i, w_i) + \psi(w_i, w_j') + \psi(w_j', w_i) + \psi(w_j', w_j')
    = 2\psi(w_j', w_i).
\end{align*}
Consequently using \cref{length-submult,small-antisymm-star,Nrd-length-bound} and \eqref{eqn:psi-fi-fj-bound},
\begin{align*}
    \abs{\psi(v_1, v_1)}_D = \abs{\psi(v_2, v_2)}_D
  & = 2\abs{\psi(w_j', w_i)}_D
    \leq 4 \abs{\Nrd_{D/F}(\psi(w_i, w_j))}_D \abs{\omega}_D
\\& \leq 4 \cdot 2^{-1/2} \abs{\psi(w_i, w_j)}_D^2 \cdot 2^{-4} \eta^7 \sqrt{\gamma_e} \abs{\disc(R)}^{2/e}
\\& = 2^{-5/2} \sqrt{\gamma_e} \eta^7 \abs{\disc(R)}^{2/e} \cdot \gamma_{4em}^2 \abs{\disc(L)}^{2/4em}.
\end{align*}
This proves the first inequality in (iv)(c).

Using \cref{disc-triangular}, we have
\begin{align*}
    [M : Rv_1 + Rv_2]
  & = \frac{\abs{\disc(Rv_1 + Rv_2)}^{1/2}}{\abs{\disc(M)}^{1/2}}
\\& = \frac{\abs{\Nm_{D/\bQ}(\psi(v_1, v_1))}^{1/2} \abs{\Nm_{D/\bQ}(\psi(v_2, v_2))}^{1/2}}{\abs{\Nm_{D/\bQ}(\psi(w_i, w_j))}^{1/2} \abs{\Nm_{D/\bQ}(\psi(w_j, w_i))}^{1/2}}
\\& = \frac{\abs{\Nrd_{D/\bQ}(\psi(v_1, v_1))} \abs{\Nrd_{D/\bQ}(\psi(v_2, v_2))}}{\abs{\Nrd_{D/\bQ}(\psi(w_i, w_j))}^2}.
\end{align*}
Now by \cref{Nrd-length-bound} and the fact that if $a \in F$, then $\Nrd_{D/\bQ}(a) = \Nm_{F/\bQ}(a)^2$,
\begin{align*}
    \abs{\Nrd_{D/\bQ}(\psi(v_1, v_1))}
\\  = \abs{\Nrd_{D/\bQ}(\psi(v_2, v_2))}
  & = \abs{\Nrd_{D/\bQ}(4 \Nrd_{D/F}(\psi(w_i, w_j)) \omega)}
\\& = 4^{2e} \abs{\Nm_{F/\bQ}(\Nrd_{D/F}(\psi(w_i, w_j))}^2 \abs{\Nrd_{D/\bQ}(\omega)}
\\& = 4^{2e} \abs{\Nrd_{D/\bQ}(\omega)} \abs{\Nrd_{D/\bQ}(\psi(w_i, w_j))}^2.
\end{align*}
Therefore by \cref{Nrd-length-bound,small-antisymm-star} and \eqref{eqn:psi-fi-fj-bound},
\begin{align*}
    [M : Rv_1 + Rv_2]
  & = \frac{4^{4e} \abs{\Nrd_{D/\bQ}(\omega)}^2 \abs{\Nrd_{D/\bQ}(\psi(w_i, w_j))}^4}{\abs{\Nrd_{D/\bQ}(\psi(w_i, w_j))}^2}
\\& = 4^{4e} \abs{\Nrd_{D/\bQ}(\omega)}^2 \abs{\Nrd_{D/\bQ}(\psi(w_i, w_j))}^2
\\& \leq 4^{4e} \cdot (2e)^{-2e} \abs{\omega}_D^{4e} \cdot (2e)^{-2e} \abs{\psi(w_i, w_j)}_D^{4e}
\\& \leq 2^{4e} e^{-4e} \cdot 2^{-16e} \eta^{28e} \gamma_e^{2e} \abs{\disc(R)}^8 \cdot \gamma_{4em}^{4e} \abs{\disc(L)}^{4e/4em}.
\qedhere
\end{align*}
\end{proof}

\subsection{Inductive construction of weakly symplectic or unitary basis}

The following theorem is a slight generalisation of \cref{minkowski-hermitian-perfect}, together with explicit values for the constants.
Compared to \cref{minkowski-hermitian-perfect}, we only require $R \subset \Stab_D(L)$ (allowing $R \subsetneqq \Stab_D(L)$ is needed for the induction) and we add an additional parameter~$\eta$.
When $R = \Stab_D(L)$, the parameter~$\eta$ is controlled by \cref{R-cap-Rdag}.

\begin{proposition} \label{weakly-unitary-induction}
Let $(D,\dag)$ be a division $\bQ$-algebra with a positive involution of type I or~II.
Let $V$ be a left $D$-vector space with a non-degenerate $(D,\dag)$-skew-Hermitian form $\psi \colon V \times V \to D$.
Let $L$ be a $\bZ$-lattice in $V$ such that $\Trd_{D/\bQ} \psi(L \times L) \subset \bZ$.
Let $R$ be an order which is contained in $\Stab_D(L)$ and let $\eta \in \bZ_{>0}$ be a positive integer such that $\eta R^\dag \subset R$.

Then there exists a $D$-basis $v_1, \dotsc, v_m$ for $V$ such that:
\begin{enumerate}[(i)]
\item $v_1, \dotsc, v_m \in L$;
\item the basis is weakly symplectic (when $D$ has type~I) or weakly unitary (when $D$ has type~II) with respect to $\psi$;
\item the index of $Rv_1 + \dotsb + Rv_m$ in $L$ is bounded as follows:
\createC{so-index-mult} \createC{so-index-eta} \createC{so-index-R} \createC{so-index-L}
\[ [L : Rv_1 + \dotsb + Rv_m] \leq \refC{so-index-mult}(d,e,m) \eta^{\refC{so-index-eta}(d,e,m)} \abs{\disc(R)}^{\refC{so-index-R}(d,e,m)} \abs{\disc(L)}^{\refC{so-index-L}(d,e,m)}; \]
\item for all $i, j \in \{ 1, \dotsc, m \}$ such that $\psi(v_i, v_j) \neq 0$,
\createC{so-psi-mult} \createC{so-psi-eta} \createC{so-psi-R} \createC{so-psi-L}
\[ \abs{\psi(v_i, v_j)}_D \leq \refC{so-psi-mult}(d, e, m) \eta^{\refC{so-psi-eta}(d,e,m)}  \abs{\disc(R)}^{\refC{so-psi-R}(d,e,m)} \abs{\disc(L)}^{\refC{so-psi-L}(d,e,m)}. \]
\end{enumerate}

The inequalities (iii) and (iv) hold with the following values of the constants:
\begin{center}
\bgroup
\renewcommand{\arraystretch}{1.4}
\begin{tabular}{c|c|c}
  & $d=1$
  & $d=2$
\\ \hline
    $\refC{so-index-mult}(d,e,m)$
  & $(em^2)^{em(m+2)/16}$
  & $(2em^2)^{em(m+2)/2}$
\\  $\refC{so-index-eta}(d,e,m)$
  & $0$
  & $14em$
\\  $\refC{so-index-R}(d,e,m)$
  & $m(m+2)/8$
  & $m(m+16)/4$
\\  $\refC{so-index-L}(d,e,m)$
  & $(m-2)/4$
  & $(m-1)/2$
\\  $\refC{so-psi-mult}(d,e,m)$
  & $(em^2)^{(m(m+2)+24)/32}$
  & $(2em^2)^{(m(m+1)+14)/8}$
\\  $\refC{so-psi-eta}(d,e,m)$
  & $0$
  & $7$
\\  $\refC{so-psi-R}(d,e,m)$
  & $\bigl( m(m+2)-8 \bigr)/16e$
  & $\bigl( m(m+1)+26 \bigr)/16e$
\\  $\refC{so-psi-L}(d,e,m)$
  & $(m+2)/8e$
  & $(m+1)/8e$
\end{tabular}
\egroup
\end{center}
\end{proposition}

\begin{proof}
The proof is by induction on~$m = \dim_D(V)$.

Let $M$ be an $R$-submodule of $L$ as in \cref{pre-induction}.
Let $r = \dim_D(D \otimes_R M) = 1$ or~$2$.
Choose $v_1$ and perhaps $v_2$ as in \cref{pre-induction}(iv).

For part~(iii), the base case of the induction will be when $m=r$, and this is dealt with in the three cases below.
For part~(iv), the base case is when $m=0$, in which case (iv) is vacuously true.

Let $M^\perp$ be the orthogonal complement of $M$ in $L$ with respect to $\psi$.
By \cref{orthog-complements}, $M^\perp$ is also the orthogonal complement of $M$ in $L$ with respect to $\Trd_{D/\bQ} \psi$.
By \cref{disc-lattice-complement} and \cref{pre-induction}(iii),
\begin{equation} \label{eqn:disc-Mperp}
\abs{\disc(M^\perp)} \leq \abs{\disc(L)} \cdot \abs{\disc(M)} \leq (\gamma_{d^2em}^2/d^3e)^{d^2er/2} \abs{\disc(R)}^r \abs{\disc(L)}^{(m+r)/m}.
\end{equation}

Now $\psi$ restricted to $M^\perp$ is non-degenerate, $\dim_D(D \otimes_R M^\perp) = m-r < m$ and $R \subset \Stab_D(M^\perp)$ so we can apply the lemma inductively to $M^\perp$.
We obtain a $D$-basis $v_{r+1}, \dotsc, v_m$ for $D \otimes_R M^\perp$ whose elements lie in $M^\perp \subset L$.

Now $v_1, \dotsc, v_r \in M$ are orthogonal to $v_{r+1}, \dotsc, v_m$ and $v_1, \dotsc, v_r$ form a weakly symplectic or weakly unitary $D$-basis for $D \otimes_R M$.
Hence by induction $v_1, \dotsc, v_m$ form a weakly symplectic or weakly unitary $D$-basis for~$V$.

Thus (i) and~(ii) are satisfied.

By induction,
\begin{align}
  & \phantom{{}\leq{}} [M^\perp:Rv_{r+1} + \dotsb + Rv_m]
\notag
\\& \leq \refC{so-index-mult}(d,e,m-r) \eta^{\refC{so-index-eta}(d,e,m-r)} \abs{\disc(R)}^{\refC{so-index-R}(d,e,m-r)} \abs{\disc(M^\perp)}^{\refC{so-index-L}(d,e,m-r)}
\notag
\\& \leq \refC{so-index-mult}(d,e,m-r) \eta^{\refC{so-index-eta}(d,e,m-r)} \abs{\disc(R)}^{\refC{so-index-R}(d,e,m-r)}
\notag
\\& \qquad \cdot (\gamma_{d^2em}^2/d^3e)^{d^2er/2 \cdot \refC{so-index-L}(d,e,m-r)} \abs{\disc(R)}^{r\refC{so-index-L}(d,e,m-r)} \abs{\disc(L)}^{(m+r)/m \cdot \refC{so-index-L}(d,e,m-r)}.
\label{eqn:Mperp-index}
\end{align}

We now split into cases depending on the type of~$D$ and on whether $r = 1$ or~$2$, as in \cref{pre-induction}(iv).
The proofs in the three cases are very similar, with just the details of the calculations varying.
For each case, the proofs of (iii) and~(iv) are independent of each other.

\medskip

\subsubsection*{Case~(a), part~(iii)}
This is the case when $D$ has type~I and $r=2$.

When $m=r=2$, from \cref{pre-induction}(iii) and (iv)(a), we have
\[ [L:Rv_1 + Rv_2] = [L:M] = \frac{\abs{\disc(M)}^{1/2}}{\abs{\disc(L)}^{1/2}}
\leq (\gamma_{2e}^2/e)^{e/2} \abs{\disc(R)}. \]
This establishes (iii) when $m=2$ because
\begin{gather*}
      (\gamma_{2e}^2/e)^{e/2} \leq (4e)^{e/2} = \refC{so-index-mult}(1,e,2),
\\    \refC{so-index-eta}(1,e,2)  = 0,
\quad \refC{so-index-R}(1,e,2)    = 1,
\quad \refC{so-index-L}(1,e,2)    = 0.
\end{gather*}

When $m \geq 3$, we have, using the fact that $M = Rv_1 + Rv_2$, \cref{disc-lattice-complement}, \cref{pre-induction}(iii) and \eqref{eqn:Mperp-index},
\begin{align*}
  & \phantom{{} = {}}  [L : Rv_1 + \dotsb + Rv_m]
    = [L : M + M^\perp] [M^\perp : Rv_3 + \dotsb + Rv_m]
\\& \leq \abs{\disc(M)} [M^\perp : Rv_3 + \dotsb + Rv_m]
\\& \leq (\gamma_{em}^2/e)^{e} \abs{\disc(R)}^2 \abs{\disc(L)}^{2/m}
  \cdot \refC{so-index-mult}(1,e,m-2) \, (\gamma_{em}^2/e)^{e\refC{so-index-L}(1,e,m-2)}
\\& \qquad \cdot \abs{\disc(R)}^{\refC{so-index-R}(1,e,m-2) + 2\refC{so-index-L}(1,e,m-2)} \abs{\disc(L)}^{(m+2)/m \cdot \refC{so-index-L}(1,e,m-2)}.
\end{align*}

Now we can calculate:
for the multiplicative constant:
\begin{align*}
  & \phantom{{} = {}} \refC{so-index-mult}(1,e,m-2) \, (\gamma_{em}^2/e)^{e(1+\refC{so-index-L}(1,e,m-2))}
\\& = (e(m-2)^2)^{e(m-2)m/16} (\gamma_{em}^2/e)^{em/4}
\\& \leq (em^2)^{e(m-2)m/16} \cdot (em^2)^{m/4}
\\& = (em^2)^{e(m^2-2m + 4m)/16}
\\& = \refC{so-index-mult}(1,e,m),
\end{align*}
for the exponent of $\abs{\disc(R)}$:
\begin{align*}
    2 + \refC{so-index-R}(1,e,m-2) + 2\refC{so-index-L}(1,e,m-2)
  & = 2 + \frac{(m-2)m}{8} + 2 \cdot \frac{m-4}{4}
\\& = \refC{so-index-R}(1,e,m),
\end{align*}
for the exponent of $\abs{\disc(L)}$:
\begin{align*}
    \frac{2}{m} + \frac{(m+2)}{m} \cdot \refC{so-index-L}(1,e,m-2)
  & = \frac{2}{m} + \frac{(m+2)(m-4)}{4m}
\\& = \frac{8+(m^2-2m-8)}{4m} = \frac{(m-2)m}{4m} = \refC{so-index-L}(1,e,m).
\end{align*}

\subsubsection*{Case~(a), part~(iv)}
For $i=1$, $j=2$, \cref{pre-induction}(iv)(a) gives
\begin{equation} \label{eqn:case-a-iv-base}
\abs{\psi(v_1, v_2)}_D \leq \gamma_{em} \abs{\disc(L)}^{1/em}.
\end{equation}

This establishes (iv) when $i=1$, $j=2$ because,
using \cref{hermite-constant-bound} and the fact that $m \geq 2$ so $1 \leq (m(m+2)+24)/32$ and $1 \leq (m(m+2)+8)/16$,
\begin{align*}
\gamma_{em} & \leq em \leq \refC{so-psi-mult}(1,e,m),
\\ 0 & \leq \frac{m(m+2)-8}{16e} = \refC{so-psi-R}(1,e,m),
\\ \frac{1}{em} & \leq \frac{2 \cdot 4}{8em} \leq \frac{m(m+2)}{8em} = \refC{so-psi-L}(1,e,m).
\end{align*}

For $i, j \geq 3$, induction gives
\begin{align*}
    \abs{\psi(v_i, v_j)}_D
  & \leq \refC{so-psi-mult}(1,e,m-2) \abs{\disc(R)}^{\refC{so-psi-R}(1,e,m-2)} \disc(M^\perp)^{\refC{so-psi-L}(1,e,m)}
\\& \leq \refC{so-psi-mult}(1,e,m-2) \abs{\disc(R)}^{\refC{so-psi-R}(1,e,m-2)}
\\& \qquad \cdot \bigl( (\gamma_{em}^2/e)^{e}  \abs{\disc(R)}^2 \abs{\disc(L)}^{(m+2)/m} \bigr)^{\refC{so-psi-L}(1,e,m-2)}.
\end{align*}
Now we can calculate:
for the multiplicative constant (using \cref{hermite-constant-bound}):
\begin{align*}
  & \phantom{{} = {}} \refC{so-psi-mult}(1,e,m-2) \, (\gamma_{em}^2/e)^{e\refC{so-psi-L}(1,e,m-2)}
\\& = (e(m-2)^2)^{((m-2)m+24)/32} \, (\gamma_{em}^2/e)^{m/8}
\\& \leq (em^2)^{((m-2)m+24)/32} \cdot (em^2)^{m/8}
\\& = (em^2)^{(m^2-2m+24 + 4m)/32}
\\& = \refC{so-psi-mult}(1,e,m),
\end{align*}
for the exponent of $\abs{\disc(R)}$:
\begin{align*}
    \refC{so-psi-R}(1,e,m-2) + 2\refC{so-psi-L}(1,e,m-2)
  & = \frac{(m-2)m-8}{16e} + 2 \cdot \frac{m}{8e}
    = \refC{so-psi-R}(1,e,m),
\end{align*}
for the exponent of $\abs{\disc(L)}$:
\begin{align*}
    \frac{m+2}{m} \refC{so-psi-L}(1,e,m-2)
  & = \frac{(m+2)}{m} \cdot \frac{m}{8e}
    = \refC{so-psi-L}(1,e,m).
\end{align*}

\medskip

\subsubsection*{Case~(b), part~(iii)}
In this case, $D$ has type~II and $r=1$.

When $m=r=1$, from \cref{pre-induction}(iii) and (iv)(b), we have
\[ [L:Rv_1] = [L:M] = \frac{\abs{\disc(M)}^{1/2}}{\abs{\disc(L)}^{1/2}} \leq (\gamma_{4e}^2/8e)^{e} \abs{\disc(R)}^{1/2}. \]
This establishes (iii) when $m=1$ because
\begin{gather*}
      (\gamma_{4e}^2/8e)^e \leq (2e)^e \leq (2e)^{3e/2} = \refC{so-index-mult}(2,e,1),
\\    \refC{so-index-eta}(2,e,1) = 14e > 0,
\quad \refC{so-index-R}(2,e,1) = 17/4 > 1/2,
\quad \refC{so-index-L}(2,e,1) = 0.
\end{gather*}

When $m \geq 2$, we have (using $M = Rv_1$, \cref{disc-lattice-complement}, \cref{pre-induction}(iii) and \eqref{eqn:Mperp-index})
\begin{align*}
  & \phantom{{} = {}}  [L : Rv_1 + \dotsb + Rv_m]
    = [L : M + M^\perp] [M^\perp : Rv_2 + \dotsb + Rv_m]
\\& \leq \abs{\disc(M)} [M^\perp : Rv_2 + \dotsb + Rv_m]
\\& \leq (\gamma_{4em}^2/8e)^{2e} \abs{\disc(R)} \abs{\disc(L)}^{1/m}
\\& \qquad \cdot \refC{so-index-mult}(2,e,m-1) \eta^{\refC{so-index-eta}(2,e,m-1)} (\gamma_{4em}^2/8e)^{2e\refC{so-index-L}(2,e,m-1)}
\\& \qquad \cdot \abs{\disc(R)}^{\refC{so-index-R}(2,e,m-1) + \refC{so-index-L}(2,e,m-1)} \abs{\disc(L)}^{(m+1)/m \cdot \refC{so-index-L}(2,e,m-1)}.
\end{align*}

Now we can calculate:
for the multiplicative constant:
\begin{align*}
  & \phantom{{} = {}} \refC{so-index-mult}(2,e,m-1) (\gamma_{4em}^2/8e)^{2e(1+\refC{so-index-L}(2,e,m-1))}
\\& = (2e(m-1)^2)^{e(m-1)(m+1)/2} (\gamma_{4em}^2/8e)^{em}
\\& \leq (2em^2)^{e(m-1)(m+1)/2} \cdot (2em^2)^{em}
\\& = (2em^2)^{e(m^2-1 + 2m)/2}
\\& \leq \refC{so-index-mult}(2,e,m),
\end{align*}
for the exponent of $\eta$:
\[ \refC{so-index-eta}(2,e,m-1) = 14e(m-1) < \refC{so-index-eta}(2,e,m), \]
for the exponent of $\abs{\disc(R)}$:
\begin{align*}
    1 + \refC{so-index-R}(2,e,m-1) + \refC{so-index-L}(2,e,m-1)
  & = 1 + \frac{(m-1)(m+15)}{4} + \frac{m-2}{2}
\\& = \frac{m^2+16m-15}{4} < \refC{so-index-R}(2,e,m),
\end{align*}
for the exponent of $\abs{\disc(L)}$:
\begin{align*}
    \frac{1}{m} + \frac{(m+1)}{m} \cdot \refC{so-index-L}(2,e,m-1)
  & = \frac{1}{m} + \frac{(m+1)(m-2)}{2m}
\\& = \frac{2+(m^2-m-2)}{2m} = \frac{(m-1)m}{2m} = \refC{so-index-L}(2,e,m).
\end{align*}

\subsubsection*{Case~(b), part~(iv)}
For $i=j=1$, \cref{pre-induction}(iv)(b) gives
\begin{equation} \label{eqn:case-b-iv-base}
\abs{\psi(v_1, v_1)}_D \leq \gamma_{4em} \abs{\disc(L)}^{1/4em}.
\end{equation}
This establishes (iv) for $i=j=1$ because, using \cref{hermite-constant-bound} and the fact that $m \geq 1$ so $(m(m+1)+14)/8 \geq 2$,
\begin{align*}
    \gamma_{4em}
  & \leq 4em \leq (2em^2)^2 \leq \refC{so-psi-mult}(2,e,m),
\\[3pt] 0 & < 7 = \refC{so-psi-eta}(2,e,m),
\\ 0 & < \frac{1 \cdot 2 + 26}{16e} \leq \frac{m(m+1)+26}{16e} = \refC{so-psi-R}(2,e,m),
\\ \frac{1}{4em} & \leq \frac{2}{8e} \leq \frac{m+1}{8e} = \refC{so-psi-L}(2,e,m).
\end{align*}

For $i=j \geq 2$, induction gives
\begin{align*}
    \abs{\psi(v_j, v_j)}_D
  & \leq \refC{so-psi-mult}(2,e,m-1) \eta^{\refC{so-psi-eta}(2,e,m-1)} \abs{\disc(R)}^{\refC{so-psi-R}(2,e,m-1)} \abs{\disc(M^\perp)}^{\refC{so-psi-L}(2,e,m-1)}
\\& \leq \refC{so-psi-mult}(2,e,m-1) \eta^{\refC{so-psi-eta}(2,e,m-1)} \abs{\disc(R)}^{\refC{so-psi-R}(2,e,m-1)}
\\& \qquad \cdot \bigl( (\gamma_{4em}^2/8e)^{2e} \abs{\disc(R)} \abs{\disc(L)}^{(m+1)/m} \bigr)^{\refC{so-psi-L}(2,e,m-1)}.
\end{align*}
Now we can calculate:
for the multiplicative constant:
\begin{align*}
  & \phantom{{}={}} \refC{so-psi-mult}(2,e,m-1) (\gamma_{4em}^2/8e)^{2e\refC{so-psi-L}(2,e,m-1)}
\\& = (2e(m-1)^2)^{((m-1)m+14)/8} \cdot (\gamma_{4em}^2/8e)^{m/4}
\\& \leq (2em^2)^{(m^2-m+14)/8} \cdot (2em^2)^{2m/8}
\\& = (2em^2)^{(m^2+m+14)/8}
\\& = \refC{so-psi-mult}(2,e,m),
\end{align*}
for the exponent of $\eta$:
\[ \refC{so-psi-eta}(2,e,m-1) = 7 = \refC{so-psi-eta}(2,e,m), \]
for the exponent of $\abs{\disc(R)}$:
\begin{align*}
    \phantom{{} = {}} \refC{so-psi-R}(2,e,m-1) + \refC{so-psi-L}(2,e,m-1)
  & = \frac{(m-1)m+26}{16e} + \frac{m}{8e}
    = \refC{so-psi-R}(2, e, m),
\end{align*}
for the exponent of $\abs{\disc(L)}$:
\begin{align*}
    \frac{m+1}{m} \refC{so-psi-L}(2,e,m-1)
  & = \frac{(m+1)}{m} \cdot \frac{m}{8e}
    = \frac{m+1}{8e}
    = \refC{so-psi-L}(2,e,m).
\end{align*}

\medskip

\subsubsection*{Case~(c), part~(iii)}
This is the case where $D$ has type~II and $r=2$.

When $m=r=2$, from \cref{pre-induction}(iii) and (iv)(c), we have
\begin{align*}
  & \phantom{{} = {}}  [L:Rv_1 + Rv_2]
    = [L:M][M:Rv_1 + Rv_2]
\\& = \frac{\abs{\disc(M)}^{1/2}}{\abs{\disc(L)}^{1/2}} [M:Rv_1 + Rv_2]
\\& \leq \frac{(\gamma_{8e}^2/8e)^{2e} \abs{\disc(R)} \abs{\disc(L)}^{1/2}}{\abs{\disc(L)}^{1/2}} (\gamma_e/8e)^{2e} (\gamma_{8e}^2/8e)^{2e} \eta^{28e} \abs{\disc(R)}^8 \abs{\disc(L)}^{1/2}
\\& = (\gamma_e/8e)^{2e} (\gamma_{8e}^2/8e)^{4e} \eta^{28e} \abs{\disc(R)}^9 \abs{\disc(L)}^{1/2}.
\end{align*}
This establishes (iii) when $m=2$ because
\begin{gather*}
      (\gamma_e/8e)^{2e} (\gamma_{8e}^2/8e)^{4e}
      \leq 1 \cdot (8e)^{4e}
      = \refC{so-index-mult}(2,e,2),
\\    \refC{so-index-eta}(2,e,2) = 28e,
\quad \refC{so-index-R}(2,e,2) = 9,
\quad \refC{so-index-L}(2,e,2) = 1/2.
\end{gather*}

When $m \geq 3$, we have (using \cref{disc-lattice-complement}, \cref{pre-induction}(iv)(c) and \eqref{eqn:Mperp-index})
\begin{align*}
  & \phantom{{} = {}}  [L : Rv_1 + \dotsb + Rv_m]
\\& = [L : M + M^\perp] [M : Rv_1 + Rv_2] [M^\perp : Rv_3 + \dotsb + Rv_m]
\\& \leq \abs{\disc(M)} [M : Rv_1 + Rv_2] [M^\perp : Rv_3 + \dotsb + Rv_m]
\\& \leq (\gamma_{4em}^2/8e)^{4e} \abs{\disc(R)}^2 \abs{\disc(L)}^{2/m}
\\& \qquad \cdot (\gamma_e/8e)^{2e} (\gamma_{4em}^2/8e)^{2e} \eta^{28e} \abs{\disc(R)}^8 \abs{\disc(L)}^{1/m}
\\& \qquad \cdot \refC{so-index-mult}(2,e,m-2) \eta^{\refC{so-index-eta}(2,e,m-2)} (\gamma_{4em}^2/8e)^{4e\refC{so-index-L}(2,e,m-2)}
\\& \qquad \cdot \abs{\disc(R)}^{\refC{so-index-R}(2,e,m-2) + 2\refC{so-index-L}(2,e,m-2)} \abs{\disc(L)}^{(m+2)/m \cdot \refC{so-index-L}(2,e,m-2)}.
\end{align*}

Now we can calculate:
for the multiplicative constant:
\begin{align*}
  & \phantom{{} = {}} \refC{so-index-mult}(2,e,m-2) (\gamma_e/8e)^{2e} (\gamma_{4em}^2/8e)^{e(6+4\refC{so-index-L}(2,e,m-2))}
\\& = (2e(m-2)^2)^{e(m-2)m/2} (\gamma_e/8e)^{2e} \, (\gamma_{4em}^2/8e)^{2em}
\\& \leq (2em^2)^{e(m-2)m/2} \cdot 1 \cdot (2em^2)^{2em}
\\& = (2em^2)^{e(m^2-2m + 4m)/2}
\\& = \refC{so-index-mult}(2,e,m),
\end{align*}

for the exponent of $\eta$:
\[ 28e + \refC{so-index-eta}(2,e,m-2) = 28e + 14e(m-2) = \refC{so-index-eta}(2,e,m), \]
for the exponent of $\abs{\disc(R)}$:
\begin{align*}
  & \phantom{{} = {}} 2 + 8 + \refC{so-index-R}(2,e,m-2) + 2\refC{so-index-L}(2,e,m-2)
\\& = 10 + \frac{(m-2)(m+14)}{4} + (m-3)
    = \frac{m^2+16m}{4} = \refC{so-index-R}(2,e,m),
\end{align*}
for the exponent of $\abs{\disc(L)}$:
\begin{align*}
  & \phantom{{}={}} \frac{2}{m} + \frac{1}{m} + \frac{m+2}{m} \cdot \refC{so-index-L}(2,e,m-2)
\\& = \frac{3}{m} + \frac{(m+2)(m-3)}{2m}
\\& = \frac{6+(m^2-m-6)}{2m} = \frac{(m-1)m}{2m} = \refC{so-index-L}(2,e,m).
\end{align*}

\subsubsection*{Case~(c), part~(iv)}
For $i=j=1$ or $2$, \cref{pre-induction}(iv)(c) gives
\begin{equation} \label{eqn:case-c-iv-base}
\abs{\psi(v_i, v_i)}_D \leq 2^{-5/2} \gamma_e^{1/2} \gamma_{4em}^2 \eta^7 \abs{\disc(R)}^{2/e} \abs{\disc(L)}^{1/2em}.
\end{equation}
This establishes (iv) for $i=j=1$ or~$2$ because, using \cref{hermite-constant-bound} and the fact that $m \geq 2$ so $(m(m+1)+14)/8 \geq 5/2$,
\begin{align*}
    2^{-5/2} \gamma_e^{1/2} \gamma_{4em}^2
  & \leq 2^{-5/2} e^{1/2} (4em)^2
    = 2^{3/2} e^{5/2} m^2
    \leq (2em^2)^{5/2}
    \leq \refC{so-psi-mult}(2,e,m),
\\[3pt]
   7 & = \refC{so-psi-eta}(2,e,m),
\\ \frac{2}{e} & = \frac{2 \cdot 3 + 26}{16e} \leq \frac{m(m+1) + 26}{16e} = \refC{so-psi-R}(2,e,m),
\\ \frac{1}{2em} & \leq \frac{2}{8e} \leq \frac{m+1}{8e} \leq \refC{so-psi-L}(d,e,m).
\end{align*}

For $i=j \geq 3$, induction gives
\begin{align*}
    \abs{\psi(v_j, v_j)}_D
  & \leq \refC{so-psi-mult}(2,e,m-2) \eta^{\refC{so-psi-eta}(2,e,m-2)} \abs{\disc(R)}^{\refC{so-psi-R}(2,e,m-2)} \abs{\disc(M^\perp)}^{\refC{so-psi-L}(2,e,m-2)}
\\& \leq \refC{so-psi-mult}(2,e,m-2) \eta^{\refC{so-psi-eta}(2,e,m-2)} \abs{\disc(R)}^{\refC{so-psi-R}(2,e,m-2)}
\\& \qquad \cdot \bigl( (\gamma_{4em}^2/8e)^{4e} \abs{\disc(R)}^2 \abs{\disc(L)}^{(m+2)/m} \bigr)^{\refC{so-psi-L}(2,e,m-2)}.
\end{align*}
Now we can calculate:
for the multiplicative constant:
\begin{align*}
  & \phantom{{}={}} \refC{so-psi-mult}(2,e,m-2) (\gamma_{4em}^2/8e)^{4e\refC{so-psi-L}(2,e,m-2)}
\\& = (2e(m-2)^2)^{((m-2)(m-1)+14)/8} \cdot (\gamma_{4em}^2/8e)^{(m-1)/2}
\\& \leq (2em^2)^{(m^2-3m+16)/8} \cdot (2em^2)^{(4m-4)/8}
\\& = (2em^2)^{(m^2+m+12)/8}
\\& \leq \refC{so-psi-mult}(2,e,m),
\end{align*}
for the exponent of $\eta$:
\[ \refC{so-psi-eta}(2,e,m-2) = 7 = \refC{so-psi-eta}(2,e,m), \]
for the exponent of $\abs{\disc(R)}$:
\begin{align*}
  & \phantom{{}={}} \refC{so-psi-R}(2,e,m-2) + 2\refC{so-psi-L}(2,e,m-2)
\\& = \frac{(m-2)(m-1)+26}{16e} + 2 \cdot \frac{m-1}{8e}
\\& = \frac{m^2+m+24}{16e}
    \leq \frac{m(m+1)+26}{16e}
    = \refC{so-psi-R}(2,e,m),
\end{align*}
for the exponent of $\abs{\disc(L)}$:
\begin{align*}
    \frac{m+2}{m} \refC{so-psi-L}(2,e,m-2)
  & = \frac{m+2}{m} \cdot \frac{m-1}{8e}
\\& = \frac{m^2+m-2}{8em}
    \leq \frac{m(m+1)}{8em}
    = \refC{so-psi-L}(2,e,m).
\qedhere
\end{align*}
\end{proof}

\begin{lemma} \label{R-cap-Rdag}
Let $(D,\dag)$ be a division $\bQ$-algebra with a positive involution.
Let $V$ be a left $D$-vector space with a non-degenerate $(D,\dag)$-skew-Hermitian form $\psi \colon V \times V \to D$.
Let $L$ be a $\bZ$-lattice in $V$ such that $\Trd_{D/\bQ} \psi(L \times L) \subset \bZ$.
Let $R = \Stab_D(L)$.
Then $\disc(L) R^\dag \subset R$.
\end{lemma}

\begin{proof}
Let $a \in R$ and $x, y \in L$.
Then
\[ \Trd_{D/\bQ} \psi(a^\dag x, y) 
   = \Trd_{D/\bQ} \bigl( a^\dag \psi(x,y) \bigr)
   = \Trd_{D/\bQ} \bigl( \psi(x,y) a^\dag \bigr)
   = \Trd_{D/\bQ} \psi(x,ay). \]
Since $x, ay \in L$, we conclude that $\Trd_{D/\bQ} \psi(a^\dag x, y) \in \bZ$.

Since this holds for all $y \in L$, we have $a^\dag x \in L^*$.
Consequently, 
\[ \disc(L)a^\dag x = [L^*:L]a^\dag x \in L. \]
This holds for all $x \in L$, so $\disc(L)a^\dag \in \Stab_D(L) = R$.
\end{proof}

To complete the proof of \cref{minkowski-hermitian-perfect}, we combine \cref{weakly-unitary-induction} and \cref{R-cap-Rdag}.  The resulting exponent of $\abs{\disc(L)}$ in (iii) is $\refC{so-index-eta}(d,e,m)+\refC{so-index-L}(d,e,m)$ and the exponent of $\abs{\disc(L)}$ in (iv) is $\refC{so-psi-eta}(d,e,m)+\refC{so-psi-L}(d,e,m)$, while the other constants in \cref{minkowski-hermitian-perfect} are the same as the corresponding constants in \cref{weakly-unitary-induction}.

\section{Application to the Zilber--Pink conjecture} \label{sec:ZP-high-level}

In this section we study special subvarieties of PEL type from the point of view of Shimura data.
The main result of the section is that Shimura datum components of simple PEL type I and~II lie in a single $\gGSp_{2g}(\bR)$-conjugacy class, which we describe explicitly.  We also establish a bound on the dimension of all special subvarieties of PEL type in $\Ag$, demonstrating that \cref{main-theorem-zp} is indeed a consequence of the Zilber--Pink conjecture.  We end the section by outlining the strategy of the proof of \cref{main-theorem-zp} carried out in the subsequent sections.

For our notation and terminology around Shimura datum components, see \cite[sec.\ 2A and~2B]{ExCM}.

\subsection{Shimura data} \label{subsec:shimura-data}

Let $L=\ZZ^{2g}$, let $V = L_\bQ$ and let $\phi:L\times L\to\ZZ$ be the symplectic form represented, in the standard basis, by the matrix $J_{2g}$. Let $\gG=\gGSp(V, \phi)=\gGSp_{2g}$ and let $\Gamma=\gSp_{2g}(\ZZ)$.
Let $X^+$ denote the $\gG(\bR)^+$-conjugacy class of the morphism $h_0 \colon \mathbb{S}\to\gG_\RR$ given by
\begin{equation} \label{eqn:h0}
h_0(a+ib) \mapsto \fullsmallmatrix{a}{b}{-b}{a}^{\oplus g}.
\end{equation}
Then $(\gG, X^+)$ is a Shimura datum component and there is a $\gG(\bR)^+$-equivariant bijection $X^+ \cong \Hg$, where $\Hg$ is the Siegel upper half-space.
The moduli space of principally polarised abelian varieties of dimension~$g$, denoted $\Ag$, is the Shimura variety whose complex points are $\Gamma \bs X^+$.

Let $S$ be a special subvariety of PEL type of $\Ag$, as defined in section~\ref{subsec:intro-zp-context}, and let $R$ be its generic endomorphism ring.
Choose a point $x \in X^+$ whose image $s \in \Ag$ is an endomorphism generic point in $S(\bC)$.
Then $x$ induces an isomorphism $H_1(A_s, \bZ) \cong L$ and hence the action of $R$ on $A_s$ induces an action of $R$ on $L$.

Let $\gH$ denote the centraliser in $\gG$ of the action of $R$ on~$L$, which is a reductive $\bQ$-algebraic group.
We call $\gH$ the \defterm{general Lefschetz group} of $S$.
Note that $\gH$ is only defined up to conjugation by $\Gamma$, because different choices of $x$ may lead to isomorphisms $H_1(A_s, \bZ) \cong L$ which differ by $\Gamma$.
(The group $\gH$ is isomorphic to the Lefschetz group of an endomorphism generic abelian variety parameterised by $S$, as defined in \cite{Mil99}, thanks to \cite[Theorem~4.4]{Mil99}.
However it seems to be more common to call $\gH \cap \gSp$ or $(\gH \cap \gSp)^\circ$ the Lefschetz group, so we have added the adjective ``general'' by analogy with the general symplectic and general orthogonal groups.)

The special subvariety of PEL type $S$ is a Shimura subvariety component of $\Ag$ associated with a Shimura subdatum component of the form $(\gH^\circ, X_\gH^+) \subset (\gG, X^+)$, where $\gH$ is the general Lefschetz group of~$S$ (see \cite[paragraph above Theorem~8.17]{Mil05}).

We say that $(\gH, X_\gH^+) \subset (\gG, X^+)$ is a \defterm{Shimura subdatum component of simple PEL type I or~II} if it is a Shimura subdatum component associated with a special subvariety of PEL type, where $\gH$ is the general Lefschetz group, and its generic endomorphism algebra is a division algebra with positive involution of type I or~II.
Note that in the simple type I or~II case, $\gH = \gH^\circ$.

\subsection{Representatives of conjugacy classes of Shimura data of simple PEL type I or~II}
\label{subsec:shimura-representatives}

The Shimura subdatum components of $(\gG, X^+)$ of simple PEL type I or~II lie in only finitely many $\gG(\bR)^+$-conjugacy classes.
Indeed, we shall now explicitly describe finitely many Shimura subdatum components which represent these $\gG(\bR)^+$-conjugacy classes.
Note that, for convenience, these representative subdatum components are not of simple PEL type, although they are of PEL type.
This generalises \cite[Lemma~6.1]{QRTUI}, which is the case $g=2$, $d=2$, $e=m=1$.

Let $d$, $e$, $m$ be positive integers such that $d^2em = 2g$, $d=1$ or $2$ and $dm$ is even.
For fixed $g$, there are only finitely many integers $d,e,m$ satisfying these conditions.
As we shall show, each triple $d,e,m$ corresponds to a single $\gG(\bR)^+$-conjugacy class of Shimura subdatum components of simple PEL type I or~II.

Let $D_0 = \rM_d(\bQ)^e$.
Define a $\bQ$-algebra homomorphism $\iota_0 \colon D_0 \to \rM_{2g}(\bQ)$ as follows:
\begin{itemize}
\item when $d=1$:
$\iota_0(a_1, \dotsc, a_e) = a_1 I_m \oplus \dotsb \oplus a_e I_m$.
\item when $d=2$:
\[ \iota_0\Bigl( \fullmatrix{a_1}{b_1}{c_1}{d_1}, \dotsc, \fullmatrix{a_e}{b_e}{c_e}{d_e} \Bigr)
= \fullmatrix{a_1I_{2m}}{b_1I_{2m}}{c_1I_{2m}}{d_1I_{2m}} \oplus \dotsb \oplus \fullmatrix{a_eI_{2m}}{b_eI_{2m}}{c_eI_{2m}}{d_eI_{2m}}. \]
\end{itemize}
We view $V$ as a left $D_0$-module via $\iota_0$.

Let $t$ denote the involution of $D_0$ which is transpose on each factor.
Since $dm$ is even, $\iota_0(D_0)$ commutes with $J_{2g}$ and so, for all $a \in D_0$ and $x, y \in V$, we have
\[ \phi(ax, y) = x^t \iota(a)^t J_{2g} y = x^t J_{2g} \iota(a)^t y = \phi(x, a^t y). \]
Thus $\phi \colon V \times V \to \bQ$ is a $(D_0,t)$-compatible symplectic form.
By \cref{tr-skew-hermitian-form,tr-non-deg}, there is a unique non-degenerate $(D_0,t)$-skew-Hermitian form $\psi_0 \colon V \times V \to D_0$ such that $\phi = \Trd_{D_0/\bQ} \psi_0$.

Let $\gH_0$ denote the centraliser of $\iota_0(D_0)$ in $\gG$.
In other words,
\begin{equation} \label{eqn:H0}
\gH_0 =  \{ g_1^{\oplus d} \oplus g_2^{\oplus d} \oplus \dotsb \oplus g_e^{\oplus d} :
g_1, \dotsc, g_e \in \gGSp_{dm},
\, \nu(g_1) = \dotsb = \nu(g_e) \},
\end{equation}
where $\nu \colon \gGSp_{dm} \to \bG_m$ denotes the symplectic multiplier character.
This is a connected $\bQ$-algebraic group, and it is equal to the general Lefschetz group of a special subvariety of PEL type in which endomorphism generic points correspond to abelian varieties isogenous to a product of the form $A_1^d \times \dotsb \times A_e^d$ where $A_1, \dotsc, A_e$ are pairwise non-isogenous simple abelian varieties of dimension $dm/2$ with $\End(A_1) = \dotsb = \End(A_e) = \bZ$.

\pagebreak 

\begin{lemma} \label{conj-class-mt}
Let $(\gH, X_\gH^+) \subset (\gGSp_{2g}, \Hg)$ be a Shimura subdatum component of simple PEL type I or II.
Let $D$ be the generic endomorphism algebra of $(\gH, X_\gH^+)$ and let $F$ be the centre of $D$.
Then $\gH_\bR$ is a $\gG(\bR)^+$-conjugate of the group $\gH_0$ constructed above for the parameters
\[ d = \sqrt{\dim_F(D)} = 1 \text{ or } 2, \quad e = [F:\bQ], \quad m = 2g/d^2e. \]
\end{lemma}

\begin{proof}
The tautological family of principally polarised abelian varieties on $X^+$ restricts to a family of principally polarised abelian varieties on $X_\gH^+$.
The polarisation induces a Rosati involution $\dag$ of the endomorphism algebra of this family, namely~$D$.
As we saw in the construction of the general Lefschetz group, $D$ acts on $V$.
Via this action, the symplectic form $\phi \colon V \times V \to \bQ$ is $(D,\dag)$-compatible.

Since $(D,\dag)$ is a simple $\bQ$-algebra with a positive involution of type I or~II, there is an isomorphism $\alpha \colon (D_{0,\bR},t) \to (D_\bR,\dag)$ of $\bR$-algebras with involution (where $D_0 = \rM_d(\bQ)^e$ for the parameters $d$ and $e$ specified in the lemma).
We obtain an action of $D_{0,\bR}$ on $V_\bR$ by composing the action of $D_\bR$ with $\alpha$.

Since $\phi$ is $(D,\dag)$-compatible, it is also $(D_{0,\bR},t)$-compatible under the via~$\alpha$.
Hence there is a unique non-degenerate $(D_{0,\bR},t)$-skew-Hermitian form $\psi_\alpha \colon V_\bR \times V_\bR \to D_{0,\bR}$ such that $\phi = \Trd_{D_{0,\bR}/\bR} \psi_\alpha$, where ``$(D_{0,\bR},t)$-skew-Hermitian'' refers to the action via~$\alpha$.
(Note that $\psi_\alpha$ is in general different from $\psi_0$ because $\psi_0$ is $(D_{0,\bR},t)$-skew-Hermitian with respect to the action via $\iota_0$.)

By \cref{D0-basis}, there exists a $D_{0,\bR}$-basis $v_1, \dotsc, v_m$ for $V_\bR$ with respect to the action via~$\iota_0$ which is symplectic (if $d=1$) or unitary (if $d=2$) with respect to $\psi_0$.
There likewise exists a $D_{0,\bR}$-basis $w_1, \dotsc, w_m$ for $V_\bR$ with respect to the action via $\alpha$ which is symplectic or unitary with respect to $\psi_\alpha$.

Define $\gamma \in \gGL(V_\bR)$ by
\[ \gamma(\iota_0(a_1)v_1 + \dotsb + \iota_0(a_m)v_m) = \alpha(a_1)w_1 + \dotsb + \alpha(a_m)w_m \]
for all $a_1, \dotsc, a_m \in D_{0,\bR}$.
Because $v_1, \dotsc, v_m$ and $w_1, \dotsc, w_m$ are symplectic or unitary bases (depending on $d$) with respect to $\psi_0$ and $\psi_\alpha$ respectively, we have
\[ \psi_\alpha(\gamma(v_i), \gamma(v_j)) = \psi_\alpha(w_i, w_j) = \psi_0(v_j, v_j) \]
for all $i, j$.
Because $\psi_0$ and $\psi_\alpha$ are $(D_{0,\bR},t)$-skew-Hermitian with respect to the actions via $\iota_0$ and $\alpha$ respectively, it follows that
\[ \psi_\alpha(\gamma(v), \gamma(w)) = \psi_0(v, w) \]
for all $v, w \in V_\bR$.
Taking the reduced trace, we obtain $\phi(\gamma(v), \gamma(w)) = \phi(v, w)$ for all $v, w \in V_\bR$.
In other words, $\gamma \in \gSp(V_\bR, \phi) \subset \gG(\bR)^+$.

Since $\gamma$ is an isomorphism between the representations of $D_{0,\bR}$ given by $\alpha$ and~$\iota_0$, $\gamma \gH_0 \gamma^{-1}$ is the centraliser in $\gG$ of the action of $D_{0,\bR}$ via $\alpha$.
In other words, $\gamma \gH_0 \gamma^{-1}$ is the centraliser in $\gG$ of the action of $D_\bR$, which is the general Lefschetz group~$\gH$.
\end{proof}

\begin{lemma}\label{unique-datum}
For each triple of positive integers $d,e,m$ satisfying $d^2em=2g$, $d=1$ or $2$ and $dm$ even, there exists a unique Shimura subdatum component $(\gH_0,X^+_0)$ of $(\gG,X^+)$ with group~$\gH_0$. Furthermore, the Hodge parameter $h_0$ from \eqref{eqn:h0} is in~$X_0^+$.
\end{lemma}

\begin{proof}
First note that $h_0 \in X^+$ and $h_0$ factors through $\gH_{0,\bR}$.
Hence if $X_0^+$ denotes the $\gH_0(\bR)^+$-conjugacy class of $h_0$ in $\Hom(\bS, \gH_{0,\bR})$, then $(\gH_0, X_0^+)$ is a Shimura subdatum component of $(\gG, X^+)$.

To establish the uniqueness, let $X_0^+$ now denote any subset of $X^+$ such that $(\gH_0, X_0^+)$ is a Shimura datum component.
Let $\gH_0^\ad$ denote the quotient of $\gH_0$ by its centre.
By \cite[Proposition 5.7 (a)]{Mil05}, $X_0^+$ is in bijection with its image $(X^+_0)^{\ad}\subset\Hom(\mathbb{S},\gH^\ad_{0,\RR})$ under composition with the natural map $\gH_{0,\RR}\to\gH^{\ad}_{0,\RR}$.

Observe that $\gH^\ad_0\cong\gPGSp_{md}^e$. Therefore, $(X^+_0)^{\ad}$ is a product of $\gPGSp_{md}(\RR)^+$-conjugacy classes of morphisms $\mathbb{S}\to\gPGSp_{md,\RR}$ satisfying conditions SV1--SV3 from \cite[section~4]{Mil05}. From \cite[Prop 1.24]{Mil05}, and the following paragraphs, there is only one $\gPGSp_{md}(\RR)$-conjugacy class $X_{md}$ of morphisms $\mathbb{S}\to\gPGSp_{md,\RR}$ satisfying SV1--SV3. It has two connected components $X^+_{md}$ and $X^-_{md}$ corresponding to the connected components of $\gPGSp_{md}(\RR)$. In other words, $(X^+_0)^{\ad}$ is equal to a direct product of copies of the spaces $X^+_{md}$ and $X^-_{md}$. 

Consider the morphisms $h^+_2,h^-_2:\mathbb{S}\to\gGL_{2,\RR}$ defined by 
\[h^+_2:a+ib\mapsto\fullsmallmatrix{a}{b}{-b}{a}\text{ and } h^-_2:a+ib\mapsto\fullsmallmatrix{a}{-b}{b}{a}.\]
Then $(h^+_2)^{\oplus md/2}$ and $(h^-_2)^{\oplus md/2}$ are non-$\gGSp_{md}(\RR)^+$-conjugate  morphisms $\mathbb{S}\to\gGSp_{md,\RR}$ satisfying SV1--SV3. Therefore, the images of their $\gGSp_{md}(\RR)^+$-conjugacy classes in $\Hom(\mathbb{S},\gPGSp_{md,\RR})$ are precisely $X^+_{md}$ and $X^-_{md}$.
 
It follows that $(X^+_0)^{\ad}$ is the image in $\Hom(\mathbb{S},\gPGSp^e_{md,\RR})$ of the $\gGSp_{md}^e(\RR)^+$-conjugacy class of an element $h\in\Hom(\mathbb{S},\gGSp^e_{md,\RR})$ of the form
\[\bigl( (h^{\pm}_2)^{\oplus md/2}, \dotsc, (h^{\pm}_2)^{\oplus md/2} \bigr),\]
for some sequence of signs in $\{\pm\}^e$.
Since the image of $h$ in $\Hom(\mathbb{S},\gGSp_{md^2e,\RR})$ (obtained by repeating each component $d$ times block-diagonally) lies in a Shimura datum, it satisfies condition~SV2 of \cite{Mil05}, that is, the stabiliser of $h$ in $\gGSp_{md^2e}(\RR)$ is compact modulo the centre. This only holds when
\[ h=h^+=:((h^+_2)^{\oplus md/2}, \dotsc, (h^+_2)^{\oplus md/2}) \text{ or } h=h^-:=((h^-_2)^{\oplus md/2}, \dotsc, (h^-_2)^{\oplus md/2}).\]
This can be checked by observing that the centraliser of $h^+_2 \oplus h^-_2$ in $\gGSp_4(\RR)$ is non-compact modulo the centre.

Note that the image of $h^+$ in $\Hom(\mathbb{S},\gGSp_{md^2e,\RR})$ is equal to $h_0$.
Since $h_0 \in X^+$ while the image of $h^-$ is not in $X^+$, we conclude that $X_0^+$ must be equal to the $\gH_0(\RR)$-conjugacy class of $h_0$.
\end{proof}

\begin{corollary} \label{conj-class-datum}
If $(\gH, X_\gH^+) \subset (\gG, X^+)$ is a Shimura subdatum component of simple PEL type I or II and $\gH = g\gH_0 g^{-1}$ for $g \in \gG(\bR)^+$, then $X_\gH^+ = gX_0^+$ where $(\gH_0, X_0^+) \subset (\gG, X^+)$ is the unique Shimura subdatum component given by \cref{unique-datum}.
\end{corollary}

\subsection{Dimension of special subvarieties of PEL type}\label{sec:dims}

In this section we prove \cref{codim-pel-intro}: \cref{codim-pel-simple} is \cref{codim-pel-intro}(i), while \cref{codim-pel} is \cref{codim-pel-intro}(ii).

\begin{proposition}\label{codim-pel-simple}
Let $S \subset \Ag$ be a special subvariety, not equal to~$\Ag$.
If $S$ is of simple PEL type, then $\dim(S) \leq \dim(\Ag) - g^2/4$.
\end{proposition}

\begin{proof}
Let $D$ be the generic endomorphism algebra of $S$.  Following our usual notation, let $F$ be the centre of $D$ and let
\[ d = \sqrt{\dim_F(D)}, \quad e = [F:\bQ], \quad m = 2g/d^2e. \]

When $D$ has Albert type~IV, we need some additional notation.
Let $s \in S(\bC)$.
Then $D_\bR \cong \rM_d(\bC)^{e/2}$ acts $\bR$-linearly on the tangent space $T_0(A_s(\bC))$.
For each $i = 1, \dotsc, e/2$, let $r_i$ denote the multiplicity in $T_0(A_s(\bC))$ of the standard representation of the $i$-th factor $\rM_d(\bC)$ of $D_\bR$.
Similarly let $s_i$ denote the multiplicity of the complex conjugate of the standard representation of the $i$-th factor of $D_\bR$.
The values $r_i$ and $s_i$ are independent of the choice of $s \in S(\bC)$, and satisfy $r_i + s_i = dm$.

The dimension of special subvarieties of simple PEL type was determined by Shimura \cite[4.1]{Shi63}.
Note that our $m$ is the same as $m$ in \cite{Shi63}, while our $e$ is called $g$ in \cite{Shi63}.
For a more modern account of this theory, see \cite[chapter~9]{BL04}.
For each type of endomorphism algebra~$D$, we quote the dimension of the special subvariety from \cite[4.1]{Shi63} and use some elementary inequalities.

When $D$ has type~I, $d=1$, $em = 2g$ and $e \geq 2$ since $S \neq \Ag$, so $m \leq g$.
Hence
\[ \dim(S) = \tfrac{1}{2} \tfrac{m}{2} \bigl( \tfrac{m}{2} + 1 \bigr)e
\leq \tfrac{1}{2} g \bigl( \tfrac{1}{2}g + 1 \bigr) = \tfrac{1}{4}g^2 + \tfrac{1}{2}g. \]

When $D$ has type~II, $d=2$, $em = g/2$ and $m \leq g/2$ so
\[ \dim(S) = \tfrac{1}{2} m(m+1)e
\leq \tfrac{1}{4}g \bigl( \tfrac{1}{2}g + 1 \bigr) = \tfrac{1}{8}g^2 + \tfrac{1}{4}g. \]

When $D$ has type~III, $d=2$, $em = g/2$ and $m \leq g/2$ so
\[ \dim(S) = \tfrac{1}{2} m(m-1)e
\leq \tfrac{1}{4}g \bigl( \tfrac{1}{2}g - 1 \bigr) = \tfrac{1}{8}g^2 - \tfrac{1}{4}g. \]

When $D$ has type~IV, $2g = d^2em$ and $e \geq 2$ since $F$ is a CM field, so $m \leq g$.
Furthermore $r_i + s_i = dm$ so $r_is_i \leq d^2m^2/4$ for each~$i$.
Hence,
\[ \dim(S) = \sum_{i=1}^{e/2} r_i s_i \leq \tfrac{1}{2}e \cdot \tfrac{1}{4}d^2m^2 = \tfrac{1}{4}gm \leq \tfrac{1}{4}g^2. \]

Hence in all cases,
\[ \dim(S) \leq \tfrac{1}{4}g^2 + \tfrac{1}{2}g = \tfrac{1}{2}g(g+1) - \tfrac{1}{4}g^2 = \dim(\Ag) - \tfrac{1}{4}g^2.
\qedhere
\]
\end{proof}

\pagebreak 

\begin{proposition}\label{codim-pel}
Let $S \subset \Ag$ be a special subvariety, not equal to~$\Ag$.
If $S$ is of PEL type, then $\dim(S) \leq \dim(\Ag) - g + 1$.
\end{proposition}

\begin{proof}
Note that $g^2/4 \geq g-1$ for all real numbers~$g$, so \cref{codim-pel-simple} implies the claim for special subvarieties of simple PEL type.

Let $S \subset \Ag$ be a special subvariety of non-simple PEL type.
By adding level structure, we may obtain a finite cover $S' \to S$ which is a fine moduli space of abelian varieties with PEL structure.
Then there is a universal abelian scheme with PEL structure $A \to S'$.
Since $S'$ is of non-simple PEL type, $A$ is a non-simple abelian scheme.
Thus there exist non-trivial abelian schemes $A_1, A_2 \to S'$ such that $A$ is isogenous to $A_1 \times A_2$. (There may be multiple choices of isogeny decompositions of $A$.  Choose any such decomposition.)
Let $g_1$, $g_2$ denote the relative dimensions of $A_1$ and $A_2$ respectively.

Let
\[ T = \{ (s, s_1, s_2) \in S' \times \cA_{g_1} \times \cA_{g_2} : A_s \text{ is isogenous to } A_{s_1} \times A_{s_2} \}. \]
Since isogenies $A_s \to A_{s_1} \times A_{s_2}$ give rise to Hodge classes on $A_s \times A_{s_1} \times A_{s_2}$, the locus~$T$ is a countable union of special subvarieties of $S' \times \cA_{g_1} \times \cA_{g_2}$.

By construction, the projection $T \to S'$ is surjective on $\bC$-points.
An irreducible complex algebraic variety cannot be contained in the union of countably many proper closed subvarieties.
Hence there exists an irreducible component $T^+ \subset T$ such that the image of $T^+$ is dense in $S'$.
Hence $\dim(T^+) \geq \dim(S') = \dim(S)$.

Given any two abelian varieties $A_{s_1}$ and $A_{s_2}$ over $\bC$, there are only countably many isomorphism classes of abelian varieties which are isogenous to $A_{s_1} \times A_{s_2}$.
Furthermore each abelian variety of dimension~$g$ carries only finitely many PEL structures parameterised by $S'$ (the natural morphism $S' \to \Ag$ is finite).
Hence the projection $T \to \cA_{g_1} \times \cA_{g_2}$ has countable fibres.
Therefore
\[ \dim(T^+) \leq \dim(\cA_{g_1} \times \cA_{g_2}) = \frac{g_1(g_1+1)}{2} + \frac{g_2(g_2+1)}{2}. \]

Since $g_1+g_2=g$, we obtain
\[ \tfrac{1}{2}g_1(g_1+1) + \tfrac{1}{2}g_2(g_2+1)
= \tfrac{1}{2}\bigl( (g_1+g_2)^2 - 2g_1g_2 + (g_1+g_2) \bigr)
= \tfrac{1}{2}(g^2+g) - g_1g_2. \]
Therefore
\[ \dim(S) \leq \dim(T^+) \leq \dim(\Ag) - g_1g_2. \]
Now $g_1g_2 = g_1(g-g_1)$ is a quadratic function of $g_1$ with a maximum at $g_1=g/2$.
Hence, for $1 \leq g_1 \leq g-1$, $g_1g_2$ is minimised when $g_1 = 1$ or $g-1$.
Thus $g_1g_2 \geq g-1$.
\end{proof}

\section{Construction of representation and closed orbit} \label{sec:representation}

This section constructs the representation required for the strategy outlined in section~\ref{subsec:parameter-space} and proves that it satisfies conditions (i) and~(ii) of \cite[Theorem~1.2]{QRTUI}.
These conditions are algebraic and geometric in nature.
We also prove a small piece of arithmetic information about the representation, namely \cref{reps-closed-orbits}(v), which will be used to obtain more substantial arithmetic properties in section~\ref{sec:rep-bound}.
This section generalises \cite[sections 5.2 and~5.3]{QRTUI}.

We will actually construct two representations $\rho_L, \rho_R \colon \gG \to \gGL(W)$, which are induced by left and right multiplication respectively in $\End(V)$.
The representation to which we shall apply \cite[Theorem~1.2]{QRTUI} is $\rho_L$, while $\rho_R$ is an auxiliary object required at the end of section~\ref{sec:rep-bound}.

\begin{proposition} \label{reps-closed-orbits}
Let $d$, $e$ and $m$ be positive integers such that $dm$ is even.
Let $n = d^2em$.
Let $L = \bZ^n$ and let $\phi \colon L \times L \to \bZ$ be the standard symplectic form as in section~\ref{subsec:shimura-data}.
Let $V = L_\bQ$, let $\gG = \gGSp(V, \phi) = \gGSp_{n,\bQ}$ and let $\Gamma = \gSp_n(\bZ)$.

Let $E_0$ be a $\bQ$-subalgebra of $\End(V) = \rM_n(\bQ)$ such that $E_{0,\bC} \cong \rM_{dm}(\bC)^e$ and the resulting $E_{0,\bC}$-module structure on $V_\bC$ is isomorphic to the direct sum of $d$ copies of each of the $e$ irreducible representations of $E_{0,\bC}$.
Let $\gH_0$ be the $\bQ$-algebraic subgroup of $\gG$ whose $k$-points are
\[ \gH_0(k) = (E_0 \otimes_\bQ k) \cap \gG(k) \]
for each field extension $k$ of~$\bQ$.

Then there exists a $\bQ$-vector space $W$, a $\bZ$-lattice $\Lambda \subset W$, $\bQ$-algebraic representations $\rho_L, \rho_R \colon \gG \to \gGL(W)$, a vector $w_0 \in \Lambda$ and a constant $\newC{du-multiplier}$ such that:
\begin{enumerate}[(i)]
\item $\Stab_{\gG,\rho_L}(w_0) = \Stab_{\gG,\rho_R}(w_0) = \gH_0$;

\item the orbit $\rho_L(\gG(\bR))w_0$ is closed in $W_\bR$;

\item $\rho_L$ and $\rho_R$ commute with each other;

\item $\rho_L(\Gamma)$ and $\rho_R(\Gamma)$ stabilise $\Lambda$;

\item for each $u \in \gG(\bR)$, if the group $\gH_u = u \gH_{0,\bR} u^{-1}$ is defined over~$\bQ$, then there exists $d_u \in \bR_{>0}$ such that
\[ d_u \rho_L(u) \rho_R(u) w_0 \in \Lambda  \quad  \text{and}  \quad  d_u \leq \refC{du-multiplier} \abs{\disc(S_u)}^{1/2} \]
where $S_u$ denotes the ring $u E_{0,\bR} u^{-1} \cap \rM_n(\bZ)$.
\end{enumerate}
\end{proposition}

In our application to \cref{main-theorem-zp}, $\gH_0$ shall be equal to the group $\gH_0$ defined in~\eqref{eqn:H0}.
To achieve this, let $d=1$ or~$2$ and define $D_0$ and $\iota_0 \colon D_0 \to \rM_n(\bQ)$ as in section~\ref{subsec:shimura-representatives} (with $n=2g$).
Let $E_0$ be the centraliser of $\iota_0(D_0)$ in $\rM_n(\bQ)$, that is,
\begin{equation} \label{eqn:E_0}
E_0 = \{ f_1^{\oplus d} \oplus f_2^{\oplus d} \oplus \dotsb \oplus f_e^{\oplus d} \in \rM_n(\bQ) :
f_1, \dotsc, f_e \in \rM_{dm}(\bQ) \}.
\end{equation}
It is immediate that intersecting this algebra $E_0$ with $\gG$ yields the same group $\gH_0$ as in \eqref{eqn:H0}.
Furthermore, the map $(f_1, \dotsc, f_e) \mapsto f_1^{\oplus d} \oplus f_2^{\oplus d} \oplus \dotsb \oplus f_e^{\oplus d}$ is an isomorphism of $\bQ$-algebras $\rM_{dm}(\bQ)^e \to E_0$.
By decomposing $V$ as a direct sum of $dm$-dimensional subspaces, matching the block diagonal decomposition of elements of $E_0$, we see that $V$ is isomorphic to the sum of $d$ copies of each of the $e$ irreducible representations of $E_0$.
After extending scalars to~$\bC$, we conclude that $E_0$ as defined by \eqref{eqn:E_0} satisfies the conditions of \cref{reps-closed-orbits}.

Allowing more general choices of $E_0$ in \cref{reps-closed-orbits} than simply \eqref{eqn:E_0}, and only imposing conditions on $E_0$ after extending scalars to~$\bC$, ensures that the proposition could be used as part of a similar strategy for proving the Zilber--Pink conjecture for special subvarieties of simple PEL type III and~IV, as well as types I and~II.

\subsection{Construction of the representation}

Let $\sigma_L, \sigma_R \colon \gG \to \gGL(\End(V))$ denote the left and right multiplication representations of~$\gG$:
\begin{equation*}
\sigma_L(g)f = gf, \quad \sigma_R(g)f = fg^{-1}.
\end{equation*}
Note that $\sigma_R(g)f = fg^{-1}$ rather than $fg$ so that $\sigma_R$ is a group representation.
The representations $\rho_L$ and $\rho_R$ in \cref{reps-closed-orbits} are induced by $\sigma_L$ and $\sigma_R$ via a linear algebra construction which we shall now explain, and hence one may think of $\rho_L(u)\rho_R(u)$ in \cref{reps-closed-orbits}(v) as being induced by conjugation by $u \in \gG(\bR)$.

Let $\nu \colon \gG = \gGSp_n \to \bG_m$ denote the symplectic multiplier character.
Let $W = \extpower^{mn} \End(V)$, which is a $\bQ$-vector space of dimension $\binom{n^2}{mn}$.
The representations required by \cref{reps-closed-orbits} are defined as
\[ \rho_L = \extpower^{mn} \sigma_L \otimes \nu^{-mn/2},  \;  \rho_R = \extpower^{mn} \sigma_R \otimes \nu^{mn/2} \; \colon \gG \to \gGL(W). \]
The powers of $\nu$ are chosen so that both $\rho_L$ and $\rho_R$ restrict to the trivial representation on the scalars $\bG_m \subset \gGSp_n$.

Next we construct a vector $w_0 \in W$ satisfying \cref{reps-closed-orbits}(i).
Observe that $\dim_\bQ(E_0) = e(dm)^2 = mn$ so $\extpower^{mn} E_0$ is a 1-dimensional subspace of $W$.
This was the reason we used the $mn$-th exterior power to define $W$.

Because $E_0$ is a subring of $\End(V)$, 
for any field extension $k$ of $\bQ$,
\[ \Stab_{\gG(k),\sigma_L}(E_0) = \gG(k) \cap (E_0 \otimes_\bQ k) = \gH_0(k). \]
Similarly $\Stab_{\gG(k),\sigma_R}(E_0) = \gH_{d,e,m}(k)$.
Consequently,
\[ \Stab_{\gG,\rho_L}(\extpower^{mn} E_0) = \Stab_{\gG,\rho_R}(\extpower^{mn} E_0) = \gH_0. \]

The action of $E_0$ on $\extpower^{mn} E_0$ via the $mn$-th exterior power of the left regular representation is multiplication by the non-reduced norm $\Nm_{E_0/\bQ}$.
Choose an isomorphism $\eta \colon E_{0,\bC} \to \rM_{dm}(\bC)$.
Let $f \in E_{0,\bC}$ and $\eta(f) = (f_1, \dotsc, f_e) \in \rM_{dm}(\bC)^e$.
Since the irreducible representations of $E_{0,\bC}$ are projections onto the simple factors of $\rM_{dm}(\bC)^e$, and each irreducible representation appears $d$ times in $V_\bC$, we have
\[ \det(f) = \prod_{i=1}^e \det(f_i)^d. \]
Hence
\[ \Nm_{E_{0,\bC}/\bC}(f) = \prod_{i=1}^e \Nm_{\rM_{dm}(\bC)/\bC}(f_i) = \prod_{i=1}^e \det(f_i)^{dm} = \det(f)^m. \]
If $f \in \gH_0(\bQ) \subset \gG(\bQ)$, then $\det(f) = \nu(f)^{n/2}$ so
\[ \Nm_{E_0/\bQ}(f) = \nu(f)^{mn/2}. \]
Hence the action of $\gH_0$ on $\extpower^{mn} E_0$ via $\rho_L$ is multiplication by $\Nm_{E_0/\bQ} \otimes \nu^{-mn/2} = 1$.
Thus for any $w \in \extpower^{mn} E_0$, we have $\rho_L(\gH_0)w = w$, while 
\[ \Stab_{\gG,\rho_L}(w) \subset \Stab_{\gG,\rho_L}(\extpower^{mn} E_0) = \gH_0. \]
Thus $\Stab_{\gG,\rho_L}(w) = \gH_0$.

For similar reasons, the action of $\gH_0$ on $\extpower^{mn} E_0$ via $\extpower^{mn} \sigma_R$ is multiplication by $\Nm_{E_0/\bQ}^{-1}$, and hence the action of $\gH_0$ on $\extpower^{mn} E_0$ via $\rho_R$ is trivial.
It follows that for any $w \in \extpower^{mn} E_0$, $\Stab_{\gG,\rho_R}(w) = \gH_0$.

Let $\Lambda = \extpower^{mn} \rM_n(\bZ)$, which is a $\bZ$-lattice in $W$.
Let $S_0 = E_0 \cap \rM_n(\bZ)$, which is an order in $E_0$.
Then $\extpower^{mn} S_0$ is a free $\bZ$-module of rank 1 contained in $\Lambda$.
Choose $w_0$ to be a generator of $\extpower^{mn} S_0$ (it does not matter which generator we choose).

Since $w_0 \in \extpower^{mn} E_0$, the argument above shows that $w_0$ satisfies \cref{reps-closed-orbits}(i).
It is clear that $\rho_L$ and $\rho_R$ commute, so \cref{reps-closed-orbits}(iii) holds.
It is also immediate that \cref{reps-closed-orbits}(iv) holds.
Most of this section will be devoted to proving \cref{reps-closed-orbits}(ii).
Since the proof of \ref{reps-closed-orbits}(v) is short, let us first include it here.

\begin{proof}[Proof of \cref{reps-closed-orbits}(v)]
By definition,
\[ \rho_L(u) \rho_R(u) = \extpower^{mn} \sigma_L(u)\sigma_R(u) \in \gGL(\extpower^{mn} \End(V)), \]
where $\sigma_L(u)\sigma_R(u) \in \gGL(\End(V))$ is conjugation by $u$.
Hence $\rho_L(u)\rho_R(u)w_0$ is a generator of the $\bZ$-module $\bigwedge^{mn} uS_0u^{-1}$.

Let $d_u = \covol(S_u)/\covol(uS_0u^{-1})$ with respect to the volume form induced by the non-reduced trace form on $S_{u,\bR}$.
Then $d_u\rho_L(u)\rho_R(u)w_0$ is a generator for $\bigwedge^{mn} S_u$ and therefore is in $\Lambda$.

Conjugation by $u$ pulls back $\Tr_{S_{u,\bR}/\bR}$ to $\Tr_{S_{0,\bR}/\bR}$.
Hence
\[ d_u = \covol(S, \Tr_{S_{u,\bR}/\bR})/\covol(S_0, \Tr_{S_{0,\bR}/\bR}) = \sqrt{\abs{\disc(S_u)}/\abs{\disc(S_0)}}.
\qedhere \]
\end{proof}

\subsection{Proof of closed orbit}

According to \cite[Prop.~6.3]{BHC62}, in order to show that $\rho_L(\gG(\bR))w_0$ is closed in $W_\bR$ (in the real topology), it suffices to prove that $\rho_L(\gG(\bC))w_0$ is Zariski closed in $W_\bC$.
Therefore, for the rest of this section, we shall deal entirely with linear algebra and algebraic geometry over $\bC$.

Let
\[ Q = \{ g \in \End(V_\bC) : \exists s \in \bC \text{ s.t.\ for all } v, v' \in V_\bC, \phi(gv, gv') = s \phi(v,v') \}. \]
Note that $Q$ is equal to the union of $\gG(\bC)$ with the set of elements of $\End(V_\bC)$ whose image is contained in a $\phi$-isotropic subspace of $V_\bC$.
In particular
\[ \gG(\bC) = \{ g \in Q : \det(g) \neq 0 \}. \]

Let $e_1, \dotsc, e_n$ be a symplectic basis for $(V_\bC, \phi)$.
Then $Q$ is a Zariski closed subset of $\End(V_\bC)$ because it is defined by the polynomial equations
\begin{gather*}
   \phi(ge_i, ge_j) = 0 \text{ for each } i, j \text{ except when } \{ i,j \} = \{ 2k-1,2k \} \text{ for some } k,
\\ \phi(ge_1, ge_2) = \phi(ge_3, ge_4) = \dotsb = \phi(ge_{n-1}, ge_n).
\end{gather*}
Furthermore, $Q$ is a homogeneous subset of $\End(V_\bC)$, that is, it is closed under multiplication by scalars.

Consequently, for any map from $\End(V_\bC)$ to another vector space whose coordinates are given by homogeneous polynomials of the same positive degree, the image of $Q$ is homogeneous and Zariski closed.
(This is because such a map induces a morphism of varieties between the associated projective spaces, and the image of the projective algebraic set $(Q \setminus \{0\})/\bG_m$ under such a morphism will again be a projective algebraic set.)

Note that $\sigma_L : \gG(\bC) \to \gGL(\End(V_\bC))$ extends to a $\bC$-algebra homomorphism $\End(V_\bC) \to \End(\End(V_\bC)) \cong \rM_{n^2}(\bC)$ defined by the formula $\sigma_L(g)f = gf$.  Considering $\sigma_L$ as a representation of the multiplicative monoid $\End(V_\bC)$, it induces a monoid representation
\[ \extpower^{mn} \sigma_L \colon\End(V_\bC) \to \End(\extpower^{mn} \End(V_\bC)). \]
Here $\extpower^{mn} \sigma_L$ is a homogeneous morphism of degree $mn$, so the set $(\extpower^{mn} \sigma_L)(Q)w_0$ is a homogeneous Zariski closed subset of $W_\bC$.

\begin{lemma} \label{exist-u}
There exist vectors $u_1, \dotsc, u_m \in V$ such that the map $\delta \colon \End(V) \to V^m$ defined by $\delta(f) = (f(u_1), \dotsc, f(u_m))$ restricts to an isomorphism of $\bQ$-vector spaces $E_0 \to V^m$.
\end{lemma}

\begin{proof}
By the hypothesis of \cref{reps-closed-orbits}, we can decompose $V_\bC$ as a direct sum of irreducible $E_{0,\bC}$-modules
\begin{equation} \label{eqn:VC-decomposition}
V_\bC = \bigoplus_{i=1}^e \bigoplus_{j=1}^d V_{ij}
\end{equation}
such that the action of $E_{0,\bC} \cong \rM_{dm}(\bC)^e$ on $V_{ij}$ factors through the $i$-th copy of $\rM_{dm}(\bC)$.
Since $\rM_{dm}(\bC)$ is a simple algebra, it has a unique irreducible representation (up to isomorphism), so we may choose an isomorphism of $\rM_{dm}(\bC)$-modules $\theta_{ij} \colon \bC^{dm} \to V_{ij}$.

Label the standard basis of $\bC^{dm}$ as $e_{k\ell}$ for $1 \leq k \leq d$, $1 \leq \ell \leq m$.

Given $f \in E_{0,\bC}$, write $\eta(f) = (f_1, \dotsc, f_e) \in \rM_{dm}(\bC)^e$.
For $i = 1, \dotsc, e$, $k = 1, \dotsc, d$ and $\ell = 1, \dotsc, m$, let $f_{i,k\ell} \in \bC^{dm}$ denote the column of $f_i$ indexed by $k$ and~$\ell$ (ordered to match the basis vectors $w_{k\ell}$).
The action of $E_{0,\bC}$ on $V_{ij}$ factors through the $i$-th copy of $\rM_{dm}(\bC)$ and $\theta_{ij}$ is an $\rM_{dm}(\bC)$-module homomorphism, so
\[ f(\theta_{ij}(e_{k\ell})) = \theta_{ij}(f_i(e_{k\ell})) = \theta_{ij}(f_{i,k\ell}). \]

For $\ell = 1, \dotsc, m$, let
\[ u_\ell = \sum_{i=1}^e \sum_{j=1}^d \theta_{ij}(e_{j\ell}) \in V_\bC. \]
(Note that the index $j$ is used twice in this expression.)
Then
\begin{equation} \label{eqn:f-ul}
f(u_\ell) = \sum_{i=1}^e \sum_{j=1}^d \theta_{ij}(f_{i,j\ell}).
\end{equation}

If $f \in \ker(\delta) \cap E_{0,\bC}$, then $f(u_\ell) = 0$ for $\ell = 1, \dotsc, m$.
Since \eqref{eqn:VC-decomposition} is a direct sum and the~$\theta_{ij}$ are injective, it follows from \eqref{eqn:f-ul} that $f_{i,j\ell} = 0$ for all $i$, $j$ and $\ell$.  In other words $f=0$.

Thus $\delta|_{E_{0,\bC}}$ is injective.  In particular $\delta|_{E_0}$ is injective.
Since $\dim_\bQ(E_0) = \dim_\bC(E_{0,\bC}) = ed^2m^2 = \dim_\bQ(V^m)$ and $\delta$ is a linear map, it follows that $\delta|_{E_0}$ is an isomorphism $E_0 \to V^m$.
\end{proof}

\begin{lemma}
There exists a linear function $\zeta \colon W \to \bQ$ such that $\zeta(w_0) \neq 0$ and
\[ \zeta \bigl( (\extpower^{mn} \sigma_L)(g)w \bigr) = \det(g)^m \zeta(w) \]
for all $g \in \End(V)$ and all $w \in W$.
\end{lemma}

\begin{proof}
Define $\zeta$ to be the linear map on $mn$-th exterior powers induced by $\delta$ from \cref{exist-u}.
Then $\zeta$ is a linear map $W = \extpower^{mn} \End(V) \to \extpower^{mn} V^m \cong \bQ$.
We identify $\extpower^{mn} V^m$ with $\bQ$ (the choice of isomorphism $\extpower^{mn} V^m \cong \bQ$ is not important).

Since $\delta|_{E_0}$ is an isomorphism $E_0 \to V^m$ and $w_0$ is a generator of $\extpower^{mn} E_0$, we deduce that $s(w_0)$ is a generator of $\extpower^{mn} V^m$.
In particular $\zeta(w_0) \neq 0$.

Let $\tau_L \colon \End(V) \to \End(V^m)$ denote the direct sum of $m$ copies of the tautological representation of $\End(V)$ on $V$.
Then
\[ \delta(\sigma_L(g)f) = \tau_L(g)\delta(f) \]
for all $f, g \in \End(V)$.
Taking the $mn$-th exterior power, we get
\[ \zeta \bigl( (\extpower^{mn} \sigma_L)(g)w \bigr) = \det(\tau_L(g)) \zeta(w) = \det(g)^m \zeta(w) \]
for all $g \in \End(V)$ and $w \in W$.
\end{proof}

\begin{lemma}
$\rho_L(\gG(\bC))w_0 = \{ w \in (\extpower^{mn} \sigma_L)(Q)w_0 : \zeta(w) = \zeta(w_0) \}$.
\end{lemma}

\begin{proof}
If $g \in \gG(\bC)$, then we can write $g = s g'$ where $s \in \bC^\times$ and $g' \in \gSp_n(\bC)$.  (Choose $s$ to be a square root of $\nu(g)$.)
Then $g' \in Q$, $\rho_L(g) = (\extpower^{mn} \sigma_L)(g')$ and
\[ \zeta(\rho_L(g)w_0) = \det((\extpower^{mn} \sigma_L)(g'))^m \zeta(w_0) = \zeta(w_0) \]
so $\rho_L(g)w_0$ is in $\{ w \in (\extpower^{mn} \sigma_L)(Q)w_0 : \zeta(w) = \zeta(w_0) \}$.

Conversely, if $w = (\extpower^{mn} \sigma_L)(g)w_0$ for some $g \in Q$ and $\zeta(w) = \zeta(w_0)$, then
\[ \det(g)^m \zeta(w_0) = \zeta \bigl( (\extpower^{mn} \sigma_L)(g)w_0 \bigr) = \zeta(w) = \zeta(w_0). \]
Since $\zeta(w_0) \neq 0$, we deduce that $\det(g)^m = 1$.
In particular $\det(g) \neq 0$.
Together with $g \in Q$, this implies that $g \in \gGSp_n(\bC)$.
Furthermore,
\[ \rho_L(g) = (\extpower^{mn} \sigma_L)(g) \otimes \det(g)^m = (\extpower^{mn} \sigma_L)(g) \]
so $\rho_L(g)w_0 = w$.
Thus $w \in \rho_L(\gG(\bC))w_0$.
\end{proof}

Thus $\rho_L(\gG(\bC))w_0$ is Zariski closed in $W_\bC$, so by \cite[Prop.~6.3]{BHC62} $\rho_L(\gG(\bR))w_0$ is closed in $W_\bR$ in the real topology.

\section{Arithmetic bound for the representation} \label{sec:rep-bound}

In this section, we bound the lengths of the vectors $v_u$ of \cite[Theorem~1.2]{QRTUI} (here renamed~$w_u$), when applied to the representation~$\rho_L$ defined in section~\ref{sec:representation}.  This bound is arithmetic in nature, being in terms of discriminants of orders in $\bQ$-division algebras.  The argument generalises \cite[section~5.5]{QRTUI} and  \cref{minkowski-hermitian-perfect} plays the role of \cite[Lemma~5.7]{QRTUI}.

\begin{proposition} \label{rep-bound-arithmetic}
Let $d$, $e$ and $m$ be positive integers such that $dm$ is even.
Let $n = d^2em$.
Let $L = \bZ^n$ and let $\phi \colon L \times L \to \bZ$ be the standard symplectic form as in section~\ref{subsec:shimura-data}.
Let $\gG = \gGSp(L_\bQ, \phi) = \gGSp_{n,\bQ}$ and let $\Gamma = \gSp_n(\bZ)$.
Let $\gH_0$ be the subgroup of $\gG$ defined in~\eqref{eqn:H0}.
Let $W$, $\Lambda \subset W$, $\rho_L, \rho_R \colon \gG \to \gGL(W)$ and $w_0 \in \Lambda$ be as in \cref{reps-closed-orbits}.

Then there exist positive constants $\newC{rep-multiplier}$, $\newC{rep-exponent}$, $\newC{guh-multiplier}$ and $\newC{guh-exponent}$ such that,
for each $u \in \gG(\bR)$, if the group $\gH_u = u \gH_{0,\bR} u^{-1}$ is defined over~$\bQ$ and $L_\bQ$ is irreducible as a representation of $\gH_u$ over $\bQ$, then
\begin{enumerate}[(a)]
\item there exists $w_u \in \Aut_{\rho_L(\gG)}(\Lambda_\bR) w_0$ such that $\rho_L(u) w_u \in \Lambda$ and
\[ \abs{w_u} \leq \refC{rep-multiplier} \abs{\disc(R_u)}^{\refC{rep-exponent}}; \]
\item there exists $\gamma \in \Gamma$ and $h \in \gH_0(\bR)$ such that
\[ \length{\gamma uh} \leq  \refC{guh-multiplier} \abs{\disc(R_u)}^{\refC{guh-exponent}}, \]
\end{enumerate}
where $R_u$ denotes the ring $\End_{\gH_u}(L) \subset \rM_n(\bZ)$.
\end{proposition}

Note that $L_\bQ$ is irreducible as a representation of $\gH_u$ if and only if $R_{u,\bQ}$ is a division algebra.
Because $R_{u,\bR}$ is $\gG(\bR)$-conjugate to $\End_{\gH_0}(L_\bR)$, $R_{u,\bQ}$ is an $\bR$-split algebra with positive involution.  Hence whenever $R_{u,\bQ}$ is a division algebra, it must be of type I or~II in the Albert classification, and $d$ must equal $1$ or~$2$ for \cref{rep-bound-arithmetic} to be non-vacuous.

\medskip

Let $V = L_\bQ = \bQ^{2g}$.
Define $D_0 = \rM_d(\bQ)^e$, $\iota_0 \colon D_0 \to \rM_n(\bQ)$, $t \colon D_0 \to D_0$ and $\psi_0 \colon V \times V \to D_0$ as in \cref{subsec:shimura-representatives}.

By \cref{D0-basis}, we can choose a $D_0$-basis $w_1, \dotsc, w_m$ for $V$ which is either symplectic or unitary depending on the type of~$D_0$.

Define a symmetric $\bQ$-bilinear form $\sigma_0 \colon V \times V \to \bQ$ by
\[ \sigma_0(x_1w_1 + \dotsb + x_mw_m, y_1w_1 + \dotsb + y_mw_m)) = \Trd_{D_0/\bQ} \left(\sum_{i=1}^m x_i y_i^t\right) \]
for all $x_1, \dotsc, x_m, y_1, \dotsc, y_m \in D_0$.
This bilinear form is positive definite because $t$ is a positive involution.
In fact, a lengthy calculation shows that $\sigma_0$ is the standard Euclidean inner product on $V = \bQ^n$, but we shall not need this fact.

\medskip

As in the statement of \cref{rep-bound-arithmetic}, let $u \in \gG(\bR)$ be such that $\gH_u = u\gH_{0,\bR}u^{-1}$ is defined over $\bQ$ and $V$ is irreducible as a representation of $\gH_u$.
Let $D = \End_{\gH_u}(V)$, which is a division algebra of type~I or~II depending on whether $d=1$ or~$2$.
By construction, $V$ is a left $D$-vector space of dimension~$m$.

Because $\iota_0(D_0) = \End_{\gH_0}(V)$ and $\gH_u = u\gH_{0,\bR}u^{-1}$, we have
\[ D = u\iota_0(D_{0,\bR})u^{-1} \cap \rM_n(\bQ). \]
Let $\alpha \colon D_{0,\bR} \to D_\bR$ be the isomorphism of $\bR$-algebras
\[ \alpha(d) = u \iota_0(d) u^{-1}. \]

Let $\dag = \alpha \circ t \circ \alpha^{-1}$, which is a positive involution of $D_\bR$.
A calculation using the fact that $u \in \gG(\bR) = \gGSp_n(\bR)$ shows that $\phi$ is $(D_\bR,\dag)$-compatible, that is, $\dag$ is the adjoint involution of $D_\bR$ with respect to~$\phi$.
This has two consequences:
\begin{enumerate}
\item $\dag$ is defined over $\bQ$, that is, $\dag$ is an involution of $D$ and not just of $D_\bR$.
\item There is a non-degenerate $(D,\dag)$-skew-Hermitian form $\psi \colon V \times V \to D$ such that $\phi = \Trd_{D/\bQ} \psi$, thanks to \cref{tr-skew-hermitian-form}.
\end{enumerate}

We are thus in a position to apply \cref{minkowski-hermitian-perfect} (with $R=R_u=\Stab_D(L)$).
Let $v_1, \dotsc, v_m$ be the resulting weakly symplectic or weakly unitary $D$-basis for $V$.

Define a $\bQ$-bilinear form $\sigma \colon V \times V \to \bQ$ by
\[ \sigma \Bigl( \sum_{i=1}^m x_i v_i, \sum_{i=1}^m y_i v_i \Bigr) = \Trd_{D/\bQ} \left(\sum_{i=1}^m x_i y_i^\dag\right) \]
for all $x_1, \dotsc, x_m, y_1, \dotsc, y_m \in D$.

\begin{lemma} \label{sigma-integral}
The bilinear form $\sigma$ is symmetric and positive definite.
It takes integer values on $R_uv_1 + \dotsb + R_uv_m$ and it satisfies
\[ \abs{\disc(R_uv_1 + \dotsb + R_uv_m, \sigma)} = d^{-d^2em} \abs{\disc(R_u)}^m. \]
\end{lemma}

\begin{proof}
The form $\sigma$ is symmetric because $\Trd_{D/\bQ}(xy^\dag) = \Trd_{D/\bQ}(yx^\dag)$ and it is positive definite because $\dag$ is a positive involution of~$D$.

For each $a \in R_u$ and $y \in L$, the map
\[ x \mapsto \phi(x, a^\dag y) = \phi(ax, y) \]
is $\bZ$-linear and maps $L$ into $\bZ$.
Since $\phi$ is a perfect pairing on $L$, this implies that $a^\dag y \in L$ for all $y \in L$.
Hence $a^\dag \in \Stab_D(L) = R_u$.

Thus if $x_1, \dotsc, x_m, y_1, \dotsc, y_m \in R_u$, then each $x_iy_i^\dag$ is in $R_u$ and so $\Trd_{D/\bQ}(x_iy_i^\dag) \in \bZ$.
Hence $\sigma(\sum x_i v_i, \sum y_i v_i) \in \bZ$.

For each $i$, the restriction of $\sigma$ to $R_uv_i$ is isometric to the inner product associated with $\abs{\cdot}_D$ on $R_u$.
Hence $\abs{\disc(R_uv_i, \sigma)} = d^{-d^2e} \abs{\disc(R_u)}$ and so
\[ \abs{\disc(R_uv_1 + \dotsb + R_uv_m, \sigma)} = d^{-d^2em} \abs{\disc(R_u)}^m.
\qedhere
\]
\end{proof}

\begin{lemma} \label{theta-def}
There exists an $\bR$-linear map $\theta \colon V_\bR \to V_\bR$ with the following properties:
\begin{enumerate}[(i)]
\item $\theta(\alpha(a)x) = \iota_0(a)\theta(x)$ for all $a \in D_{0,\bR}$, $x \in V_\bR$;
\item $\psi = \alpha \circ \theta^* \psi_0$;
\item $\sigma_0(\theta(x), \theta(x)) \leq \newC{theta-sigma-multiplier} \abs{\disc(R_u)}^{\newC{theta-sigma-exponent}} \sigma(x,x)$ for all $x \in V_\bR$, where the constants depend only on $d$, $e$ and $m$ (and not on $u$).
\end{enumerate}
\end{lemma}

\begin{proof}
Use \cref{semi-orthogonal-normalise} to choose $s_1, \dotsc, s_m \in D_\bR^\times$ such that $s_1^{-1} v_1, \dotsc, s_m^{-1} v_m$ is a symplectic or $\alpha$-unitary $D_\bR$-basis for $V_\bR$.

Define $\theta \colon V_\bR \to V_\bR$ by
\[ \theta(x_1v_1 + \dotsb + x_mv_m) = \iota_0(\alpha^{-1}(x_1 s_1))w_1 + \dotsb + \iota_0(\alpha^{-1}(x_m s_m))w_m \]
for all $x_1, \dotsc, x_m \in D_\bR$.

Claim~(i) holds because $\alpha \colon D_{0,\bR} \to D_\bR$ is a ring homomorphism.

Claim~(ii) holds because $s_1^{-1} v_1, \dotsc, s_m^{-1} v_m$ is a symplectic or $\alpha$-unitary $D_\bR$-basis for $V_\bR$ while $w_1, \dotsc, w_m$ is a symplectic or unitary $D_{0,\bR}$-basis for $D_{0,\bR}^m$.
Thus
\[ \alpha(\psi_0(\theta(s_i^{-1} v_i), \theta(s_j^{-1} v_j))) = \alpha(\psi_0(w_i, w_j)) = \psi(s_i^{-1} v_i, s_j^{-1} v_j) \text{ for all } i, j. \]

For claim~(iii): for every $x = x_1v_1 + \dotsb + x_mv_m \in V_\bR$, where $x_1, \dotsc, x_m \in D_\bR$,
\begin{align}
    \sigma_0(\theta(x), \theta(x))
  & = \Trd_{D_{0,\bR}/\bR} \bigl( \sum_{i=1}^m \alpha^{-1}(x_is_i) \alpha^{-1}(x_is_i)^t \bigr)
\notag
\\& = \sum_{i=1}^m \Trd_{D_{0,\bR}/\bR} \bigl( \alpha^{-1}(x_is_i s_i^\dag x_i^\dag) \bigr)
\notag
\\& = \sum_{i=1}^m \Trd_{D_\bR/\bR}(x_is_i s_i^\dag x_i^\dag)
    = \sum_{i=1}^m \abs{x_is_i}_D^2
\notag
\\& \leq \sum_{i=1}^m \abs{x_i}_D^2 \abs{s_i}_D^2
    \leq \bigl( \max_{i=1,\dotsc,m} \abs{s_i}_D^2 \bigr) \sigma(x, x).
\label{eqn:sigma0-sigma-bound}
\end{align}
Thanks to \cref{semi-orthogonal-normalise} and \cref{minkowski-hermitian-perfect}(iv), we have
\[ \max_{i=1, \dotsc, m} \abs{s_i}_D^2
   \leq (de)^{1/2} \max_{i,j = 1, \dotsc, m} \abs{\psi(v_i, v_j)}_D
   \leq \newC* \abs{\disc(R_u)}^{\newC*} \]
where the constants depend only on $d$, $e$ and $m$.
Combined with \eqref{eqn:sigma0-sigma-bound}, this proves claim~(iii).
\end{proof}

\begin{lemma} \label{h-in-h0}
Let $h = u^{-1} \theta^{-1} \colon V_\bR \to V_\bR$.
Then $uh = \theta^{-1} \in \gSp_n(\bR)$ and $h \in \gH_0(\bR)$.
\end{lemma}

\begin{proof}
Firstly, $\theta \in \gSp_n(\bR)$ by the following calculation, which relies on \cref{theta-def}(ii):
\begin{align*}
    \theta^* \phi
    = \theta^*({\Trd_{D_{0,\bR}/\bR}} \circ \psi_0)
  & = \theta^*({\Trd_{D_\bR/\bR}} \circ \alpha \circ \psi_0)
\\& = {\Trd_{D_\bR/\bR}} \circ \alpha \circ (\theta^* \psi_0)
    = {\Trd_{D_\bR/\bR}} \circ \psi
    = \phi.
\end{align*}

Since $\gSp_n(\bR) \subset \gGSp_n(\bR) = \gG(\bR)$ and $u \in \gG(\bR)$, it follows that $h \in \gG(\bR)$.

By definition, $\gH_0 = Z_\gG(\iota_0(D_0))$ and so it remains to prove that $h$ commutes with the action of $D_0$ on $V$.
For $a \in D_{0,\bR}$ and $x \in V_\bR$, we have
\begin{align*}
    h(\iota_0(a)x)
    = u^{-1} \theta^{-1} (\iota_0(a)x)
  & = u^{-1} \alpha(a) \theta^{-1}(x)
\\& = \iota_0(a) u^{-1} \theta^{-1}(x)
    = \iota_0(a)h(x)
\end{align*}
where we use \cref{theta-def}(i) and the fact that $\alpha(a) = u\iota_0(a)u^{-1}$ (from the definition of $\alpha$).
Thus $h$ commutes with all $a \in \iota_0(D_0)$.
\end{proof}

\begin{lemma} \label{z-basis-bound}
There exists a $\bZ$-basis $\{ e_1', \dotsc, e_n' \}$ for $L$ such that the coordinates of the vectors $\theta(e_1'), \dotsc, \theta(e_n')$ in $V_\bR = \bR^n$ are polynomially bounded in terms of $\abs{\disc(R_u)}$.
\end{lemma}

\begin{proof}
Let $\lambda_1, \dotsc, \lambda_n$ denote the successive minima of $R_u v_1 + \dotsb + R_u v_m$ with respect to~$\sigma$.
By \cref{minkowski-2nd} and \cref{sigma-integral}, we have
\[ \lambda_1 \lambda_2 \dotsm \lambda_n
\leq \gamma_{d^2em}^{n/2} \covol(R_u v_1 + \dotsb + R_u v_m)
 \leq \newC{sigma-lambda-multiplier} \abs{\disc(R_u)}^{-m} \]
where $\refC{sigma-lambda-multiplier}$ depends only on $d$, $e$ and~$m$.

For each $i$, $\lambda_i^2 = \sigma(v,v)$ for some $v \in R_u v_1 + \dotsb + R_u v_m$ and so $\lambda_i \geq 1$ by \cref{sigma-integral}.
We deduce that, for each $i$,
\[ \lambda_i \leq \refC{sigma-lambda-multiplier} \abs{\disc(R_u)}^{-m}. \]

Let $\lambda_1', \dotsc, \lambda_n'$ denote the successive minima of $L$ with respect to $\sigma$.
Since $R_u v_1 + \dotsb + R_u v_m \subset L$, $\lambda_i' \leq \lambda_i$ for each~$i$.
By \cite[Theorem~4]{Wey40}, there exists a $\bZ$-basis $e_1', \dotsc, e_n'$ for $L$ such that
\[ \sqrt{\sigma(e_i', e_i')} \leq \newC{weyl-multiplier} \lambda_i' \]
where $\refC{weyl-multiplier}$ depends only on~$n$.
Combining the above inequalities, we obtain
\begin{equation*}
\sigma(e_i', e_i') \leq \newC* \abs{\disc(R_u)}^{-2m}.
\end{equation*}

Combining this with \cref{theta-def}(iii), we obtain that
\[ \sigma_0(\theta(e_i'), \theta(e_i'))
   \leq \refC{theta-sigma-multiplier} \abs{\disc(R_u)}^{\refC{theta-sigma-exponent}} \sigma(e_i', e_i')
   \leq \refC{sigma0-multiplier} \abs{\disc(R_u)}^{\refC{sigma0-exponent}} \]
for some constants $\newC{sigma0-multiplier}$, $\newC{sigma0-exponent}$ independent of $u \in \gG(\bR)$.
Since $\sigma_0$ is a fixed positive definite quadratic form on $V_\bR$, this implies that the coordinates of the vectors $\theta(e_1'), \dotsc, \theta(e_n')$ are likewise bounded by a polynomial in $\abs{\disc(R_u)}$.
\end{proof}

Let $\gamma'$ be the matrix in $\gGL_n(\bZ)$ which maps the vectors $e'_1, \dotsc, e'_n$ to the standard basis of $L = \bZ^n$.

\begin{lemma} \label{entries'-bound}
The entries of the matrices $\gamma' uh, (\gamma' uh)^{-1} \in \gGL_n(\bR)$ are polynomially bounded in terms of $\disc(R_u)$.
\end{lemma}

\begin{proof}
Let $A = \gamma' uh = \gamma' \theta^{-1} \in \gGL_n(\bR)$.
Observe that $A$ maps the vectors $\theta(e_1'), \dotsc, \theta(e_n')$ to the standard basis.
In other words, the entries of $A^{-1}$ are the coordinates of $\theta(e_1'), \dotsc, \theta(e_n')$ and so are bounded by \cref{z-basis-bound}.

By \cref{h-in-h0}, $\det(uh)=1$, while $\abs{\det(\gamma')} = 1$ since $\gamma' \in \gGL_n(\bZ)$.
Hence $\abs{\det(A)}=1$.
By Cramer's rule, each entry of $A$ is a fixed polynomial in the entries of $A^{-1}$, multiplied by $\det(A)$.
We conclude that the entries of $A$ are polynomially bounded in terms of $\disc(R_u)$.
\end{proof}

We now show that we can modify $\gamma' \in \gGL_n(\bZ)$ to obtain $\gamma \in \gSp_n(\bZ)$, with a similar bound on $\gamma uh$.  This establishes \cref{rep-bound-arithmetic}(b), and we will subsequently use it to prove \cref{rep-bound-arithmetic}(a).

\begin{lemma} \label{entries-bound}
There exists $\gamma \in \Gamma = \gSp_n(\bZ)$ such that the entries of $\gamma uh$ and $(\gamma uh)^{-1}$ are polynomially bounded in terms of $\abs{\disc(R_u)}$.
\end{lemma}

\begin{proof}
Let $e_1, \dotsc, e_n$ denote the standard basis of $L = \bZ^n$.

According to \cref{h-in-h0}, $uh \in \gSp_n(\bR)$.
Consequently
\[ \phi(\gamma'^{-1} e_i, \gamma'^{-1} e_j) = \phi((uh)^{-1}\gamma'^{-1} e_i, \, (uh)^{-1} \gamma'^{-1} e_j) \text{ for all } i, j \in \{ 1, \dotsc, n \}. \]
By \cref{entries'-bound}, the entries of $(uh)^{-1} \gamma'^{-1}$ are polynomially bounded in terms of $\abs{\disc(R_u)}$, and hence the same is true of the values $\phi(\gamma'^{-1} e_i, \gamma'^{-1} e_j)$.

Hence, by \cite[Lemma~4.3]{Orr15}, there exists a symplectic $\bZ$-basis $\{ f_1, \dotsc, f_n \}$ for $(L, \phi)$ whose coordinates with respect to $\{ \gamma'^{-1}e_1, \dotsc, \gamma'^{-1}e_n \}$ are polynomially bounded in terms of $\abs{\disc(R_u)}$.
Applying $\gamma'$, we deduce that the coordinates of $\gamma' f_1, \dotsc, \gamma' f_n$ with respect to the standard basis are polynomially bounded.

Let $\gamma \in \gGL_n(\bZ)$ be the matrix such that $e_i = \gamma f_i$ for each~$i = 1, \dotsc, n$.
Since $\{ f_1, \dotsc, f_n \}$ is a symplectic basis, we have $\gamma \in \Gamma$.
We have just shown that the coordinates of $\gamma' f_i = \gamma' \gamma^{-1} e_i$ are polynomially bounded for each~$i$.
In other words, the entries of the matrix $\gamma' \gamma^{-1}$ are polynomially bounded in terms of $\abs{\disc(R_u)}$.

Multiplying $(\gamma'uh)^{-1}$ by $\gamma' \gamma^{-1}$ and applying \cref{entries'-bound}, we deduce that the entries of $(\gamma uh)^{-1}$ are polynomially bounded in terms of $\abs{\disc(R_u)}$.
Thanks to \cref{h-in-h0}, $\abs{\det(\gamma uh)} = 1$, 
so it follows that the entries of $\gamma uh$ are also polynomially bounded in terms of $\abs{\disc(R_u)}$.
\end{proof}

Let $S_u = \End_{R_u}(L) = uE_{0,\bR}u^{-1} \cap \rM_n(\bZ)$, where $E_0$ is defined in \eqref{eqn:E_0}.
By \cref{reps-closed-orbits}(v), there exists $d_u \in \bR_{>0}$ such that
\[ d_u \rho_R(u) \rho_L(u) w_0 \in \Lambda  \quad  \text{and}  \quad  d_u \leq \refC{du-multiplier} \abs{\disc(S_u)}^{1/2}. \]

In order to prove \cref{rep-bound-arithmetic}(a), we shall use the vector
\[ w_u = d_u\rho_R(\gamma u)w_0 \in W_\bR. \]
Observe first that $d_u\rho_R(\gamma u) \in \Aut_{\rho_L(\gG)}(\Lambda_\bR)$ thanks to \cref{reps-closed-orbits}(iii), and that $\rho_L(u)w_u = \rho_R(\gamma) d_u \rho_R(u) \rho_L(u) w_0$ is in $\Lambda$ thanks to \cref{reps-closed-orbits}(iv) .
Hence $w_u$ satisfies the qualitative conditions of \cref{rep-bound-arithmetic}(a), and it only remains to prove the bound for $\abs{w_u}$.

\begin{lemma} \label{length-wu}
$\abs{w_u} \leq \newC* \abs{\disc(R_u)}^{\newC*}$.
\end{lemma}

\begin{proof}
According to \cref{reps-closed-orbits}(i), $\gH_{d,e,m} = \Stab_{\rho_R(\gG)}(w_0)$.
Therefore
\begin{align*}
w_u = d_u\rho_R(\gamma u)w_0 = d_u\rho_R(\gamma uh)w_0.
\end{align*}

The homomorphism $\rho_R \colon \gG \to \gGL(W)$ is given by fixed polynomials in the entries and inverse determinant.
Since the entries of $\gamma uh$ and $\det(\gamma uh)^{-1}$ are bounded by \cref{entries-bound}, we deduce that the entries of $\rho_R(\gamma uh)$ are likewise polynomially bounded in terms of $\disc(R_u)$.

Meanwhile, by definition, $d_u$ is polynomially bounded in terms of $\disc(S_u)$.
By \cref{disc-R-S}, $\disc(S_u)$ is polynomially bounded in terms of $\disc(R_u)$.
We conclude that $\abs{w_u}$ is polynomially bounded in terms of $\abs{\disc(R_u)}$, as required.
\end{proof}

\section{Cases of Zilber--Pink}\label{cases-of-ZP}

In this section, we prove Theorems \ref{main-theorem-zp} and \ref{unconditional}. The proofs follow closely \cite[sec.~6]{QRTUI}. We refer to notation and terminology from \cite[sec. 2.2 and 2.4]{Orr18}.

\subsection{Proof of Theorem \ref{main-theorem-zp}}

In fact, instead of proving \cref{main-theorem-zp}, we will prove the following, more general theorem. (Recall that, by \cref{codim-pel}, for $g\geq 3$, all proper special subvarieties of PEL type of $\cA_g$ have codimension at least~$2$.)

\begin{theorem}\label{ZP-end}
Let $g\geq 3$ and let $C$ be an irreducible algebraic curve in $\cA_g$. Let $S$ denote the smallest special subvariety of $\cA_g$ containing $C$.
Let $\Omega$ denote the set of special subvarieties of $\cA_g$ of simple PEL type I or~II of dimension at most $\dim(S)-2$. Let $\Sigma$ denote the set of points in $\cA_g(\CC)$ which are endormorphism generic in some $Z\in\Omega$.

If $C$ satisfies Conjecture \ref{LGO-general}, then $C\cap\Sigma$ is finite.
\end{theorem}

Conjecture \ref{LGO-general} is the natural generalisation of Conjecture \ref{galois-orbits}.

\begin{conjecture}\label{LGO-general}
Let $C$ and $\Sigma$ be as in Theorem \ref{ZP-end} and let $L$ be a finitely generated subfield of $\CC$ over which $C$ is defined.
Then there exist positive constants $\newC{ZP-end-mult}$ and $\newC{ZP-end-exp}$ such that
\begin{align*}
    \#\Aut(\CC/L)\cdot s\geq\refC{ZP-end-mult}|\disc(\End(A_s))|^{\refC{ZP-end-exp}}
\end{align*}
for all $s\in C\cap\Sigma$.
\end{conjecture}

Let $L=\ZZ^{2g}$ and let $\phi:L\times L\to\ZZ$ be the standard symplectic form as in section~\ref{subsec:shimura-data}.
Let $\gG=\gGSp(L_\bQ, \psi)=\gGSp_{2g}$ and let $\Gamma=\gSp_{2g}(\ZZ)$.
Define $h_0 \colon \mathbb{S}\to\gG_\RR$ as in \eqref{eqn:h0} and let $X^+$ denote the $\gG(\bR)$-conjugacy class of $h_0$ in $\Hom(\bS, \gG_\bR)$.
Then $(\gG, X^+)$ is a Shimura datum component and so $\Stab_{\gG(\bR)}(h_0) = \bR^\times K^+_\infty$ where $K^+_\infty$ is a maximal compact subgroup of $\gG(\bR)^+$ \cite[chapter~6]{Mil05}.

Let $(\gP, \gS, K_\infty)$ be a Siegel triple for $\gG$, as defined in \cite[sec.~2B]{Orr18}, where $K_\infty$ is a maximal compact subgroup of $\gG(\RR)$ such that $K^+_\infty=\gG(\RR)^+\cap K_\infty$.
By the results of Borel quoted in \cite[sec.~2D]{Orr18}, there exists a Siegel set $\fS \subset \gG(\RR)$ with respect to $(\gP, \gS, K_\infty)$ and a finite set $C_\gG\subset\gG(\QQ)$ such that $\cF_{\gG}=C_\gG\fS$ is a fundamental set for $\Gamma$ in $\gG(\RR)$.

Let $\cF = (\cF_\gG \cap \gG(\bR)^+) h_0$.  Since $\Gamma \subset \gG(\bR)^+$, $\cF$ is a fundamental set in $X^+$ for~$\Gamma$.  If we denote by $\pi:X^+\to\cA_g$ the uniformising map, then $\pi|_{\cF}$ is definable in the o-minimal structure $\RR_{\rm an,exp}$ (see \cite{PS10} for the original result and \cite{kuy:ax-lindemann} for a formulation in notations more similar to ours).

As explained in section \ref{subsec:proof-strategy-high-level}, $\Sigma$ is the union of sets $\Sigma_{d,e,m}$, where $d$, $e$, $m$ are positive integers satisfying $d^2em=2g$, $d=1$ or $2$ and $dm$ is even.
Since there are only finitely many choices for such $d$, $e$, $m$ (given~$g$), in order to prove \cref{ZP-end}, it suffices to prove that $C \cap \Sigma_{d,e,m}$ is finite for each $d$, $e$, $m$.

From now on, we fix such integers $d$, $e$ and~$m$.  Let $\gH_0 \subset \gG$ be the group defined in \eqref{eqn:H0} associated with these parameters.
Let $X_0^+ = \gH_0(\bR)^+h_0$, so that $(\gH_0, X_0^+)$ is the unique Shimura subdatum component of $(\gG, X^+)$ given by \cref{unique-datum}.

By \cref{reps-closed-orbits,rep-bound-arithmetic}, there exists a finitely generated, free $\ZZ$--module $\Lambda$, a representation $\rho_L:\gG\to\gGL(\Lambda_\QQ)$ such that $\Lambda$ is stabilised by $\rho_L(\Gamma)$, a vector $w_0\in\Lambda$ and positive constants $\refC{rep-multiplier}$ and $\refC{rep-exponent}$ such that:
\begin{enumerate}[(i)]
\item $\Stab_{\gG,\rho_L}(w_0) = \gH_0$;

\item the orbit $\rho_L(\gG(\bR))w_0$ is closed in $\Lambda_\bR$;

\item for each $u \in \gG(\bR)$, if the group $\gH_u = u \gH_{0,\bR} u^{-1}$ is defined over~$\bQ$ and $L_\bQ$ is irreducible as a representation of $\gH_u$ over $\bQ$, then there exists $w_u \in \Aut_{\rho_L(\gG)}(\Lambda_\bR) w_0$ such that $\rho_L(u) w_u \in \Lambda$ and
\[ \abs{w_u} \leq \refC{rep-multiplier} \abs{\disc(R_u)}^{\refC{rep-exponent}}, \]
where $R_u$ denotes the ring $\End_{\gH_u}(L) \subset \rM_{2g}(\bZ)$.
\end{enumerate}

By \cite[Theorem 1.2]{QRTUI}, there exist positive constants $\newC{QRTUI-multiplier}$ and $\newC{QRTUI-exponent}$ with the following property: for every $u\in\gG(\RR)$ and $w_u\in \Aut_{\rho_L(\gG)}(\Lambda_\bR) w_0$ such that $\gH_u = u \gH_{0,\bR} u^{-1}$ is defined over~$\bQ$ and $\rho_L(u) w_u \in \Lambda$, there exists a fundamental set for $\Gamma\cap\gH_u(\RR)$ in $\gH_u(\RR)$ of the form 
\[ B_u\cF_\gG u^{-1}\cap\gH_u(\RR),\]
where $B_u\subset\Gamma$ is a finite set such that
\[\abs{\rho_L(b^{-1}u)w_u} \leq \refC{QRTUI-multiplier} \abs{w_u}^{\refC{QRTUI-exponent}}\]
for every $b\in B_u$.

\medskip

For any $w\in\Lambda_\RR$, we write $\gG(w)$ for the real algebraic group $\Stab_{\gG_\RR,\rho_L}(w)$. Fixing a basis for $\Lambda$, we may refer to the height $\rH(w)$ of any $w\in\Lambda$ (namely, the maximum of the absolute values of its coordinates with respect to this basis.)

\begin{lemma} \label{z-w}
Let $P\in\Sigma_{d,e,m}$. There exists $z\in \pi^{-1}(P)\cap\cF$ and
\[w\in \Aut_{\rho_L(\gG)}(\Lambda_\bR)\rho_L(\gG(\RR)^+)w_0\cap\Lambda\] 
such that $z(\mathbb{S})\subset\gG(w)$ and
\[\rH(w)\leq \refC{QRTUI-multiplier}\refC{rep-multiplier}^{\refC{QRTUI-exponent}} \abs{\disc(R)}^{\refC{rep-exponent}\refC{QRTUI-exponent}},\]
where $R=\End(A_P)\cong\End_{\gG(w)}(L)\subset \rM_{2g}(\ZZ)$.
\end{lemma}

\begin{proof}
Let $z'\in \pi^{-1}(P)\cap\cF$.
Since $P\in\Sigma_{d,e,m}$, it is an endomorphism generic point of a special subvariety $S \subset \Ag$ of simple PEL type I or~II with parameters $d,e,m$.
Therefore, there is a Shimura subdatum component $(\gH, Y^+) \subset (\gG, X^+)$ of simple PEL type I or~II such that $\pi(Y^+) = S$ and $z' \in Y^+$. (In particular, $z'(\mathbb{S})\subset\gH_{\RR}$.) By \cref{conj-class-mt}, $\gH_{\bR} = u\gH_{0,\RR}u^{-1}$ for some $u\in\gG(\RR)^+$, and so we write $\gH_u=\gH$.
By \cref{conj-class-datum},
\[Y^+=uX^+_0=u\gH_0(\RR)^+h_0=\gH_u(\RR)^+uh_0. \]

Let $R_u = \End_{\gH_u}(L)$.
Since $\gH_u$ is the general Lefschetz group of $S$, $R_u$ is the generic endomorphism ring of~$S$ and, hence, isomorphic to $\End(A_P)$.  Since $S$ is a special subvariety of simple PEL type,  $R_{u,\bQ}$ is a division algebra.
Hence $L_\bQ$ is irreducible as a representation of~$\gH_u$.

By \cref{rep-bound-arithmetic}(a), there exists $w_u \in \Aut_{\rho_L(\gG)}(\Lambda_\bR) w_0$ such that
\[ \rho_L(u) w_u \in \Lambda \quad \text{ and } \quad \abs{w_u} \leq \refC{rep-multiplier} \abs{\disc(R_u)}^{\refC{rep-exponent}}. \]
Hence, by \cite[Theorem 1.2]{QRTUI}, there exists a fundamental set for $\Gamma\cap\gH_u(\RR)$ in $\gH_u(\RR)$ of the form 
\[ \cF_u=B_u\cF_\gG u^{-1}\cap\gH_u(\RR),\]
where $B_u\subset\Gamma$ is a finite set such that
\[\abs{\rho_L(b^{-1}u)w_u} \leq \refC{QRTUI-multiplier} \abs{w_u}^{\refC{QRTUI-exponent}}\]
for every $b\in B_u$.
Therefore, we can write $z' \in \gH_u(\bR)^+ uh_0$ as
\[ z' = \gamma b f u^{-1}\cdot uh_0 \]
for some $\gamma \in \Gamma \cap \gH_u(\bR)$, $b \in B_u$, and $f \in \cF_\gG$.

Let
\[ z = b^{-1}\gamma^{-1}z' = fh_0 \in \cF_\gG h_0\cap X^+=\cF, \]
where the last equality uses the fact that $\Stab_{\gG(\bR)}(h_0)\subset\gG(\RR)^+$.
Since $b, \gamma \in \Gamma$, we obtain $z \in \pi^{-1}(P) \cap \cF$.

Let $w=\rho_L(b^{-1}u)w_u$.
As in \cite[Proposition 6.3]{QRTUI}, we can show that $z(\bS) \subset \gG(w)$ and that $\gG(w)$ is a $\Gamma$-conjugate of $\gH_u$, so $R_u \cong \End_{\gG(w)}(L)$.
Consequently, $z$ and $w$ satisfy the requirements of the lemma.
\end{proof}

We can now deduce \cref{prop:finitely-many-intro}.

\begin{corollary}\label{prop:finitely-many}
Define $\Sigma \subset \Ag$ as in \cref{main-theorem-zp}.
For each $b\in\RR$, the points $s\in\Sigma$ such that $\abs{\disc(\End(A_s))} \leq b$ belong to only finitely many proper special subvarieties of simple PEL type I or II.
\end{corollary}

\begin{proof}
The proof is essentially the same as \cite[Corollary 6.4]{QRTUI}.
\end{proof}

The proof of Theorem \ref{ZP-end} now proceeds as in 
\cite[sec.~6.5]{QRTUI} with some modifications, which we outline below (following the notation from \cite[sec.~6.5]{QRTUI} \textit{mutatis mutandis}).
\begin{enumerate}[(1)]
\item The argument is carried out inside $X^+ \cong \Hg$ instead of $\cH_2$.

\smallskip

\item If $P \in \Sigma_{d,e,m}$, then $P$ is endomorphism generic in some special subvariety $Z \in \Omega$ (where $\Omega$ is defined in \cref{ZP-end}).
Then $Z$ is an irreducible component of $\cM_R$, where $R = \End(A_P)$ (see definitions in section~\ref{subsec:intro-zp-context}).
Since $\cM_R$ is $\Aut(\CC)$-invariant and its (analytic) irreducible components are algebraic subvarieties of $\Ag$, for each $\sigma \in \Aut(\CC)$, $\sigma(Z)$ is also an irreducible component of $\cM_R$.
Thus, $\sigma(Z)$ is also a special subvariety of simple PEL type I or~II with the same parameters $d,e,m$. 
Furthermore, $\dim(\sigma(Z)) = \dim(Z)$, so $\sigma(Z) \in \Omega$.
Since $\End(A_{\sigma(P)}) \cong \End(A_P)$, $\sigma(P)$ is endomorphism generic in $\sigma(Z)$, so $\sigma(P) \in \Sigma_{d,e,m}$.

\smallskip

\item In the definition of the definable set~$D$, we replace $\gG(\RR)$ with $\gG(\RR)^+$. That is, $w\in\Aut_{\rho_L(\gG)}(\Lambda_\bR)\rho_L(\gG(\RR)^+)w_0$, as in \cref{z-w}. Then 
\[g_t \in \Aut_{\rho_L(\gG)}(\Lambda_\bR) \rho_L(\gG(\bR)^+)\] for all $t$.  So $g_t^{-1}z_t$ is in the same connected component of $X$ as $z_t \in \cC \subset X^+$. We conclude that $g_t^{-1}z_t$ lies on the unique pre-special subvariety of $X^+\cong\Hg$ associated with $\gH_0$, namely, $X^+_0$ (see \cref{unique-datum}).

\smallskip

\item By the inverse Ax--Lindemann conjecture, the smallest algebraic subset of $X^+$ containing $\tilde{C}$ is an irreducible component of $\pi^{-1}(S)$, which we call~$\tilde{S}$.

\smallskip

\item As in the penultimate paragraph of \cite[sec.~6.5]{QRTUI}, we choose a complex algebraic subset $\tilde{B} \subset \Aut_{\rho_L(\gG)}(\Lambda_\bC) \rho_L(\gG(\bC))$ of dimension at most~$1$ whose image under the map $g \mapsto g \cdot w_0$ is~$B$.
Here, the map
\[ \cdot : \Aut_{\rho_L(\gG)}(\Lambda_\bC) \rho_L(\gG(\bC)) \times (X^+)^\vee\to (X^+)^\vee\cong\Hg^\vee \] (which is used in \cite[sec.~6.5]{QRTUI}, but not explicitly defined there) is given by
\begin{align*}
    (a\rho_L(g),x)\mapsto g\cdot x,
\end{align*}
for each $a\in\Aut_{\rho_L(\gG)}(\Lambda_\bC)$ and $\rho_L(g)\in\rho_L(\gG(\CC))$. This is well-defined since
\begin{align*}
    \Aut_{\rho_L(\gG)}(\Lambda_\bC)\cap\rho_L(\gG(\CC))\subset\rho_L(Z(\gG)(\CC))\text{ and }\ker(\rho_L)\subset Z(\gG),
\end{align*}
and $Z(\gG)$, the centre of $\gG$, acts trivially on $(X^+)^\vee$.

\smallskip

\item In the final step, we conclude that $\tilde{B}\cdot(X^+_0)^\vee$ has uncountable intersection with $\tilde{C}$ and, hence, contains it. Therefore, $\tilde{S}$ is contained in $\tilde{B}\cdot(X^+_0)^\vee$, but
\[ \dim(\tilde{B} \cdot (X^+_0)^\vee) \leq 1 + \dim(X_0^+) \leq \dim(S) - 1, \]
delivering the contradiction.
\end{enumerate}

\subsection{Proof of Theorem~\ref{unconditional}}

If $C$ is an algebraic curve over a number field, and $\mathfrak{A}\to C$ is an abelian scheme of even relative dimension $g$, we say that $s\in C(\Qbar)$ is an \defterm{exceptional quaternionic point} if $\End(\fA_s)\otimes_\ZZ\QQ$ is a non-split totally indefinite quaternion algebra over a totally real field of degree $e$ such that $4e$ does not divide~$g$.
Note that these are precisely the points for which:
\begin{enumerate}[(i)]
\item $\fA_s$ is simple and $D := \End(\fA_s) \otimes \bQ$ has type I or~II; and
\item $\fA_s$ is exceptional in the sense of \cite[Definition~8.1]{ExCM}, that is, $D$ is not isomorphic to a subring of $\rM_g(\QQ)$.
\end{enumerate}
Indeed, if $\fA_s$ is simple, then $D$ is a division algebra and hence embeds into $\rM_g(\QQ)$ if and only if $\dim_\QQ(D)$ divides~$g$.
If $D$ has type~I, then $\dim_\QQ(D)$ always divides~$g$, while if $D$ has type~II, then $\dim_\QQ(D) = 4e$.

In order to prove \cref{unconditional}, it suffices to prove the following theorem, by the same argument as in \cite[sec.~6.7]{QRTUI}.
This theorem is a direct generalisation of \cite[Theorem~6.5]{QRTUI}.
Note that the image of $C \to \Ag$ is Hodge generic if and only if the generic Mumford--Tate group of the abelian scheme $\fA \to C$ is $\gGSp_{2g,\bQ}$.

\begin{theorem}\label{EQ-scheme}
Let $C$ be a irreducible algebraic curve and let $\fA\to C$ be a principally polarised non-isotrivial abelian scheme of even relative dimension $g$ such that the image of the morphism $C\to\cA_g$ induced by $\fA$ is Hodge generic.

Suppose that $C$ and $\mathfrak{A}$ are defined over a number field $L$ and that there exists a smooth curve \( C' \), a semiabelian scheme \( \fA' \to C' \) and an open immersion \( \iota \colon C \to C' \), all defined over \( \Qbar \), such that \( \fA \cong \iota^* \fA' \) and, for some point \( s_0 \in C'(\ov\bQ) \setminus C(\ov\bQ) \), the fibre \( \fA'_{s_0} \) is a torus.

Then there exist positive constants $\newC{GOEQ-scheme-mult}$ and $\newC{GOEQ-scheme-exp}$ such that, for any exceptional quaternionic point $s\in C$,
\[\#\Aut(\CC/L)\cdot s\geq\refC{GOEQ-scheme-mult}\abs{\disc(\End(\fA_s))}^{\refC{GOEQ-scheme-exp}}.\]
\end{theorem}

\begin{proof}
After replacing $L$ by a finite extension, we may assume that  \( C' \),  \( \fA' \to C' \) and \( \iota \colon C \to C' \) are all defined over $L$.
After replacing $C'$ by its normalisation and $\fA'$ by its pullback to this normalisation, we may assume that $C'$ is smooth.  (Note that this step, which is required in order to apply \cite[Theorem~8.2]{ExCM}, was erroneously omitted in the proofs of \cite[Prop.~9.2]{ExCM} and \cite[Theorem~6.5]{QRTUI}.)
Observe that \( \fA \to C \) satisfies the conditions of \cite[Theorem 8.2]{ExCM}.

Let $s\in C$ be an exceptional quaternionic point. The image of $s$ under the map $C\rightarrow\mathcal{A}_g$ induced by $\fA\to C$ is in the intersection between the image of $C$ and a proper special subvariety of PEL type.
Since $C$ is a curve defined over $\Qbar$ and special subvarieties of $\Ag$ are defined over $\Qbar$, it follows that $s\in C(\Qbar)$.

The remainder of the proof proceeds as in the proof of \cite[Theorem 6.5]{QRTUI}.
The key ingredients are:
\begin{enumerate}
\item \cite[Theorem~8.2]{ExCM}, a height bound for exceptional points of~$C$ (including exceptional quaternionic points) which generalises \cite[Ch.~X, Theorem~1.3]{And89};
\item endomorphism estimates of Masser and W\"ustholz \cite{MW94} (a version using present notations is \cite[Theorem 6.6]{QRTUI}).
\end{enumerate}
\end{proof}

\bibliography{PEL13}
\bibliographystyle{amsalpha}

\end{document}